\documentclass[12pt,a4paper]{amsart}
\calclayout
\numberwithin{equation}{section}
\allowdisplaybreaks
\theoremstyle{plain}

\newtheorem{theorem}{Theorem}[section]
\newtheorem{lemma}[theorem]{Lemma}

\newtheorem{corollary}[theorem]{Corollary}
\newtheorem{remark}{Remark}

\numberwithin{equation}{section}
\usepackage{enumerate}

\usepackage{amssymb}
\usepackage{mathtools}
\mathtoolsset{showonlyrefs}
\usepackage{mathrsfs}
\usepackage{comment}
\usepackage{hyperref}
\usepackage{xcolor}
\usepackage{footnote}

\def\P{\mathbb{P} }
\def\R{\mathbb{R} }

\def\E{\mathbb{E} }

\begin{document}
	\title[	Subcritical branching killed L\'{e}vy process]{
		Asymptotic behaviors of subcritical branching killed L\'{e}vy processes
		}
	\thanks{The research of Yan-Xia Ren is supported by NSFC
(Grant No 12231002) and
the Fundamental Research Funds for the Central Universities, Peking University LMEQF}
\thanks{ Research supported in part by a grant from the Simons	Foundation	(\#960480, Renming Song).}
\thanks{The research of Yaping Zhu is supported by the Fundamental Research Funds for the Central Universities.}
\thanks{*Yaping Zhu is the corresponding author}
\author[Y.-X. Ren]{Yan-Xia Ren}
\author[R. Song]{Renming Song}
\author[Y. Zhu]{Yaping Zhu$^\ast$}
	\address{Yan-Xia Ren\\ LMAM School of Mathematical Sciences \& Center for
		Statistical Science  \\ Peking University\\ Beijing 100871\\ P. R. China}
	\email{yxren@math.pku.edu.cn}
	
	\address{Renming Song\\ Department of Mathematics \\ University of Illinois\\ Urbana, IL 61801 \\ U.S.A.}
	\email{rsong@illinois.edu}
	
	\address{Yaping Zhu\\ Department of Mathematics \\ Shanghai University of Finance and Economics\\ Shanghai\\ 200433\\ P. R. China}
	\email{zhuyaping@mail.sufe.edu.cn}
	
	\begin{abstract}
		In this paper, we investigate the asymptotic behaviors of the survival probability and maximal displacement of a subcritical branching killed L\'{e}vy process $X$ in $\mathbb{R}$.
		Let $\zeta$ denote the extinction
		time, $M_t$ be  the  maximal position of all the particles alive
		at time $t$, and $M:=\sup_{t\ge 0}M_t$ be the all-time maximum.
		Under the assumption that the offspring distribution satisfies the $L\log L$ condition and some conditions on the spatial motion, we find the
		decay rate of the survival probability $\mathbb{P}_x(\zeta>t)$ and the tail behavior of $M_t$ as $t\to\infty$. As a consequence, we
		establish a Yaglom-type theorem. We also find the asymptotic behavior of $\mathbb{P}_x(M>y)$ as $y\to\infty$.
	\end{abstract}
	\subjclass[2020]{Primary: 60J80; Secondary: 60J65}
	\keywords{Branching killed L\'{e}vy process, subcritical branching process, survival probability, maximal displacement, Yaglom limit, spectrally negative L\'{e}vy process, Feynman-Kac representation
	}
	\maketitle

	\section{Introduction}\label{Sec1}
	
	\subsection{Background and motivation}

	A branching L\'{e}vy process on $\mathbb{R}$ is defined as follows:
	at time $0$, there is a particle at $x\in \mathbb{R}$ and it moves according to a L\'{e}vy process $(\xi_t,\mathbf{P}_x)$ on $\mathbb{R}$. After an exponential time with parameter $\beta>0$, independent of the spatial motion, this particle dies and is replaced by $k$ offspring with probability $p_k$, $k\ge 0$.
	The offspring move independently according to the same L\'{e}vy process
		starting from the death position of their parent.
	This procedure goes on. Let $N_t$ be the set of particles alive at time $t$ and for each $u\in N_t$, we denote by $X_u(t)$ the position of $u$ at time $t$. Also, for any $u\in N_t$ and $s\le t$, we use $X_u(s)$ to denote the position of $u$ or its ancestor at time $s$. Then the point process $Z=(Z_t)_{t\ge 0}$ defined by
	\begin{align}
		Z_t:=\sum_{u\in N_t}\delta_{X_u(t)}
	\end{align}
	is called a branching L\'{e}vy process.
	We shall denote by $\P_x$ the law of this process when the initial particle
	starts from $x$ and use $\E_x$ to denote the corresponding expectation.
	Let
	\begin{align}
		\widetilde{\zeta}:=\inf\{t>0:Z_t(\mathbb{R})=0\}
	\end{align}
	be the extinction time of $Z$. Note that
	$\widetilde{\zeta}$ is equal in law to that of the extinction time of
	a continuous-time Galton-Watson
	process with the same branching mechanism as the branching L\'{e}vy process.
	Let $m:=\sum_{k=0}^{\infty}kp_k$ be the mean number of offspring. It is well-known that $Z$ will become extinct in finite time with probability 1 if and only if $m<1$ (subcritical) or $m=1$ and $p_1\neq 1$ (critical). Moreover,
	the process $Z$ survives with positive probability when $m>1$ (supercritical).

	The focus of this paper is on the  asymptotic behaviors of a branching killed L\'{e}vy process,
	in which particles are killed upon entering the negative
	half-line. 	The point process
	 $Z^0=(Z^0_t)_{t\ge 0}$ defined by
	\begin{align}
		Z^0_t
		:=\sum_{u\in N_t}1_{\{\inf_{s\le t}X_u(s)>0\}}\delta_{X_u(t)}
	\end{align}
	is called a branching killed  L\'{e}vy process.
	For any $t\ge 0$, let
	\begin{align}
		M_t:=\sup_{u\in N_t, \inf_{s\le t }X_u(s)>0}X_u(t)
	\end{align}
	be the maximal position of all the particles alive
	at time $t$ in the process $Z^0$.
	We define the all-time maximum position and the extinction time of $Z^0$ by
	\begin{align}
		M:=\sup_{t\ge 0}M_t, \qquad
		\zeta:=\inf\{t>0:Z^0_t((0,\infty))=0\}.
	\end{align}

	In the critical case, i.e., when $m=1$ and $p_1\neq 1$,
	the asymptotic behaviors of the tails of the extinction time and the maximal displacement of $Z^0$
	were established in \cite{Hou24} under the assumption
	that the offspring distribution belongs to the domain of attraction of an $\alpha$-stable distribution, $\alpha\in (1,2]$,
      and some moment assumptions on the spatial motion.
		 It was also shown in  \cite{Hou24}
	 that the scaling limit under $\P_{\sqrt{t}y}(\cdot|\zeta>t)$ can be represented in terms of a super killed Brownian motion.
	In the subcritical case, i.e., $m\in(0,1)$, under the assumption $\sum^\infty_{k=1}k(\log k)p_k<\infty$, the asymptotic behaviors of the survival
	probability and the all time maximal position of branching killed Brownian motion with drift
	were established in \cite{Hou2024} recently.

	The asymptotic behavior of branching L\'{e}vy processes have been studied earlier.
	In the critical case, i.e. $m=1$ and $p_1\neq 1$, Sawyer and Fleischman \cite{Sawyer79} investigated the tail behavior of the all time maximal position of branching Brownian motion under the assumption that the offspring
	distribution
	has finite third moment. For a critical branching random walk with spatial motion having finite $(4+\varepsilon)$th moment, the tail behavior of the all time maximum was obtained by Lalley and Shao \cite{Lalley15}.
	Hou et al. \cite{HJRS23} studied the asymptotic behavior of  the all time maximum of critical branching L\'{e}vy processes with offspring distribution belonging to the domain of attraction of an $\alpha$-stable distribution with $\alpha \in(1,2]$, under some assumptions on the spatial motion. In the subcritical case, Profeta \cite{Profeta24} gave the asymptotic behavior of the all time maximal position under the assumption that the offspring distribution has finite third moment.  For related results about subcritical branching random walks, we refer the reader to \cite{Neuman17}.

	The purpose of this paper is to extend the results of \cite{Hou2024} to subcritical branching killed L\'{e}vy processes. This extension is quite challenging since properties of Brownian motion were used crucially in
	\cite{Hou2024}.  Fluctuation theory of L\'{e}vy processes will play
	an important role in this paper.
	Another important tool is the conditioned limit theorem in Theorem \ref{lemma-tau0>t-rho<0} below.

	\subsection{Main results}
	Before we state our main results, we introduce some notation and some basic results on L\'{e}vy processes.  We always assume that the offspring distribution is subcritical, i.e., $m\in (0, 1)$. Let $\alpha:=\beta(1-m)$ and let $f$ be the generating function of the offspring distribution, i.e. $f(s)=\sum_{k=0}^{\infty}p_k s^k$, $s\in[0,1]$. Define
	\begin{align}\label{def_Phi}
		\Phi(u):= \beta\left(f(1-u)-(1-u)\right)=: \left(\alpha +\varphi(u)\right)u, \quad u\in [0, 1],
	\end{align}
	where $\varphi(u)=\frac{\Phi(u)-\alpha u}{u}$ for $u\in (0, 1]$ and $\varphi(0)= \Phi'(0+)-\alpha=0$.
	According to  \cite[Lemma 2.7]{Hou2024}, $\varphi(\cdot)$ is increasing on $[0,1]$ and under the condition
	\begin{align}\label{LLogL-moment-condition}
		\sum_{k=1}^\infty   k (\log k )  p_k <\infty,
	\end{align}
	it holds that
	 \begin{align}\label{property-varphi}
		 \int_0^\infty \varphi\left(e^{-ct }\right)\mathrm{d}t<\infty, \quad \text{for any}~ c>0.
	 \end{align}
	 Moreover, it is well-known (see Theorem 2.4 in \cite[p.121]{AH1983}) that
	\begin{align}\label{Decay-Survival-Probability}
		\lim_{t\to\infty} e^{\alpha t}
		\P_0(\widetilde{\zeta}>t)
		= C_{sub}\in (0,\infty)
	\end{align}
	holds if and only if \eqref{LLogL-moment-condition} holds.
	For any $t>0$, define
	\begin{align}\label{def-g(t)}
		g(t):=\P_0(\widetilde{\zeta}>t).
	\end{align}
	It is well-known that $g(t)$ satisfies the equation
	\begin{align}
		\frac{\mathrm{d}}{\mathrm{d} t} g(t)= -\Phi(g(t))= -\left(\alpha + \varphi (g(t))\right)g(t),
	\end{align}
	thus
	\begin{align}\label{eq-eg(t)}
		e^{\alpha t}g(t) = \exp\left\{-\int_0^t \varphi(g(s))\mathrm{d}s \right\}.
	\end{align}
	It follows from  \eqref{Decay-Survival-Probability} that
	\begin{align}\label{Constant-C-sub}
		C_{sub}= \exp\left\{-\int_0^\infty \varphi(g(s))\mathrm{d}s \right\}.
	\end{align}
	Therefore, \eqref{LLogL-moment-condition} is equivalent to
	\begin{align}\label{Integral-of-phi-is-finite}
		\int_0^\infty \varphi(g(s))\mathrm{d} s<\infty.
	\end{align}
	In this paper, we always assume that
	$\xi=((\xi_t)_{t\ge 0},(\mathbf{P}_x)_{x\in \R})$
	is a L\'{e}vy process on $\mathbb{R}$ with
$$-\log \mathbf{E}_x\left(e^{i\theta(\xi_1-\xi_0)}\right)=i a\theta+\frac12\eta^2\theta^2+\int^\infty_{-\infty}(1-e^{i\theta x}
		+ i\theta x1_{\{|x|<1\}})\Pi(dx), \quad \theta\in \R, $$
where $\mathbf{E}_x$ stands for the expectation with respect to $\mathbf{P}_x$, $a\in \mathbb{R}$,
	$\eta\ge 0$ and  the L\'{e}vy measure $\Pi$ satisfies
	$\int_{\mathbb{R}}\left(1\land x^2\right)\Pi (\mathrm{d} x)<+\infty$.
	For any $z\in \mathbb{R}$, define
	\begin{align}
		\tau_z^{+}:=\inf\{t>0:\xi_t\ge z\}
		\quad \text{and} \quad
		\tau_z^{-}:=\inf\{t>0:\xi_t< z\}.
	\end{align}
	 Define the function
	\begin{align}\label{def-R}
		R(x):=x-\mathbf{E}_x\left(\xi_{\tau_0^-}\right)=-\mathbf{E}_0\left(\xi_{\tau_{-x}^-}\right),\quad x\ge 0.
	\end{align}
	It follows from \cite[Lemma 2.8]{Hou24} that if $\mathbf{E}_0(\xi_1)=0$ and $\mathbf{E}_0(\xi_1^2)\in(0,\infty)$, then $\mathbf{E}_x|\xi_{\tau_0^-}|<\infty$ and
	 $R(x)$ satisfies the following:
	\begin{enumerate}
			\item $R(x)\ge x$
			and $R(x)$ is non-decreasing in $x$;
			\item
			there exists a constant $c>0$ such that $R(x)\le c(1+x)$ and
			\begin{align}\label{limit-property-R}
				\lim_{x\to\infty}\frac{R(x)}{x}=1-\lim_{x\to\infty}\frac{\mathbf{E}_x(\xi_{\tau_0^-})}{x}=1;
			\end{align}
			\item
			$\big(R(\xi_s)1_{\{\tau_0^->s\}}\big)_{s\ge 0}$ is a $\mathbf{P}_x$-martingale for any $x>0$.
		\end{enumerate}
			In the case $\xi$ is a Brownian motion with drift, it is obvious that
			\begin{equation}\label{e:R}
			R(x)=x, \quad x>0.
			\end{equation}

	In some results, we will assume that $\xi$
	satisfies one or both of the following conditions:
	
	{\bf(H1)} There exists $\delta\in(0,1)$ such that $\mathbf{E}_x\left(|\xi_1|^{2+\delta}\right)<\infty.$

	{\bf(H2)} The law of $\xi_1$ is non-lattice,
	i.e., $\mathbf{P}_x\left(\xi_1\in h\mathbb{Z}+a\right) \neq 1,\forall h>0, a\in [0,h)$.
	\begin{remark}
			Condition {\bf (H2)} will be assumed in the case $\mathbf{E}_0(\xi_1) < 0$. In this case, we rely on the conditioned limit theorem for random walks established by \cite{GX-AIHP}, which requires the non-lattice condition.
	\end{remark}
		In the case  $\mathbf{E}_0\left(\xi_1\right)<0$, we will perform an Esscher transform on the L\'evy process. For this, we assume that

{\bf(H3)}
The Laplace exponent $\Psi(\lambda):=\log \mathbf{E}_0\left(e^{\lambda \xi_1}\right)$ is finite
for all $\lambda \in (\Lambda_1,\Lambda_2)$ with $\Lambda_1 \in [-\infty, 0]$ and $\Lambda_2 \in (0, \infty]$.
Moreover, there exists a unique $\lambda_*\in (0, \Lambda_2)$ such that  $\Psi'(\lambda_*)=0$.

Note that
$\Psi(\lambda)$ is finite if and only if
$\int_{\{|x|\ge 1\}}e^{\lambda x}\Pi(dx)<\infty$ and that for any $\lambda\in (\Lambda_1,\Lambda_2)$,
	\begin{align}\label{def-Psi}
		\Psi(\lambda)
		=a  \lambda
		+\frac{\eta^2}{2}\lambda^2+\int_{\mathbb{R}}\left(e^{\lambda x}-1-\lambda x 1_{\{|x|<1\}}\right) \Pi (\mathrm{d} x).
	\end{align}
Note also that 
$\Psi$ is convex in $(\Lambda_1,\Lambda_2)$.

\begin{remark}
If $\xi=((\xi_t)_{t\ge 0},(\mathbf{P}_x)_{x\in \R})$
is a spectrally negative L\'{e}vy process, then 
$\xi$ has finite Laplace exponent in $(0,\infty)$.
If $\xi$ is a spectrally negative L\'{e}vy process with $\mathbf{E}_0\left(\xi_1\right)<0$, then $\Psi$ admits a unique minimum at a $\lambda_*>0$ and $\Psi(\lambda_*)<0$, $\Psi'(\lambda_*)=0$ and $\Psi''(\lambda_*)>0$. So in this case {\bf(H3)}  is automatically satisfied.
\end{remark}

For any $c\in(\Lambda_1,\Lambda_2)$ and $x\in \R$,
	since $\{e^{c(\xi_t-x)-\Psi(c)t}:t\ge 0\}$ is a  $\mathbf{P}_x$-martingale, we can define the change of measure
	\begin{align}\label{change-measure}
		\frac{\mathrm{d} \mathbf{P}_x^c}{\mathrm{d} \mathbf{P}_x}\Big|_{\mathcal{F}_t}=e^{c(\xi_t-x)-\Psi(c)t},
	\end{align}
	where $\mathcal{F}_t:={\sigma\{\xi_s:s\le t\}}$, $t\ge 0$.
	According to  \cite[Theorem 3.9]{Kyprianou14},
	$\xi^{(c)}=((\xi_t)_{t\ge 0}, ( \mathbf{P}_x^c)_{x\in \R})$
	is also a  L\'{e}vy process and its Laplace exponent $\Psi_c(\lambda)$ is given by
		$\Psi_c(\lambda):=\Psi(\lambda+c)-\Psi(c)$.
	We will use $\mathbf{E}^c_x$ to denote expectation with respect to $\mathbf{P}_x^c$.

Recall that $\int_{\{|x|\ge 1\}}e^{\lambda x}\Pi(dx)<\infty$ for any $\lambda \in (\Lambda_1,\Lambda_2)$.
	According to \cite[Theorem 3.9]{Kyprianou14}, the L\'{e}vy measure of
	$\xi^{(c)}$ 
	is given by
		$e^{cx}\Pi(\mathrm{d}x)$.
	Combining the two facts above, we get that,  if  $\xi$ 
 has finite $p$-th moment with $p\ge 1$, then
for any $c\in(\Lambda_1,\Lambda_2)$, 
$\xi^{(c)}$ 
also has finite $p$-th moment 
and  so 
$\xi^{(c)}$ 
satisfies {\bf(H1)}.
        It is also easy to see that $\xi^{(c)}$ 
	is non-lattice if and only if $\xi$ 
	is non-lattice.
		We note  that,  by  \cite[Theorem 3.9]{Kyprianou14}, (i) if 
		$\xi$ 	is a  spectrally negative L\'evy process with Laplace exponent $\Psi$, then $\xi^{(c)}$
	is  a spectrally negative L\'{e}vy process with Laplace exponent $\Psi_c(\lambda)$ 
given by $\Psi_c(\lambda):=\Psi(\lambda+c)-\Psi(c)$; and 	
	(ii) if $\xi$ 
	 is a Brownian motion with drift, 
	 $\xi^{(c)}$ is also a Brownian motion with drift.

		When $\mathbf{E}_0\left(\xi_1\right)<0$ and \textbf{(H3)} holds, we take $c=\lambda_*$ and define the change of measure
	\begin{align}\label{change-measure-levy}
		\frac{\mathrm{d}\mathbf{P}_x^{\lambda_*}}{\mathrm{d}\mathbf{P}_x}\Big|_{\mathcal{F}_t}=e^{\lambda_* (\xi_t-x)-\Psi(\lambda_*)t}.
	\end{align}
	Then $\xi^{(\lambda_*)}$ is a
	L\'{e}vy process and its Laplace exponent is given
	 by $\Psi_{\lambda_*}(\lambda):=\Psi(\lambda+\lambda_*)-\Psi(\lambda_*)$.
	It is easy to see that 
	$\Psi_{\lambda_*}'(0+)=\Psi'(\lambda_*)=0$.
	Let $\mathbf{E}_x^{\lambda_*}$ be the expectation with respect to $\mathbf{P}_x^{\lambda_*}$.
	If $\xi$
	satisfies {\bf(H1)}, then since $\mathbf{E}_0^{\lambda_*}(\xi_1)=\Psi_{\lambda_*}'(0+)=0$, by
	\cite[Lemma 2.8]{Hou24},
	we have $\mathbf{E}_x^{\lambda_*}|\xi_{\tau_0^-}|<\infty$.
	Define
	\begin{align}\label{def-widetilde-R}
		R^*(x)=x-\mathbf{E}_x^{\lambda_*}\left(\xi_{\tau_0^-}\right),\quad x\ge 0.
	\end{align}
 Define the dual process of
 $\xi$ by:
	\begin{align}\label{def-eta*}
		\widehat{\xi}_s:=-\xi_s,\quad s\ge 0.
	\end{align}
	For any $z\in \mathbb{R}$, we define
	$\widehat{\tau}_z^-:=\inf\{s>0:\widehat{\xi}_s<z\}$ and
	\begin{align}\label{def-R*}
		\widehat{R}^*(x):=x-\mathbf{E}_x^{\lambda_*}\left(\widehat{\xi}_
		{\widehat{\tau}_0^-}
		\right),\quad x\ge 0.
	\end{align}

Denote $\R_+=[0, \infty)$.
Let $\sigma^2:=\mathbf{E}_0(\xi_1^2)$.
	Our first main result is on the large-time 
	asymptotic behavior
	of the survival probability.
	
	\begin{theorem}\label{thm-survival-probability}
		Assume \eqref{LLogL-moment-condition} holds and
		$\xi$ 	is a L\'{e}vy process satisfying {\bf (H1)}.
		Let $x>0$.
		\begin{enumerate}
			\item
			If
						$\mathbf{E}_0\left(\xi_1\right)=0$,
			then
			\begin{align}
				\lim_{t\to\infty} \sqrt{t} e^{\alpha t} \P_x(\zeta>t)= \frac{2C_{sub}R(x)}{\sqrt{2\pi \sigma^2}},
			\end{align}
			where $C_{sub}$ is defined in \eqref{Constant-C-sub} and $R(x)$ in \eqref{def-R}.
			
			\item
			If
						$\mathbf{E}_0\left(\xi_1\right)>0$,
			then
			\begin{align}
				\lim_{t\to\infty}e^{\alpha t}\P_x(\zeta>t)
				=q_x C_{sub},
			\end{align}
			where $q_x:=\mathbf{P}_x\left(\tau_0^-=\infty\right)>0.$

			\item
						If $\mathbf{E}_0\left(\xi_1\right)<0$ and 
						$\xi$ satisfies {\bf (H2)} and {\bf (H3)},
			then
			\begin{align}
				\lim_{t\to\infty}t^{3/2}e^{(\alpha-\Psi(\lambda_*))t}\P_x(\zeta>t)
			 =\frac{2C_0R^*(x)e^{\lambda_*x}}{\sqrt{2\pi \Psi''(\lambda_*)^3}},
			\end{align}
			where
			$C_0:=\lim_{N\to\infty}e^{(\alpha-\Psi(\lambda_*))N}\int_{\R_+}\P_z(\zeta>N)e^{-\lambda_* z}\widehat{R}^*(z)\mathrm{d}z \in (0,\infty)$,
		$R^*$ is defined in \eqref{def-widetilde-R} and $\widehat{R}^*$ in \eqref{def-R*}.
		\end{enumerate}
	\end{theorem}

	\begin{remark}
		\cite[Theorem 1.1]{Hou2024} investigates the asymptotic behavior of the survival probability of a branching killed Brownian motion with drift $-\rho$. 
The first two statements of \cite[Theorem 1.1]{Hou2024} are as follows.

		(1) if $\rho=0$, then $\lim_{t\to\infty}\sqrt{t}e^{\alpha t}\mathbf{P}_x(\zeta>t)=\sqrt{\frac{2}{\pi}}C_{sub}x.$

		(2) If $\rho<0$, then $\lim_{t\to \infty}e^{\alpha t}\mathbf{P}_x(\zeta>t)
		=\left(1-e^{2\rho x}\right).$

\noindent  Combining Theorem \ref{thm-survival-probability} (1) and (2)
with \eqref{e:R}, we immediately recover the first two conclusions of 
 \cite[Theorem 1.1]{Hou2024}. 
	Furthermore, when 	$\xi$
	is a standard Brownian motion with drift $-\rho$, we have $\Psi(\lambda)=-\rho \lambda+\frac{1}{2}\lambda^2$ and $\lambda_*=\rho$. When $\rho>0$, a straightforward calculation yields that
	\begin{align}
		\lim_{t\to\infty}t^{3/2}e^{(\alpha+\frac{\rho^2}{2})t}\P_x(\zeta>t)
		=\frac{2C_0xe^{\rho x}}{\sqrt{2\pi }}.
	\end{align}
	This result is consistent with \cite[Theorem 1.1, (iii)]{Hou2024}, where
$C_0(\rho)=C_0$.
	\end{remark}

	Our second main result is on the asymptotic  behavior of the tail probability of $M_t$.
	
	\begin{theorem}\label{thm-tail probability-Mt}
		Assume \eqref{LLogL-moment-condition} holds and $\xi$ 
		is a L\'evy process satisfying {\bf (H1)}. Let $x>0$.
		\begin{enumerate}
			\item If
					$\mathbf{E}_0\left(\xi_1\right)=0$,
			then for any $y\ge 0$,
			we have
			\begin{align}
				\lim_{t\to\infty}\sqrt{t}e^{\alpha t} \P_x\left(M_t>\sqrt{t}y\right)
				= \frac{2C_{sub}R(x)}{\sqrt{2\pi \sigma^2}}
				e^{-\frac{y^2}{2\sigma^2}}.
			\end{align}
			
			\item If
						$\mathbf{E}_0\left(\xi_1\right)>0$,
			then for any $y\in\mathbb{R}$,
			we have
			\begin{align}
				\lim_{t\to\infty}e^{\alpha t}\mathbb{P}_x\left(M_t>\sqrt{t}y+\mathbf{E}_0\left(\xi_1\right) t\right)
			 =\frac{q_x C_{sub}}{\sqrt{2\pi}}\int_{\frac{y}{\sigma}}^{\infty}e^{-\frac{z^2}{2}} \mathrm{d}z.
			\end{align}
			
			\item
			If $\mathbf{E}_0\left(\xi_1\right)<0$ and $\xi$ satisfies {\bf (H2)} and {\bf (H3)}, then for any $y\ge 0$,
			we have
			\begin{align}
				\lim_{t\to\infty}t^{3/2}e^{\left(\alpha-\Psi(\lambda_*)\right)t} \mathbb{P}_x\left(M_t>y\right)
			= \frac{2
			C_1(y)
			R^*(x)e^{\lambda_*x}}{\sqrt{2\pi \Psi''(\lambda_*)^3}},
			\end{align}
			where $C_1(y):=\lim_{N\to\infty}e^{(\alpha-\Psi(\lambda_*))N}\int_{\R_+}\P_z(M_N>y)e^{-\lambda_* z}\widehat{R}^*(z)\mathrm{d}z\in (0,\infty)$.
		\end{enumerate}
	\end{theorem}

Note that
$\P_z(M_N>0)=\P_z(\zeta>N)$ for $z>0$, thus $C_0$ in Theorem \ref{thm-survival-probability}  and $C_1(0)$ in Theorem \ref{thm-tail probability-Mt} are the same.
	Combining the result above with \eqref{e:R} and \eqref{e:scaleW}, we immediately recover
	 \cite[Theorem 1.3]{Hou2024} as a corollary.

	Combining Theorems \ref{thm-survival-probability} and \ref{thm-tail probability-Mt},
	we immediately get the following Yaglom-type conditional limit theorem.
	
	\begin{corollary}\label{thm-yaglom-limit-theorem}
		Assume \eqref{LLogL-moment-condition} holds and
		$\xi$ is a L\'{e}vy process satisfying {\bf (H1)}.
		Let $x>0$.
		\begin{enumerate}
			\item
						If $\mathbf{E}_0\left(\xi_1\right)=0$,
			then we have
			\begin{align}
				\P_x\left(\frac{M_t}{\sqrt{t}}\in \cdot \Big|\zeta>t\right)
				\quad \stackrel{\mathrm{d}}{\Longrightarrow} \quad
							\mathcal{R}(\cdot),
			\end{align}
					where $\mathcal{R}$ is the
			Rayleigh distribution with density $\rho(z)=ze^{-z^2/2}1_{\{z>0\}}$.
			
			\item
			If
					$\mathbf{E}_0\left(\xi_1\right)>0$,
			then we have
			\begin{align}
				\P_x\left(\frac{M_t-\mathbf{E}_0\left(\xi_1\right) t}{\sqrt{t}}\in \cdot \Big|\zeta>t\right)
				\quad \stackrel{\mathrm{d}}{\Longrightarrow} \quad
N(0, \sigma^2),
			\end{align}
where 	$N(0, \sigma^2)$ is normal distribution with  mean $0$  and variance $\sigma^2$.
			\item
			If $\mathbf{E}_0\left(\xi_1\right)<0$ and $\xi$ satisfies {\bf (H2)} and {\bf (H3)},
			then there exists a random variable
			$(X,\mathbb{P})$
whose law is independent of $x$ such that
			\begin{align}
				\P_x\left(M_t \in \cdot \Big|\zeta>t\right)
				\quad \stackrel{\mathrm{d}}{\Longrightarrow} \quad
 \mathbb{P}(X\in \cdot).
			\end{align}
		\end{enumerate}
	\end{corollary}
	
	In the following theorem we assume that $\xi$
	is a spectrally negative L\'{e}vy process with Laplace exponent $\Psi$. For $q\ge 0$, let
	\begin{align}
		\psi(q):=\sup\{\lambda\ge 0:\Psi(\lambda)=q\}
	\end{align}
	be the right inverse of $\Psi$.
	By Kyprianou \cite[Theorem 8.1]{Kyprianou14},
	for any $q\ge 0$, there exists a  scale function $W^{(q)}:\R \to [0,\infty)$ such that $W^{(q)}(x)=0$ for $x<0$ and $W^{(q)}$ is a strictly increasing and continuous function on $[0,\infty)$ with Laplace transform
	\begin{align}\label{Laplace-W}
		\int_{0}^{\infty}e^{-r x}W^{(q)}(x) \mathrm{d} x
		=\frac{1}{\Psi(r)-q},
		\quad \text{for}~r>\psi(q).
	\end{align}
	In the case when $\xi$
	is a standard Brownian motion with drift $-b$, by using tables of Laplace transforms, one can easily get that
	\begin{equation}\label{e:scaleW}
	W^{(q)}(x)=\frac{2e^{bx}}{\sqrt{b^2+2q}}
\sinh(\sqrt{b^2+2q}x), \quad x\ge 0,~q\ge 0.
	\end{equation}
	 Our third main result is on the asymptotic behavior of the all-time maximum $M$
	of branching killed spectrally negative L\'{e}vy process.

	\begin{theorem}\label{thm-tail-M}
		Assume that \eqref{LLogL-moment-condition}
		holds and that $\xi$
		is a spectrally negative L\'{e}vy process.
There exists a constant $C_2(\alpha)\in(0, 1]$ such that
for any $x>0$,
		\begin{align}
			\lim_{y\to\infty}e^{\psi(\alpha)y}\P_x(M>y)=
			C_2(\alpha)
			W^{(\alpha)}(x)\Psi'(\psi(\alpha)),
		\end{align}
where $W^{(\alpha)}$ is the scale function of $((\xi_t)_{t\ge 0}, (\mathbf{P}_x)_{x\in \R})$.
	\end{theorem}

\begin{remark}
	The reason we consider spectrally negative L\'{e}vy processes here, rather than general L\'{e}vy processes, is that the proof of Theorem \ref{thm-tail-M} is closely related to the two-sided exit problem. For general L\'{e}vy processes, there are no tractable expressions for quantities of interest related to the two-sided exit problem.
	Combining the result above with \eqref{e:scaleW}, we immediately recover \cite[Theorem 1.2]{Hou2024} as a corollary.
	Profeta \cite[Theorem 1.1]{Profeta24}  proved the following asymptotic behavior of the all-time maximum $\widetilde{M}$ for spectrally negative branching L\'{e}vy processes without killing
	\begin{align}\label{tail-behavior-M-without-killing}
		\mathbb{P}(\widetilde{M}\ge x) \sim \kappa e^{-\psi(\alpha) x},
\quad \mbox{ as } x\to\infty,
	\end{align}
under the third-moment condition on the offspring distribution $\{p_k\}_{k\ge 0}$, 	
	where $\kappa$ is a positive constant. Comparing  Theorem \ref{thm-tail-M} with \eqref{tail-behavior-M-without-killing},
	we observe that the killing barrier does not affect the exponential
    decay rate of the tail probability of the all-time maximum, it only affects  
    the limits after the same exponential scaling.
\end{remark}

	\subsection{Proof strategies and organization of the paper}
	The rest of the paper is organized as follows. In Section \ref{Preliminary}, we give some results on	L\'{e}vy processes which will be used in the proofs of our main results. We establish the conditioned limit theorem for	L\'{e}vy processes
	in Section \ref{section-CLT-LP}. The proofs of Theorems \ref{thm-survival-probability} and \ref{thm-tail probability-Mt} are given in Section \ref{section-proof-thm1-thm2}, and the proof of Theorem \ref{thm-tail-M} is given in Section \ref{section-proof-thm-tail-M}.

	Now we sketch the main idea of the proof of Theorem \ref{thm-survival-probability}. The main idea for the proof of Theorem \ref{thm-tail probability-Mt} is similar, and Corollary \ref{thm-yaglom-limit-theorem} follows from
	Theorems  \ref{thm-survival-probability} and  \ref{thm-tail probability-Mt}.
	For any $x,t>0$, let
	\begin{align}
		u(x,t):=\mathbb{P}_x(\zeta>t).
	\end{align}
	In Lemma \ref{lemma-expression-Q}, we derive a representation for $u(x,t)$. Lemma \ref{lemma-lower-bound-liminf} then establishes a lower bound for $u(x,t)$, while Lemmas \ref{lemma-upper-bound-limsup-rho=0} and \ref{lemma-upper-bound-limsup-rho>0} provide upper bounds for $u(x,t)$ in the cases $\mathbf{E}_0(\xi_1)=0$ and $\mathbf{E}_0(\xi_1)>0$, respectively. Theorem \ref{thm-survival-probability} (1) and (2) follow immediately from the above lemmas. In the case $\mathbf{E}_0(\xi_1)<0$, a quasi-stationary distribution exists, and the proof technique differs from those used in the previous two cases. The analysis of its asymptotic behavior relies on Theorem \ref{lemma-tau0>t-rho<0}, which establishes a conditioned limit theorem for L\'{e}vy processes.

	In this paper, we use $\phi(\cdot)$ to denote the standard normal density, i.e., $\phi(t)=\frac1{\sqrt{2\pi}}e^{-t^2/2}$, use
	$\rho(\cdot)$ to denote the Rayleigh density, i.e.,
	 $\rho(x)=xe^{-x^2/2}1_{\{x>0\}}$, and use $\mathcal{R}(x)$ to denote the Rayleigh distribution function, i.e., $\mathcal{R}(x)=(1-e^{-x^2/2})1_{\{x\ge 0\}}$.
	For $v>0$, we define
	$\phi_v(x)=\frac1{\sqrt{2\pi v}}e^{-x^2/(2v)}$ and
	$\rho_{v}(x)=(x/v)e^{-x^2/(2v)}1_{\{x>0\}}$.
	We use $F(x)\sim G(x)$
as $x\to\infty$ to denote $\lim_{x\to\infty}F(x)/G(x)=1$.
In this paper, capital letters $C_i$ and $T_i$, $i=1, 2, \dots$, are used to denote constants in the statements of results and  their value remain the same throughout the paper. Lower case letters
$c_i$, $i=1, 2, \dots$, are  used for constants used in the proofs and their labeling starts anew
in each proof. $c_i(\epsilon)$ and $C_i(\epsilon)$ mean that the constants $c_i$ and $C_i$
depend on $\epsilon$.

	\section{Preliminaries}\label{Preliminary}
	In this section, we first present
		some preliminary results
	for spectrally negative L\'{e}vy processes,
		followed by a result for general L\'{e}vy processes.
	Assume for now that	$\xi$
	is a spectrally negative L\'{e}vy process with Laplace exponent $\Psi$.
Then for any $x>0$, 
	\begin{align}
		\mathbf{E}_0(\xi_{\tau_x^+}=x|\tau_x^+<\infty)=1.
	\end{align}
	Moreover, it is well known, see \cite[Section 8]{Kyprianou14}, that for 
	any $x>0$ and $q\ge 0$,
	\begin{align}
		\mathbf{E}_0\left(e^{-q \tau_x^+} 1_{\{\tau_x^+<\infty\}}\right)=e^{-\psi(q)x},
	\end{align}
	where $\psi$ is the right inverse of $\Psi$.
	The following result on exit probabilities is contained in \cite[Theorem 8.1]{Kyprianou14}.
	\begin{theorem}\label{thm-exit-problems}
			Assume that $\xi$
			is a spectrally negative L\'{e}vy process with Laplace exponent $\Psi$.
			For any
			$0<x\le y$ and $q\ge 0$,
			\begin{align}
				\mathbf{E}_x\left(e^{-q \tau_y^+} 1_{\{\tau_0^->\tau_y^+\}}\right)=\frac{W^{(q)}(x)}{W^{(q)}(y)},
			\end{align}
			where $W^{(q)}$ is the scale function of $\xi$.
	\end{theorem}

	The following result, which can be found in \cite[Lemma 8.4]{Kyprianou14} and \cite[Proposition 1]{Surya08}, gives
	the relationship between $W_c^{(q)}$ for different values of $q$, $c$,
	and the asymptotic  behavior of $W^{(q)}(x)$ as $x\to \infty$.

	\begin{lemma}\label{lemma-scale-property}
		Assume that $\xi$ 
		is a spectrally negative L\'{e}vy process with Laplace exponent $\Psi$.
		For any $x\ge 0$, the function $q\mapsto W^{(q)}(x)$ may be analytically extended to $q\in\mathbb{C}$.
Furthermore, for any $q\in \mathbb{C}$ and $c\in \R$ with $\Psi(c)<\infty$, we have
		\begin{align}\label{relation-W}
			W^{(q)}(x)=e^{cx} W_c^{(q-\Psi(c))}(x),\quad x\ge 0,
		\end{align}
		where $W_c^{(q-\Psi(c))}$ is the scale function of $\xi^{(c)}$.
		Furthermore,
		\begin{align}\label{limit-behvior-W}
			W^{(q)}(x)
			\sim \frac{e^{\psi(q)x}}{\Psi'(\psi(q))},\quad \textit{as}~x\to\infty.
		\end{align}
	\end{lemma}
	
	The following lemma is an important tool for proving Theorem \ref{thm-tail-M}.

	\begin{lemma}\label{lemma-change-measure}
		Assume that $\xi$
		is a spectrally negative L\'{e}vy process with Laplace exponent $\Psi$.
		For any $a>0$, $0<x\le y$ and nonnegative Borel function $h$, we have
		\begin{align}
			\mathbf{E}_x\left(1_{\{\tau_y^{+}<\tau_0^{-}\}} e^{-a \tau_y^{+}-\int_{0}^{\tau_y^{+}}h(\xi_s)\mathrm{d}s}\right)
			=e^{\psi(a)(x-y)}\mathbf{E}_x^{\psi(a)}\left(1_{\{\tau_y^{+}<\tau_0^{-}\}} e^{-\int_{0}^{\tau_y^{+}}h(\xi_s)\mathrm{d}s}\right).
		\end{align}
	\end{lemma}
	\begin{proof}
	By Theorem 6 on p16 of \cite{Chung05},
	$\{\tau_y^{+}< \tau_0^{-}\} \cap \{\tau_y^{+}< t\}=\{\tau_y^{+}\land t< \tau_0^{-}\} \cap \{\tau_y^{+}\land t < t\}$ is $\mathcal{F}_{\tau_y^{+}\land t}$-measurable.
For $a>0$, since $e^{-a \tau_y^{+}}1_{\{\tau_y^+=\infty\}}$=0,
		using \eqref{change-measure} with $c=\psi(a)$, we have
		\begin{align}\label{change-measue-stopping-time}
			&\mathbf{E}_x\left(
		1_{\{\tau_y^{+}<\tau_0^{-}\}}
		e^{-a \tau_y^{+}-\int_{0}^{\tau_y^{+}}h(\xi_s)\mathrm{d}s}\right)
			=\lim_{t\to\infty}\mathbf{E}_x\left(
		1_{\{\tau_y^{+}<\tau_0^{-},\tau_y^+<t\}}
		e^{-a \tau_y^{+}-\int_{0}^{\tau_y^{+}}h(\xi_s)\mathrm{d}s}
		 \right)\\
			&=\lim_{t\to\infty}\mathbf{E}_x^{\psi(a)}\left( e^{-\psi(a)(\xi_t-x)+a t}e^{-a \tau_y^{+}-\int_{0}^{\tau_y^{+}}h(\xi_s)\mathrm{d}s}
		1_{\{\tau_y^{+}<\tau_0^{-},\tau_y^+<t\}}
			\right)\nonumber\\
			&=\lim_{t\to\infty}\mathbf{E}_x^{\psi(a)}\left( e^{-a \tau_y^{+}-\int_{0}^{\tau_y^{+}}h(\xi_s)\mathrm{d}s}
		1_{\{\tau_y^{+}<\tau_0^{-},\tau_y^+<t\}}
			\mathbf{E}_x^{\psi(a)}\left(e^{-\psi(a)(\xi_t-x)+a t}\Big|{\mathcal{F}_{\tau_y^+\land t}}\right)
						\right).\nonumber
		\end{align}
		Note that $(e^{-\psi(a)(\xi_t-x)+a t})_{t\ge 0}$ is a $\mathbf{P}_x^{\psi(a)}$ is a martingale with respect to $\mathcal{F}_t$. Using the optional stopping theorem and the absence of positive jumps, we get that, on $\{\tau_y^+<t\}$,
		\begin{align}
			\mathbf{E}_x^{\psi(a)}\left(e^{-\psi(a)(\xi_t-x)+a t}\Big|{\mathcal{F}_{\tau_y^+\land t}}\right)
			=e^{-\psi(a)(\xi_{\tau_y^+\land t}-x)+a (\tau_y^+\land t)}
			=e^{-\psi(a)(y-x)+a \tau_y^+}.
		\end{align}
		Combining this with \eqref{change-measue-stopping-time} and using the fact that $\mathbf{P}_x^{\psi(a)}(\tau_y^+<\infty)=1$, we get
		\begin{align}\label{change-measure-1}
			\mathbf{E}_x\left(
		1_{\{\tau_y^{+}<\tau_0^{-}\}}
		e^{-a \tau_y^{+}-\int_{0}^{\tau_y^{+}}h(\xi_s)\mathrm{d}s}\right)
			&=e^{\psi(a)(x-y)}\mathbf{E}_x^{\psi(a)}\left(
		1_{\{\tau_y^{+}<\tau_0^{-}\}}
		e^{-\int_{0}^{\tau_y^{+}}h(\xi_s)\mathrm{d}s}\right).
		\end{align}
		This gives the desired result.
	\end{proof}

	The following lemma gives the joint asymptotic behavior of the tail of $\tau_0^-$ and the L\'{e}vy process $\xi$ when $\mathbf{E}_0\left(\xi_1\right)>0$,
	and this result holds for general L\'{e}vy processes rather than being restricted to the spectrally negative case.
	
	\begin{lemma}\label{lemma-tau0>t-rho>0}
		Assume that $\xi$
		is a L\'{e}vy process such that
		$\mathbf{E}_0(\xi_1)>0$ and $\sigma^2:=\mathbf{E}_0(\xi_1^2)<\infty$, then for any $x>0$,
		\begin{align}\label{lim-tua_0>t}
			\lim_{t\to\infty}\mathbf{P}_x(\tau_0^->t)
			=\mathbf{P}_x(\tau_0^-=\infty)
			=:q_x>0.
		\end{align}
		Moreover, for any $y\in \R$, we have
		\begin{align}
			\lim_{t\to\infty}\mathbf{P}_x\left(\tau_0^->t,\xi_t-\mathbf{E}_0\left(\xi_1\right) t>\sqrt{t}y\right)=
			\mathbf{P}_x(\tau_0^-=\infty)\int_{\frac{y}{\sigma}}^{\infty}\phi(z)\mathrm{d} z.
		\end{align}
	\end{lemma}
	\begin{proof}
		Note that \eqref{lim-tua_0>t} follows immediately from
		\cite[Proposition 17, p172]{Bertoin96}.
		Fix $t>0$, for $m\in (0, t)$,
		by the Markov property,
		\begin{align}
			\mathbf{P}_x\left(\tau_0^->t,\xi_t-
						    \mathbf{E}_0\left(\xi_1\right) t
			>\sqrt{t}y\right)
			&\le \mathbf{P}_x\left(\tau_0^->m,\xi_t-
						\mathbf{E}_0\left(\xi_1\right) t
			>\sqrt{t}y\right)\\
			&=\mathbf{E}_x
					\left(
				1_{\{\tau_0^->m\}}\mathbf{P}_{\xi_m}\left(\xi_{t-m}
							-\mathbf{E}_0\left(\xi_1\right)  t
				>\sqrt{t}y\right)
					\right).
		\end{align}
		By the  central limit theorem, for any $z$, as $t\to \infty$, we get
		\begin{align}\label{central-limit-theorem-xi}
			\mathbf{P}_{z}\left(\xi_{t-m}
					    -\mathbf{E}_0\left(\xi_1\right) t
			>\sqrt{t}y\right) \to
			 \int_{y}^{\infty}
			 \phi_{\sigma^2}(u)
			 \mathrm{d}u.
		\end{align}
		Letting $t\to \infty$ first, then  $m\to \infty$, we get that
		\begin{align}\label{upper-bound-limsup-joint-proba}
			\limsup_{t\to \infty}\mathbf{P}_x\left(\tau_0^->t,\xi_t-
					\mathbf{E}_0\left(\xi_1\right) t
			>\sqrt{t}y\right)
			\le \mathbf{P}_x\left(\tau_0^-=\infty\right) \int_{y}^{\infty}
			\phi_{\sigma^2}(z)
			\mathrm{d}z.
		\end{align}
		On the other hand, we have
		\begin{align}
			\mathbf{P}_x\left(\tau_0^->m,\xi_t-\mathbf{E}_0\left(\xi_1\right)  t>\sqrt{t}y\right)
			\le \mathbf{P}_x\left(\tau_0^->t,\xi_t-\mathbf{E}_0\left(\xi_1\right) t>\sqrt{t}y\right)
			+\mathbf{P}_x\left(\tau_0^-\in(m,t]\right).
		\end{align}
		It follows from \eqref{central-limit-theorem-xi} that
		\begin{align}
			&\lim_{m\to \infty}\lim_{t\to\infty}\mathbf{P}_x\left(\tau_0^->m,\xi_t-\mathbf{E}_0\left(\xi_1\right)  t>\sqrt{t}y\right)\\
			=&\lim_{m\to \infty}\lim_{t\to\infty}\mathbf{E}_x\left(
				1_{\{\tau_0^->m\}}\mathbf{P}_{\xi_m}\left(\xi_{t-m}
				-\mathbf{E}_0\left(\xi_1\right)  t
				>\sqrt{t}y\right)\right)\\
			=&\lim_{m\to \infty}\mathbf{P}_x\left(\tau_0^->m\right) \int_{y}^{\infty}
			\phi_{\sigma^2}(u)
			\mathrm{d}u
			=\mathbf{P}_x\left(\tau_0^-=\infty\right) \int_{y}^{\infty}
			\phi_{\sigma^2}(u)
			\mathrm{d}u ,
		\end{align}
		this combined with
		\begin{align}
			\lim_{m\to \infty}\lim_{t\to\infty}\mathbf{P}_x\left(\tau_0^-\in(m,t]\right)
			=\lim_{m\to \infty}\mathbf{P}_x\left(\tau_0^-\in(m,\infty)\right)
			=0
		\end{align}
		yields that
		\begin{align}
			\liminf_{t\to \infty}\mathbf{P}_x\left(\tau_0^->t,\xi_t-\mathbf{E}_0\left(\xi_1\right)  t>\sqrt{t}y\right)
			\ge \mathbf{P}_x\left(\tau_0^-=\infty\right) \int_{y}^{\infty}
			\phi_{\sigma^2}(z)
			\mathrm{d}z.
		\end{align}
		Combining this with \eqref{upper-bound-limsup-joint-proba}, we get the the desired result.
	\end{proof}

 \section{Conditioned limit theorems for L\'{e}vy processes}\label{section-CLT-LP}
	
	The purpose of this section is to prove Theorem \ref{lemma-tau0>t-rho<0}, a conditioned limit theorem for L\'{e}vy processes. Theorem \ref{lemma-tau0>t-rho<0} will play an important role in this paper.
	We make some preparations first. The following result follows from \cite[Lemmas 2.12 and 4.1]{Hou24}.

	\begin{lemma}\label{lemma-tau0>t-rho=0}
		Assume that $\xi$
		is a L\'{e}vy process satisfying $\mathbf{E}_0(\xi_1)=0$ and {\bf (H1)}.
		 Then for any $x>0$ and $a\in (0,\infty]$, it holds that
		 \begin{align}
			\lim_{t\to\infty}\sqrt{t}\mathbf{P}_x\left(\xi_t\le a\sqrt{t},\tau_0^->t\right)
			=\frac{2R(x)}{\sqrt{2\pi \sigma^2}}\int_{0}^{\frac{a}{\sigma}}
			\rho(z)
			\mathrm{d}z,
		 \end{align}
		 where $\sigma^2:=\mathbf{E}_0(\xi_1^2)$
		 and $\rho(z)$ denotes the Rayleigh density.
		 Furthermore, for any $x>0$ and any bounded continuous function $h$ on $(0,\infty)$, it holds that
		\begin{align}
			\lim_{t\to\infty}\sqrt{t}\mathbf{E}_x\left(h\left(\frac{\xi_t}{\sigma \sqrt{t}}\right)1_{\{\tau_0^->t\}}\right)
			=\frac{2R(x)}{\sqrt{2\pi \sigma^2}}\int_{0}^{\infty}
			\rho(z)
			h(z) \mathrm{d}z.
		\end{align}
	\end{lemma}

	Recall that $\delta$ is the constant in {\bf(H1)}
	and $\sigma^2=\mathbb{E}_0(\xi_1^2)$.
	The following result is \cite[Lemma 2.11]{Hou24}.

	\begin{lemma}\label{lemma-relation-with-BM}
		Assume that $\xi$ 
		is a L\'evy process satisfying $\mathbf{E}_0(\xi_1)=0$ and {\bf (H1)}.
		Then there exists a Brownian motion
		$W$ with diffusion coefficient $\sigma^2$
		starting from the origin
		such that
		for any $\kappa  \in(0,\frac{\delta}{2(2+\delta)})$,  there exists a constant $C_3(\kappa)>1$ such that for all $t\ge 1$,
		\begin{align}
			\mathbf{P}_x
			\left(\sup_{s\in[0,1]}|\xi_{ts}-x-W_{ts}|>t^{\frac{1}{2}-\kappa}\right)
			\le \frac{C_3(\kappa)
			}{t^{(\frac{1}{2}-\kappa)(\delta+2)-1}}.
		\end{align}
	\end{lemma}

The following lemma is a conditional limit theorem for L\'{e}vy processes. Its proof is similar to that of \cite[Lemma 4.1]{Hou24}, but it provides a more precise bound.
See  \cite[Theorem 2.7]{GX-AIHP} for an  analogous result for random walks.

	\begin{lemma}\label{lemma-effective-conditioned-integral-limit}
		Assume that $\xi$
		is a L\'evy process satisfying
         $\mathbf{E}_0(\xi_1)=0$, $\mathbf{E}_0(\xi_1^2)=\sigma^2$ and {\bf (H1)}.
	Then one can find a constant  $\varepsilon_0\in (0, \frac{\delta}{4(2+\delta)})$
	with the property that for any $\varepsilon\in(0,\varepsilon_0)$ there exist positive constants
	$T_0(\varepsilon)$ and $C_4(\varepsilon)$
	such that for any $x,y>0$ and $t>T_0(\varepsilon)$,
\begin{align}
			\Big| \mathbf{P}_x\left(\frac{\xi_t}{\sigma \sqrt{t}}\le y,\tau_0^->t\right)-\frac{2R(x)}{\sigma\sqrt{2\pi t} }
			\mathcal{R}
			(y)\Big|\le \frac{
			C_4(\varepsilon)
			(1+x)}
			{t^{1/2+\varepsilon}}.
		\end{align}
		\end{lemma}
	\begin{proof}
		Let $W$ be the Brownian motion in Lemma \ref{lemma-relation-with-BM}.
		For any $r>0$ and
				    $\epsilon \in(0,\delta/(4(5+2\delta)))$,
		define
		\begin{align}\label{def-Ar}
			A_r:
			=\Big\{\sup_{s\in[0,1]}|\xi_{sr}-\xi_0-W_{sr}|> r^{\frac{1}{2}-2\epsilon}\Big\}.
		\end{align}
		Let $(S_n)_{n\ge 0}$ be the random walk defined by $S_n:=\xi_n$, $n\in \mathbb{N}$.
		For any $b\in \R$, define
		\begin{align}
			\tau_b^{S,+}:
			=\inf\{j\in \mathbb{N}, |S_j|>b\}.
		\end{align}
		By the Markov property, we have the following decomposition:
		\begin{align}
			\mathbf{P}_x\left(\frac{\xi_t}{\sigma \sqrt{t}}\le y,\tau_0^->t\right)=\sum_{k=1}^{4}I_k,
		\end{align}
		where
		\begin{align*}
			I_1:=\mathbf{P}_x\left(\frac{\xi_t}{\sigma \sqrt{t}}\le y,\tau_0^->t,\tau_{t^{1/2-\epsilon}}^{S,+}>[t^{1-\epsilon}] \right),
		\end{align*}
		\begin{align*}
			I_2:=\sum_{k=1}^{[t^{1-\epsilon}]}
			\mathbf{E}_x
			\left(\mathbf{P}_{\xi_k}\left(\frac{\xi_{t-k}}{\sigma \sqrt{t}}\le y,\tau_0^->t-k,A_{t-k} \right);\tau_0^->k, \tau_{t^{1/2-\epsilon}}^{S,+}=k
				\right),
		\end{align*}
		\begin{align*}
			I_3:=\sum_{k=1}^{[t^{1-\epsilon}]}
			\mathbf{E}_x
			\left(\mathbf{P}_{\xi_k}\left(\frac{\xi_{t-k}}{\sigma \sqrt{t}}\le y,\tau_0^->t-k,A_{t-k}^c \right);\tau_0^->k, \xi_k>t^{(1-\epsilon)/2}, \tau_{t^{1/2-\epsilon}}^{S,+}=k\right),
		\end{align*}
		\begin{align*}
			I_4:=\sum_{k=1}^{[t^{1-\epsilon}]}
			\mathbf{E}_x
			\left(\mathbf{P}_{\xi_k}\left(\frac{\xi_{t-k}}{\sigma \sqrt{t}}\le y,\tau_0^->t-k,A_{t-k}^c \right);\tau_0^->k, \xi_k\le t^{(1-\epsilon)/2}, \tau_{t^{1/2-\epsilon}}^{S,+}=k\right).
		\end{align*}
		We now deal with $I_i$, $i=1, 2, 3, 4$, separately.
		
			(i) Upper bound of $I_1$.
			Set $K:=[t^{\epsilon}-1]$ and $l:=[t^{1-2\epsilon}]$. Since $Kl\le [t^{1-\epsilon}]$, we have
			\begin{align}\label{upper-bound-I1}
				I_1\le &\mathbf{P}_x\left(\tau_{t^{1/2-\epsilon}}^{S,+}>[t^{1-\epsilon}] \right)
				\le \mathbf{P}_0\left(\max_{1\le j\le Kl}|x+S_j|\le t^{1/2-\epsilon}\right)\\
				\le &\mathbf{P}_0\left(\max_{1\le j\le K}|x+S_{lj}|\le t^{1/2-\epsilon}\right)
				,\quad x>0.
			\end{align}
			By the Markov property, we have
			\begin{align}\label{upper-bound-I1-1}
				\mathbf{P}_0\left(\max_{1\le j\le K}|x+S_{lj}|\le t^{1/2-\epsilon}\right)
				\le \left(\sup_{x\in \mathbb{R}_+}\mathbf{P}_0\left(|x+S_{l}|\le t^{1/2-\epsilon}\right)\right)^K.
			\end{align}
			According to
			the display below \cite[(4.6)]{Hou24},
			there exist positive constants $c_1\in(0,1)$
			and
			$t_1(\epsilon)$
			such that for $t>t_1(\epsilon)$,
			\begin{align}
				\mathbf{P}_0\left(|x+S_{l}|\le t^{1/2-\epsilon}\right)<
				c_1
				\quad x\in \mathbb{R}_+.
			\end{align}
			Plugging this into \eqref{upper-bound-I1-1},
			taking $c_2=-\ln c_1$,
			and combining with \eqref{upper-bound-I1},
		we get that for $t>t_1(\epsilon)$,
			\begin{align}\label{upper-bound-I1-lemma3.2}
				I_1
				\le c_1^K=e^{-c_2[t^{\epsilon}-1]}
				\le
							\frac{c_2}{t^{1/2+\epsilon/8}}.
			\end{align}

			(ii) Upper bound of $I_2$.
			By part (ii) of the proof
			of \cite[Lemma 4.1]{Hou24},
			for any
					$\epsilon \in(0,\delta/(4(5+2\delta)))$, we have
			\begin{align}\label{upper-bound-I2-lemma3.2}
				I_2 \le
				\frac{
				C_3(2\epsilon)
				x}{t^{
									1/2+\delta/2
					-(5+2\delta)\epsilon}}
						\le \frac{
					C_3(2\epsilon)
					x}
										{t^{\frac{1}{2} +\epsilon/8}},
			\end{align}
where $C_3(2\epsilon)$
is the constant in Lemma \ref{lemma-relation-with-BM}.

			(iii)
			Upper bound of $I_3$.
			Repeating the argument in part (iii) of
			the  proof of
		\cite[Lemma 4.1]{Hou24}
			leading to \cite[(4.8)]{Hou24} and using \cite[Lemma 7.7]{GX},
	we can find
	$\epsilon_1>0$ with the property that for any
		$\epsilon\in (0,\epsilon_1\land \delta/(4(5+2\delta)))$
	there exists a positive constant
	$c_3(\epsilon)$ such that
\begin{align}\label{upper-bound-I3-lemma3.2}
				I_3&
				\le
				\frac{1}{\sqrt{t}}
				\sum_{k=1}^{[t^{1-\epsilon}]}
				\mathbf{E}_x\left(S_k;\tau_0^{S,-}>k, S_k>
				t^{(1-\epsilon)/2}, \tau_{t^{1/2-\epsilon}}^{S,+}=k\right)\\
				&\le \frac{
				c_3(\epsilon)
				(1+x)}{t^{1+\delta/2-\epsilon(1+\epsilon+\delta/2)}}
							\le  \frac{
					c_3(\epsilon)
					(1+x)}
									{t^{\frac{1}{2}+\epsilon/8}}.
		\end{align}

			(iv)
			Upper bound of $I_4$.
			For $k\le [t^{1-\epsilon}]$ and $x'>0$, define
		\begin{align}
			K(k,x'):=\mathbf{P}_{x'}\left(\frac{\xi_{t-k}}{\sigma \sqrt{t}}\le y,\tau_0^->t-k,A_{t-k}^c \right).
		\end{align}
		Set
		\begin{align}
			x^*:=\frac{x'+(t-k)^{\frac{1}{2}-2\epsilon}}{\sigma}\quad \text{and} \quad
			y^*:=\frac{y\sqrt{t}}{\sqrt{t-k}}+\frac{2}{\sigma (t-k)^{2\epsilon}}.
		\end{align}
		It follows from \cite[(4.13)]{Hou24} that
		\begin{align}\label{upper-bound-K}
			K(k,x')
			\le \frac{2}{\sqrt{2\pi (t-k)}} \left(\frac{x'}{\sigma}+\frac{t^{\frac{1}{2}-2\epsilon}}{\sigma}\right) \int_{0}^{y^*}
			\rho(z)e^{\frac{zx^*}{\sqrt{t-k}}}
			\mathrm{d}z.
		\end{align}
We claim that there exist positive constants $t_2(\epsilon)$, $c_4(\epsilon)$
such that for $t>t_2(\epsilon)$ and $p\ge 2$ sufficiently large,
\begin{align}\label{claim}
			K(k,x') \le \frac{2}{\sigma \sqrt{2\pi t}}  \left(1+\frac{
	                 c_4(\epsilon)
				}{t^{\epsilon/2-\epsilon^p}}\right)
			\left(\mathcal{R}(y)+\frac{
			c_4(\epsilon)
							}{t^{\epsilon/2}}
			\right)\left(x'+t^{\frac{1}{2}-2\epsilon}\right).
		\end{align}
To prove this claim, note
that for any $k\le [t^{1-\epsilon}]$ and $x'\le t^{(1-\epsilon)/2}$,
there exist positive constants $t_3(\epsilon)$,
$c_5(\epsilon)$ and $c_6(\epsilon)$ such that for $t>t_3(\epsilon)$,
the following holds:
		\begin{align}
			\frac{x^*}{\sqrt{t-k}}
			\le c_5(\epsilon)\frac{t^{(1-\epsilon)/2}+t^{\frac{1}{2}-2\epsilon}}{\sqrt{t}}
			\le c_6(\epsilon)t^{-\epsilon/2}.
		\end{align}
For $y\in [0,t^{\epsilon^p}]$ with $p\ge 2$ being a positive constant,  and $z\leq y^*$,
there exist
positive constants $t_4(\epsilon)$
and $c_7(\epsilon)$ such that for $t>t_4(\epsilon)$,
		\begin{align}
			z \le \frac{y\sqrt{t}}{\sqrt{t-k}}+\frac{2}{\sigma (t-k)^{2\epsilon}}
	\leq
	c_7(\epsilon)
			t^{\epsilon^p},
		\end{align}
		and thus there exists a positive constant
	 $c_8(\epsilon)$ such that for $t>t_3(\epsilon)\vee t_4(\epsilon)$,
	\begin{align}
			e^{\frac{zx^*}{\sqrt{t-k}}}
			\le e^{
		c_6(\epsilon)c_7(\epsilon)
				t^{-\epsilon/2}t^{\epsilon^p}}
			\le 1+\frac{
				c_8(\epsilon)
				}{t^{\epsilon/2-\epsilon^p}}, \quad \mbox{ for }z\leq y^*.
		\end{align}
		This implies that when $y\in [0,t^{\epsilon^p}]$,
	for $t>t_3 (\epsilon)\vee t_4(\epsilon)$,
		it holds that
		\begin{align}\label{upper-bound-[0,t^{varepsilon^p}]}
			\int_{0}^{y^*}
			\rho(z)e^{\frac{zx^*}{\sqrt{t-k}}}
			\mathrm{d}z
			&\le \left(1+\frac{
			c_8(\epsilon)
				}{t^{\epsilon/2-\epsilon^p}}\right)\int_{0}^{y^*}
			\rho(z)
			\mathrm{d}z.
		\end{align}
	Moreover, by the definition of $y^*$, there exist positive constants
	$t_5(\epsilon)$  and  $c_9(\epsilon)$ such that for $t>t_5(\epsilon)$,
		\begin{align}
			y^*-y\le \frac{
			c_9(\epsilon)
			}{t^{\varepsilon/2}}
		\end{align}
	Thus using the fact that $\rho(z)\le 1$ for all $z\ge 0$, we get
		that for any
				$\epsilon\in (0, \epsilon_1\land \delta/(4(5+2\delta)))$
		and $t>\max\{t_i(\epsilon): 3\le i\le 5\}$,
		\begin{align}\label{upper-bound-y-small}
			\int_{0}^{y^*}
			\rho(z)e^{\frac{zx^*}{\sqrt{t-k}}}
			 \mathrm{d}z
			&\le\left(1+\frac{
		c_8(\epsilon)
		}{t^{\epsilon/2-\epsilon^p}}
			\right)\left(\mathcal{R}(y) +y^*-y\right)\\
		&\le \left(1+\frac{
		c_8(\epsilon)
		}{t^{\epsilon/2-\epsilon^p}}\right)
			\left(\mathcal{R}(y)+\frac{
			c_9(\epsilon)
			}{t^{\epsilon/2}}\right).
		\end{align}
		For $y>t^{\epsilon^p}$,
		using \cite[(7.31)]{GX}
		we get that
	there exist positive constants
	$t_6(\epsilon)$ and $c_{10}(\epsilon)$
	such that for $t>t_6(\epsilon)$,
	\begin{align}\label{upper-bound-y-large}
			\int_{0}^{y^*}
			\rho(z)e^{\frac{zx^*}{\sqrt{t-k}}}
			\mathrm{d}z
			&\le \left(1+\frac{
			c_{10}(\epsilon)
			}{t^{\epsilon/2-\epsilon^p}}\right) \int_{0}^{y}\rho(z)\mathrm{d}z+
	                 c_{10}(\epsilon)
	                 e^{-
	                 c_{10}(\epsilon)
	                 t^{\epsilon^{p}}}\\
                         &\le  \left(1+\frac{
                         c_{10}(\epsilon)
                         }{t^{\epsilon/2-\epsilon^p}}\right)\left(
			\mathcal{R}(y)+\frac{
			c_{10}(\epsilon)
			}{t^{\epsilon}}
			\right).
		\end{align}
		Combining \eqref{upper-bound-K}, \eqref{upper-bound-y-small} and \eqref{upper-bound-y-large},
		we get that there exists a positive constant $c_{11}(\epsilon)$ such that  for any
		$y>0$ and  $t>\max\{t_i(\epsilon): 3\le i\le 6\}$, we have
		\begin{align}\label{upper-bound-K-1}
			K(k,x') \le \frac{2x^*}{\sqrt{2\pi (t-k)}}  \left(1+\frac{
			c_{11}(\epsilon)
			}{t^{\epsilon/2-\epsilon^p}}\right)\left(\mathcal{R}(y)+\frac{
			c_{11}(\epsilon)
			}{t^{\epsilon/2}}\right).
		\end{align}
		Since $k\le [t^{1-\epsilon}]$,
	there exists a constant
	 $t_7(\epsilon)>0$ such that when $t>t_7(\epsilon)$,
		$$\frac{1}{\sqrt{t-k}}\le \frac{1}{\sqrt{t}}\left(1+\frac{
			c_{12}(\epsilon)
			}{t^{\epsilon}}\right),$$
		 and
		\begin{align}\label{upper-bound-x*}
			\frac{x^*}{\sqrt{ t-k}}
			\le  \frac{1}{\sigma \sqrt{t}}\left(1+\frac{
			c_{13}(\epsilon)
			}{t^{\epsilon}}\right)\left(x'+t^{\frac{1}{2}-2\epsilon}\right),
		\end{align}
		for some positive constants $c_{12}(\epsilon)$ and $c_{13}(\epsilon)$.
            Taking $t_2(\epsilon):=\max\{t_i(\epsilon):3 \le i \le  7\}$,
		then the claim \eqref{claim} follows from \eqref{upper-bound-K-1}
		and \eqref{upper-bound-x*}.
		
		Note that on $\{\tau_{t^{1/2-\epsilon}}^{S,+}=k\}$, we have
		$\xi_k=S_k\ge t^{1/2-\epsilon}$.
        Thus, $t^{1/2-2\epsilon}\le t^{-\epsilon} \xi_k$
        on $\{\tau_{t^{1/2-\epsilon}}^{S,+}=k\}$.
        Also note that, by using that $\big(R(\xi_s)1_{\{\tau_0^->s\}}\big)_{s\ge 0}$ is a $\mathbf{P}_x$-martingale for any $x>0$ and the optional stopping theorem,
        \begin{align}\label{harmonic-R}
        R(x)=\mathbf{E}_x\left(R\big(\xi_{\tau_{t^{1/2-\epsilon}}^{S,+}}\big);\tau_0^->\tau_{t^{1/2-\epsilon}}^{S,+}\right), \quad x>0,\ t>0.
        \end{align}
	Hence, by \eqref{claim}, for $t>t_2(\epsilon)$,
		\begin{align}
			I_4 &\le
			\frac2{\sigma \sqrt{2\pi t}}\left(1+\frac{c_4(\epsilon)}{t^{\epsilon/2-\epsilon^p}}\right)
			\left(\mathcal{R}(y)+\frac{c_4(\epsilon)}{t^{\epsilon/2}}\right)\\
			&\quad \times \sum_{k=1}^{[t^{1-\epsilon}]}
			\mathbf{E}_x\left(\xi_k+t^{\frac{1}{2}-2\epsilon};\tau_0^->k,
			\xi_k \le t^{(1-\varepsilon)/2},
			\tau_{t^{1/2-\epsilon}}^{S,+}=k\right)
			\nonumber\\
			&\le
			\frac{2(1+t^{-\epsilon})}{\sigma \sqrt{2\pi t}}\left(1+\frac{c_4(\epsilon)}{t^{\epsilon/2-\epsilon^p}}\right)\left(\mathcal{R}(y)+\frac{c_4(\epsilon)}{t^{\epsilon/2}}\right)
			\sum_{k=1}^{[t^{1-\epsilon}]}\nonumber\\
			&\quad \times \mathbf{E}_x\left(\xi_k;\tau_0^->k,
			 \xi_k \le t^{(1-\epsilon)/2},
			\tau_{t^{1/2-\epsilon}}^{S,+}=k\right)\nonumber\\
			&\le
			\frac{2(1+t^{-\epsilon})}{\sigma \sqrt{2\pi t}}
			\left(1+\frac{c_4(\epsilon)}{t^{\epsilon/2-\epsilon^p}}\right)\left(\mathcal{R}(y)+\frac{c_4(\epsilon)}{t^{\epsilon/2}}\right)
			\mathbf{E}_x\left(\xi_{\tau_{t^{1/2-\epsilon}}^{S,+}};\tau_0^->\tau_{t^{1/2-\epsilon}}^{S,+},
			\tau_{t^{1/2-\epsilon}}^{S,+}\le [t^{1-\varepsilon}]\right)\nonumber\\
			&\le
			\frac{2(1+t^{-\epsilon})}{\sigma \sqrt{2\pi t}}
			\left(1+\frac{c_4(\epsilon)}{t^{\epsilon/2-\epsilon^p}}\right)\left(\mathcal{R}(y)+\frac{c_4(\epsilon)}{t^{\epsilon/2}}\right)
			\mathbf{E}_x\left(\xi_{\tau_{t^{1/2-\epsilon}}^{S,+}};\tau_0^->\tau_{t^{1/2-\epsilon}}^{S,+}\right)\nonumber\\
			&=
			 \frac{2R(x)(1+t^{-\epsilon})
			 }{\sigma \sqrt{2\pi t}}\left(1+\frac{
			c_4(\epsilon)
				}{t^{\epsilon/2-\epsilon^p}}\right)\left(\mathcal{R}(y)+\frac{c_4(\epsilon)}{t^{\epsilon/2}}\right),
		\end{align}
		where in the last equality we used \eqref{harmonic-R}.
	Thus, there exist positive constants
	$t_8(\epsilon)$, $c_{14}(\epsilon)$ and $c_{15}(\epsilon)$
	such that for $t>t_8(\epsilon)$,
        		\begin{align}\label{upper-bound-I4-lemma3.2}
			I_4
			\le \frac{2 R(x)
			\left(1+\frac{c_{14}(\epsilon)}{t^{\epsilon/2-\epsilon^p}}\right)\left(\mathcal{R}(y) +\frac{c_{14}(\epsilon)}{t^{\epsilon/2}}\right)
			}{\sigma \sqrt{2\pi t}}
				\le \frac{2 R(x)\mathcal{R}
				(y)}{\sigma \sqrt{2\pi t}}
				+\frac{
					c_{15}(\epsilon)
					(1+x)}
					{t^{\frac{1}{2}+\epsilon/8}},
						\end{align}
	where in the last inequality we use the fact that $R(x)\le c(1+x)$ for some constant $c>0$.

	(v) Lower bound of $I_4$.
		Repeating the proof of \cite[(7.40)]{GX},
we get that there exist positive constants
$t_{9}(\epsilon)$,  $c_{16}(\epsilon)$ and $c_{17}(\epsilon)$
such that for $t>t_{9}(\epsilon)$,
		\begin{align}\label{lower-bound-I4-lemma3.2}
			I_4
			&\ge \frac{2 R(x)}{\sigma \sqrt{2\pi t}}
			\left(1-\frac{
				c_{16}(\epsilon)
				}{t^{\epsilon/2-\epsilon^p}}\right)
					\left(\mathcal{R}(y)-\frac{1}{t^{2\epsilon}}\right)
			-\frac{
				c_{16}(\epsilon)
				(1+x)}{t^{\delta/2-\epsilon(1+\epsilon+\delta/2)}}\\
			&\ge \frac{2 R(x)\mathcal{R}(y)}{\sigma \sqrt{2\pi t}}
			-\frac{
				c_{17}(\epsilon)
				(1+x)}
				{t^{\frac{1}{2}+\epsilon/8}}.
		\end{align}
Set $\epsilon_0:=\min \{\delta/(4(5+2\delta)),\epsilon_1\}$,
$\varepsilon:=\epsilon/8$, $\varepsilon_0:=\epsilon_0/8$
and $T_0(\epsilon):=\max\{t_i(\epsilon):i=1,2, 8, 9\}$.
Using the fact that there exists $c_{18}>0$ such that $R(x)\le c_{18}(1+x)$,
		 and combining \eqref{upper-bound-I1-lemma3.2}, \eqref{upper-bound-I2-lemma3.2}, \eqref{upper-bound-I3-lemma3.2}, \eqref{upper-bound-I4-lemma3.2} and \eqref{lower-bound-I4-lemma3.2}, we arrive at the conclusion of the lemma.
	\end{proof}

The duality relations in the following lemma, especially \eqref{duality-formula-1}, are well known in  probabilistic potential theory. We give an elementary proof here for the reader's convenience.
	
	\begin{lemma}\label{lemma-duality-formula}
		For any $t>0$ and any bounded
		Borel functions
		$g,h:\mathbb{R}\to \mathbb{R}_+$,
		 we have
		\begin{align}\label{duality-formula-1}
			\int_{\mathbb{R}_+}
			h(x)\mathbf{E}_x\left(g(\xi_t)1_{\{\tau_0^->t\}}\right)\mathrm{d}x
			=
			 \int_{\mathbb{R}_+}
			g(y)\mathbf{E}_y\left(
			h(\widehat{\xi}_t)1_{\{
			\widehat{ \tau}_0^->t
			\}}\right)\mathrm{d}y
		\end{align}
		and
		\begin{align}\label{duality-formula-2}
					\int_{\mathbb{R}}
			h(x)\mathbf{E}_x\left(g(\xi_t)1_{\{\tau_0^-\le t\}}\right)\mathrm{d}x
					 =\int_{\mathbb{R}}g(y)\mathbf{E}_y\left(h(\widehat{\xi}_t)1_{\{\widehat{ \tau}_0^-\le t\}}\right)\mathrm{d}y.
		\end{align}
	\end{lemma}
	\begin{proof}
For $x>0$, by the change of variables $x+\xi_t=y$, we get
		\begin{align}
			&
			 \int_{\mathbb{R}_+}
				h(x)\mathbf{E}_x\left(g(\xi_t)1_{\{\tau_0^->t\}}\right)\mathrm{d}x
			= \int_{\mathbb{R}_+}
						h(x)\mathbf{E}\left(g(x+\xi_t)1_{\{\tau_{-x}^->t\}}\right)\mathrm{d}x\\
			&=	 \int_{\mathbb{R}_+}
					h(x)\mathbf{E}\left(g(x+\xi_t),
					 \inf_{s\le t}\xi_s>-x\right)\mathrm{d}x
			= \int_{\mathbb{R}_+}
						h(x)\mathbf{E}\left(g(x+\xi_t),\inf_{s\le t}\xi_{t-s}>-x\right)\mathrm{d}x\\
			&=	\int_{\mathbb{R}_+}
				g(y)\mathbf{E}\left(h(y-\xi_t),\inf_{s\le t}(\xi_{t-s}-\xi_t)>-y\right)\mathrm{d}y\\
			&=\int_{\mathbb{R}_+}
						g(y)\mathbf{E}
			\left(h(y+\widehat{\xi}_t),\inf_{s\le t}\widehat{\xi}_s>-y\right)\mathrm{d}y
			=\int_{\mathbb{R}_+}
						g(y)\mathbf{E}_y\left(
			h(\widehat{\xi}_t)1_{\{
			\widehat{ \tau}_0^->t
			\}}\right)\mathrm{d}y,
		\end{align}
		which completes the proof of \eqref{duality-formula-1}.
		Using the same argument, we can also get
		\begin{align}\label{dual-on-R}
			&\int_{\mathbb{R}}
			h(x)\mathbf{E}_x(g(\xi_t))\mathrm{d}x
			=\int_{\mathbb{R}}
			h(x)\mathbf{E}(g(x+\xi_t))\mathrm{d}x
			=\int_{\mathbb{R}}g(y)
			\mathbf{E}(h(y-\xi_t))\mathrm{d}y\\
			&=\int_{\mathbb{R}}g(y)
			\mathbf{E}(h(y+\widehat{\xi}_t))\mathrm{d}y
			=\int_{\mathbb{R}}g(y)
			\mathbf{E}_y(h(\widehat{\xi}_t))\mathrm{d}y.
		\end{align}
			Note that for $x<0$, $\mathbf{P}_x(\tau_0^->t)=\mathbf{P}_x(\widehat{\tau}_0^->t)=0$.
		Therefore,
				 \eqref{duality-formula-1} is equivalent to
\begin{align}\label{duality-formula-1'}
			 \int_{\mathbb{R}}
			h(x)\mathbf{E}_x\left(g(\xi_t)1_{\{\tau_0^->t\}}\right)\mathrm{d}x
			=	 \int_{\mathbb{R}}
			g(y)\mathbf{E}_y\left(
			h(\widehat{\xi}_t)1_{\{
			\widehat{ \tau}_0^->t
			\}}\right)\mathrm{d}y.
		\end{align}
		Combining this with
		\eqref{dual-on-R},
		we get \eqref{duality-formula-2}.
	\end{proof}

	Before stating Theorem \ref{lemma-tau0>t-rho<0},
	we first introduce some necessary notation and definitions. Let
	$h_1,h_2:\R \to \R_+$
	be Borel functions and $\varepsilon>0$.
	We say that $h_1$ $\varepsilon$-dominates $h_2$ and
	 write $h_2 \le_{\varepsilon} h_1$ if
	\begin{align}
		h_2(u)\le h_1(u+v),\quad
		\forall
		u\in \R,~\forall~v\in [-\varepsilon, \varepsilon].
	\end{align}
	For any $a>0$ and Borel function
 $h:\R\to\R_+$,
we define $I_{k,a}=[ka,(k+1)a]$ for
	$k\in\mathbb{Z}$ and
	\begin{align}
		\bar{h}_{a}(u)
		:=\sum_{k\in\mathbb{Z}}
		1_{I_{k,a}}(u)\sup_{u'\in I_{k,a}}f(u'),\quad
		\underline{h}_{a}(u)
		:=
		\sum_{k\in\mathbb{Z}}
		1_{I_{k,a}}(u)\inf_{u'\in I_{k,a}}f(u'),\quad
u\in \R.
	\end{align}
	The function $h$ is called directly Riemann integrable if $\int_{\R}\bar{h}_{a}(u) \mathrm{d}u<\infty$ for any $a>0$ small enough and
	\begin{align}
			\lim_{a\to 0}
		\int_{\R} \left(\bar{h}_{a}(u)-\underline{h}_{a}(u)\right)\mathrm{d}u=0.
	\end{align}
		Define
	\begin{align}\label{def-varepsilon-domain}
		\bar{h}_{a,\varepsilon}(u):=
\sup_{[u-\varepsilon, u+\varepsilon]}
\bar{h}_{a}(v),
		\quad
		\underline{h}_{a,-\varepsilon}(u):=
\inf_{v\in [u-\varepsilon, u+\varepsilon]}
\bar{h}_{a}(v),
		\quad
u\in \R,
	\end{align}
	then it holds that
	\begin{align}
		\underline{h}_{a,-\varepsilon}\le_{\varepsilon}\underline{h}_{a}\le h\le
		\bar{h}_{a}
		\le_{\varepsilon}\bar{h}_{a,\varepsilon}
\quad \mbox{on } \R.
	\end{align}
		For more details about  directly Riemann integrability, see \cite[Section XI.1]{Feller}.

	The following theorem will play an important role in this paper.
	We refer the reader to \cite[Theorem 1.9]{GX-AIHP} for an analogous result for random walks.

	\begin{theorem}\label{lemma-tau0>t-rho<0}
		Assume that $\xi$ is a L\'{e}vy process satisfying
		 {\bf (H1)}, {\bf (H2)}, {\bf (H3)}
		and $\mathbf{E}_0\left(\xi_1\right)<0$.
		Let $f:\mathbb{R}\to \mathbb{R}_+$
		be a Borel function, which is not 0 almost everywhere on
		$\mathbb{R}_+$,
		such that $f(x)e^{-\lambda_* x}(1+|x|)$ is directly Riemann integrable.
		Then for any $x>0$,
		it holds that
		\begin{align}
			\lim_{t\to\infty} t^{3/2}e^{-\Psi(\lambda_*)t}\mathbf{E}_x\left(f(\xi_t), \tau_0^->t\right)
			=\frac{2R^*(x)
			e^{\lambda_*x}}{\sqrt{2\pi\Psi''(\lambda_*)^3}}\int_{\mathbb{R}_+}f(z)e^{-\lambda_* z}
			\widehat{R}^*(z) \mathrm{d}z,
		\end{align}
		where $\Psi$ is the Laplace exponent of $\xi$.
	\end{theorem}
	\begin{remark}
			Recall that when $\mathbf{E}_0\left(\xi_1\right)<0$ and {\bf (H3)} holds,
		$\xi^{(\lambda_*)}$ 
		is a L\'evy process with Laplace exponent  $\Psi_{\lambda_*}(\lambda)=\Psi(\lambda+\lambda_*)-\Psi(\lambda_*)$ and that $\Psi_{\lambda_*}'(0+)=\Psi'(\lambda_*)=0$.
		Using \eqref{change-measure-levy}, we get that
		\begin{align}\label{use-change-measure}
			\mathbf{E}_x\left(f(\xi_t), \tau_0^->t\right)
			=e^{\Psi(\lambda_*)t+\lambda_* x}\mathbf{E}_x^{\lambda_*}\left(f(\xi_t)e^{-\lambda_* \xi_t},\tau_0^->t\right).
		\end{align}
		Therefore, to get the assertion of Theorem \ref{lemma-tau0>t-rho<0}, we only need to consider the asymptotic behavior of
		$$\mathbf{E}_x^{\lambda_*}\left(f(\xi_t)e^{-\lambda_* \xi_t},\tau_0^->t\right),\quad t\to \infty.$$
	\end{remark}

	\begin{theorem}\label{thm-tau-rho<0}
		Assume that $\xi$ is a L\'{e}vy process satisfying
		{\bf (H1)}, {\bf (H2)}, {\bf (H3)}
         	and $\mathbf{E}_0\left(\xi_1\right)<0$. Let
		$h:\mathbb{R}\to \mathbb{R}_+$
		 be a Borel function, which is not 0 almost everywhere on
		$\mathbb{R}_+$,
		such that $h(x)(1+|x|)$ is directly Riemann integrable.
		Then for any $x>0$, it holds that
		\begin{align}\label{limit-tau-rho<0}
			\lim_{t\to\infty}t^{3/2}\mathbf{E}_x^{\lambda_*}\left(h(\xi_t)1_{\{ \tau_{0}^->t\}}\right)
			=\frac{2R^*(x)}{\sqrt{2\pi\Psi''(\lambda_*)^3} }\int_{\mathbb{R}_+}h(z) \widehat{R}^*(z)\mathrm{d}z,
		\end{align}
		where $\Psi$ is the Laplace exponent of $\xi$.
	\end{theorem}
	\begin{remark}
		The difference between the asymptotic behavior in Theorem \ref{thm-tau-rho<0} (with $t^{-3/2}$ decay) and that in Lemma \ref{lemma-tau0>t-rho=0} (with $t^{-1/2}$ decay) arises from the fact that Theorem \ref{thm-tau-rho<0}
				is a conditioned limit theorem
		for the process $\left(\xi_t\right)_{t\ge 0}$ itself, whereas
				Lemma \ref{lemma-tau0>t-rho=0}
				is a  conditioned limit result
		for the normalized process $\left(\frac{\xi_t}{\sigma \sqrt{t}}\right)_{t\ge 0}$.
	\end{remark}
	
	\textbf{Proof of Theorem \ref{lemma-tau0>t-rho<0}:}
	Taking $h(x)=f(x)e^{-\lambda_* x}$
		in Theorem \ref{thm-tau-rho<0} and using \eqref{use-change-measure},
		we immediately get the conclusion of Theorem \ref{lemma-tau0>t-rho<0}.
	\qed

	In the next four lemmas, we provide some upper and lower bounds for $\mathbf{E}_x^{\lambda_*}\left(h(\xi_t)1_{\{ \tau_{0}^->t\}}\right)$.
	In the remainder of this section,  $\varepsilon_0$  will be the constant in Lemma \ref{lemma-effective-conditioned-integral-limit}. Recall that $\rho(\cdot)$ stands for the Rayleigh density.
	
	\begin{lemma}\label{lemma-rough-upper-bound-limsup}
		Assume that $\xi$
		is a L\'{e}vy process satisfying
	 {\bf (H1)}, {\bf (H2)}, {\bf (H3)}
     		and $\mathbf{E}_0\left(\xi_1\right)<0$.
		Then one can find a constant
		$C_5>0$ with the property that for any $\varepsilon \in(0,\varepsilon_0)$
		there exist positive constants
		$T_1(\varepsilon)$ and $C_6(\varepsilon)$
		such that for any $x>0$, $t>T_1(\varepsilon)$
		and any integrable functions
		$h,H:\R \to \R_+$
		satisfying $h \le_{\varepsilon}H$,
		\begin{align}
			\mathbf{E}_x^{\lambda_*}\left(h(\xi_t)1_{\{ \tau_{0}^->t\}}\right)
     \le&
  \frac{2(1+C_5\varepsilon)R^*(x)
			}{\sqrt{2\pi}\Psi''(\lambda_*) t}\int_{\mathbb{R}_+}H(w)
			\rho
			\left(\frac{w}{ \sqrt{\Psi''(\lambda_*) t}}\right) \mathrm{d}w\\
			&+\frac{2
			C_5
			\sqrt{\varepsilon}R^*(x)}{\Psi''(\lambda_*) t}
			\int_{-\varepsilon}^{\infty}
			H(w)\phi\left( \frac{w}{\sqrt{\Psi''(\lambda_*) t}}\right)
			\mathrm{d}w\\
			&+
			C_6(\varepsilon)
			(1+x)
			\|H1_{[-\varepsilon,\infty)}\|_1
			\left(\frac{1}{t^{1+\varepsilon}}
			+\frac{1}{t^{1+\delta/2}}\right),
		\end{align}
		where $\Psi$ is the Laplace exponent
			of $\xi$.
	\end{lemma}
	\begin{proof}
Fix $\varepsilon \in(0,\varepsilon_0)$ and let
$h,H: \R \to \R_+$
be integrable functions  satisfying $h \le_{\varepsilon}H$.
		Fix $t\ge 1$ and set $m=[ \varepsilon t]$.
		By the Markov property,
		\begin{align}\label{decomposition-expectation}
			\mathbf{E}_x^{\lambda_*}\!\left(h(\xi_t)1_{\{ \tau_{0}^->t\}}\right)
			=\int_{\R_+}\!\!\mathbf{E}_y^{\lambda_*}\!\left(h(\xi_m)1_{\{\tau_{0}^->m\}}\right)\!
			\mathbf{P}_x^{\lambda_*}\!\left(\xi_{t-m}\in \mathrm{d}y, \tau_0^->t-m\right).
		\end{align}
			Define a random walk $(S_n)_{n\ge 0}$ by
		$S_n:=\xi_n$,
		$n\in \mathbb{N}$.
		Since $h1_{[0,\infty)}\le_{\varepsilon}H1_{[-\varepsilon,\infty)}$,
		it follows from \cite[Theorem 2.7]{GX-AIHP} that
		there exist constants
		$c_1$ (independent of $\varepsilon)$ and $c_2(\varepsilon)$
		such that for any $n\ge 1$,
		\begin{align}\label{local-limit-RW}
			&\mathbf{E}_x^{\lambda_*}
			    \left(h(S_n)1_{\{S_n\ge 0\}}\right)
			-\frac{
				1+c_1\varepsilon
				}{ \sqrt{\Psi''(\lambda_*)n}}
			\int_{\R}
			H(z)
			1_{\{z\ge -\varepsilon\}}
			\phi\left(\frac{z-x}{ \sqrt{\Psi''(\lambda_*)n}}\right)\mathrm{d}z\\
			&\le \frac{
				c_2(\varepsilon)
				}{n^{(1+\delta)/2}}
			\|H1_{[-\varepsilon,\infty)}\|_1.
		\end{align}
		Thus, for any $y>0$,
		\begin{align}
			&\mathbf{E}_y^{\lambda_*}\left(h(\xi_m)1_{\{\tau_{0}^->m\}}\right)
			\le \mathbf{E}_y^{\lambda_*}\left(h(S_m)1_{\{S_m \ge 0\}}\right)\\
			&\le \frac{
				1+c_1\varepsilon
				}{ \sqrt{\Psi''(\lambda_*) m}}
			\int_{\mathbb{R}}H(z)1_{\{z\ge -\varepsilon\}}
			\phi\left(\frac{z-y}{ \sqrt{\Psi''(\lambda_*) m}}\right)\mathrm{d} z
			+\frac{
				c_2(\varepsilon)}
				{m^{(1+\delta)/2}}
			\|H1_{[-\varepsilon,\infty)}\|_1.
		\end{align}
		Plugging this into \eqref{decomposition-expectation} yields that
		\begin{align}\label{upper-bound-expectation-decomposition}
			\mathbf{E}_x^{\lambda_*}\left(h(\xi_t)1_{\{ \tau_{0}^->t\}}\right)
			&\le \int_{\R_+}\left(
				\frac{
					1+c_1\varepsilon
					}{ \sqrt{\Psi''(\lambda_*) m}}
			\int_{\mathbb{R}}H(z)1_{\{z\ge -\varepsilon\}}
			\phi\left(\frac{z-y}{ \sqrt{\Psi''(\lambda_*) m}}\right)\mathrm{d} z
			\right)\\
			&\quad \times \mathbf{P}_x^{\lambda_*}\left(\xi_{t-m}\in \mathrm{d}y, \tau_0^->t-m\right)\nonumber\nonumber\\
			&\quad+\int_{\R_+}\frac{
				c_2(\varepsilon)
			\|H1_{[-\varepsilon ,\infty)}\|_1
			}{m^{(1+\delta)/2}}\mathbf{P}_x^{\lambda_*}\left(\xi_{t-m}\in \mathrm{d}y, \tau_0^->t-m\right)\\
&=:A_1(x)+A_2(x).\nonumber
		\end{align}
		By the definition of $\tau_0^-$, we have
		\begin{align}\label{upper-bound-tua}
			\mathbf{P}_x^{\lambda_*}\left(\tau_0^->s\right)
			=\mathbf{P}_x^{\lambda_*}\left(\inf_{l\le s}\xi_l>0\right)
			\le \mathbf{P}_x^{\lambda_*}\left(\inf_{j\le [s]}S_j>0\right)
			\le c_3 \frac{1+x}{\sqrt{s}},
		\end{align}
for some positive constant $c_3$ (independent of $\varepsilon$),
where in the last inequality we used  \cite[(2.7)]{Aidekon13}.
		Therefore,
		by \eqref{upper-bound-tua}
		and the definition of $m$,
		there exists a positive constant
		$c_4(\varepsilon)$ such that
		\begin{align}\label{upper-bound-A2}
			A_2(x)=\frac{
				c_2(\varepsilon)
			\|H1_{[-\varepsilon ,\infty)}\|_1
			}{m^{(1+\delta)/2}}
			\mathbf{P}_x^{\lambda_*}
			\left(\tau_0^->t-m\right)
			\le  \frac{c_4(\varepsilon)
				(1+x)}{t^{1+\delta/2}}
			\|H1_{[-\varepsilon ,\infty)}\|_1.
		\end{align}
		Now, by a change of variables, we get
		\begin{align*}
			A_1(x)
			&=\int_{\R_+}\left(\frac{
				1+c_1\varepsilon
				}{\sqrt{\Psi''(\lambda_*) m}}
			\int_{\R}H(z)1_{\{z\ge -\varepsilon\}}
			\phi\left(\frac{z- \sqrt{\Psi''(\lambda_*)(t-m)}u}{ \sqrt{\Psi''(\lambda_*) m}}\right)\mathrm{d} z\right)\\
			&\quad \times
			\mathbf{P}_x^{\lambda_*}
			\left(\frac{\xi_{t-m}}{ \sqrt{\Psi''(\lambda_*)(t-m)}}\in \mathrm{d}u, \tau_0^->t-m\right)\\
			&=\int_{\R_+}\varphi_t(u)
			\mathbf{P}_x^{\lambda_*}
			\left(\frac{\xi_{t-m}}{ \sqrt{\Psi''(\lambda_*)(t-m)}}\in \mathrm{d}u, \tau_0^->t-m\right),
		\end{align*}
		where the function $\varphi_t$ is defined by
		\begin{align}
			&\varphi_t(u)
			:=\frac{
				1+c_1\varepsilon
				}{ \sqrt{\Psi''(\lambda_*) m}}
			\int_{\R}H(z)1_{\{z\ge -\varepsilon\}}
			\phi\left(\frac{z- \sqrt{\Psi''(\lambda_*)(t-m)}u}{\sqrt{\Psi''(\lambda_*)m}}\right)\mathrm{d} z\\
			&=
			(1+c_1\varepsilon)
			\sqrt{\frac{t-m}{m}}
			\int_{\R}
			H(\sqrt{\Psi''(\lambda_*)(t-m)}w)
			1_{\{\sqrt{\Psi''(\lambda_*)(t-m)}w\ge -\varepsilon\}}
			 \phi\left(\frac{w-u}{\sqrt{\frac{m}{t-m}}}\right) \mathrm{d}w.
		\end{align}
		Using integration by parts, we get that for any $x\in \R_+$,
		\begin{align}\label{expression-A1}
	A_1(x)\le
\int_{\R_+}\varphi_t'(u)
			\mathbf{P}_x^{\lambda_*}
			\left(\frac{\xi_{t-m}}{ \sqrt{\Psi''(\lambda_*)(t-m)}}>u, \tau_0^->t-m\right)\mathrm{d}u.
		\end{align}
		It follows from Lemma \ref{lemma-effective-conditioned-integral-limit} that
		for $t-m>T_0(\varepsilon)$,
		\begin{align}
			&\Big|\mathbf{P}_x^{\lambda_*}
			\left(\frac{\xi_{t-m}}{ \sqrt{\Psi''(\lambda_*)(t-m)}}>u, \tau_0^->t-m\right)
			-\frac{2R^*(x)
			}{\sqrt{2\pi (t-m)\Psi''(\lambda_*)}}
			\int_{u}^{\infty}
			\rho(z)\mathrm{d}z\Big|\\
			&\le \frac{
				C_4(\varepsilon)
				(1+x)}{(t-m)^{\frac{1}{2}+\varepsilon}},
		\end{align}
		which together with \eqref{expression-A1} implies that
		there exists a constant $c_5(\varepsilon)$
		such that
		for $t-m>T_0(\varepsilon)$,
		\begin{align}\label{upper-bound-A1}
			A_1(x)
			-\frac{2R^*(x)}{\sqrt{2\pi (t-m)\Psi''(\lambda_*)}}\int_{\R_+}\!\varphi_t'(u)e^{-\frac{u^2}{2}} \mathrm{d}u
			\le \frac{
				c_5(\varepsilon)
				(1+x)}{(t-m)^{\frac{1}{2}+\varepsilon}}
			\int_{\R_+}\!|\varphi_t'(u)|\mathrm{d}u.
		\end{align}
		By the definition of $\varphi_t$ and a change of variables, we get that
		\begin{align*}
			&\int_{\R_+}|\varphi_t'(u)|\mathrm{d}u\\
			&\le
			(1+c_1\varepsilon)
			\int_{\R_+}\!
			\int_{\R}
			\frac{t-m}{m} H(\sqrt{\Psi''(\lambda_*)(t-m)}w)
			1_{\{\sqrt{\Psi''(\lambda_*)(t-m)}w\ge -\varepsilon\}}
			\Big|\phi'\left(\frac{w-u}{\sqrt{\frac{m}{t-m}}}\right)\Big|\mathrm{d}w\mathrm{d}u\\
			&=
			(1+c_1\varepsilon)
			\int_{\R_+}
			\int_{\R}
			H( \sqrt{\Psi''(\lambda_*)m}u)
			1_{\{\sqrt{\Psi''(\lambda_*)m}u\ge -\varepsilon\}}
			|\phi'(u-y)|\mathrm{d}u \mathrm{d}y\\
			&=
			(1+c_1\varepsilon)
			\int_{\R}H( \sqrt{\Psi''(\lambda_*)m}u)
			1_{\{\sqrt{\Psi''(\lambda_*)m}u\ge -\varepsilon\}}\mathrm{d}u
			\int_{\R_+}|\phi'(u-y)|\mathrm{d}y.
		\end{align*}
		Since there exists a constant
		$c_6>0$ such that $\int_{\R_+}|\phi'(u-y)|\mathrm{d}y\le c_6$,
		a change of variables yields that
		\begin{align}\label{upper-bound-integral-varphi}
			\int_{\R_+}|\varphi_t'(u)|\mathrm{d}u
			\le c_6(1+c_1\varepsilon)
			\frac{\|H1_{[-\varepsilon,\infty)}\|_1 }{\sqrt{\Psi''(\lambda_*) m}}.
		\end{align}
		Using integration by parts, we get
		\begin{align*}
			&\int_{\R_+}\varphi_t'(y)e^{-\frac{y^2}{2}} \mathrm{d}y
			=\int_{\R_+}\varphi_t(y)
			\rho(y)
			\mathrm{d}y\\
			&=
			(1+c_1\varepsilon)
			 \int_{\R_+}\!
						\int_{\R}\!
			\sqrt{\frac{t-m}{m}} H( \sqrt{\Psi''(\lambda_*)(t-m)}w)
			1_{\{\sqrt{\Psi''(\lambda_*)(t-m)}w\ge -\varepsilon\}}
			\phi\left(\frac{w-y}{\sqrt{\frac{m}{t-m}}}\right)\! \mathrm{d}w \rho(y)
			\mathrm{d}y\\
			&=
			(1+c_1\varepsilon)
			\int_{\R_+} \! \int_{\R}\!
			\sqrt{\frac{t}{m}}H( \sqrt{\Psi''(\lambda_*) t} u)
			1_{\{\sqrt{\Psi''(\lambda_*) t} u\ge -\varepsilon\}}
			\phi\left(\!\frac{u-z}{\sqrt{\frac{m}{t}}}\!\right)\mathrm{d}u 	\rho
			\left( \!\sqrt{\frac{t}{t-m}}z\!\right)\sqrt{\frac{t}{t-m}}\mathrm{d}z.
		\end{align*}
		According to \cite[Lemma 3.3]{GX}, for any $v\in (0,\frac{1}{2}]$ and $s\ge 0$, it holds that
		\begin{align}\label{two-side-bound-1}
			\sqrt{1-v}
			\rho(s)
			\le \phi_v*
			\rho_{1-v}(s)
			\le \sqrt{1-v}
			\rho(s)
			+\sqrt{v}e^{-\frac{s^2}{2v}}.
		\end{align}
		Letting $v=\frac{m}{t}$, we get
		\begin{align}
			&\int_{\R_+}\varphi_t'(y)e^{-\frac{y^2}{2}} \mathrm{d}y
			=(1+c_1\varepsilon)
			\int_{\R}
			H( \sqrt{\Psi''(\lambda_*)t}u)
			1_{\{\sqrt{\Psi''(\lambda_*) t} u\ge -\varepsilon\}}
			\phi_{\frac{m}{t}}*
			\rho_{\frac{t-m}{t}}(u)
			\mathrm{d}u\\
			&=
			\frac{
			(1+c_1\varepsilon)
				}{ \sqrt{\Psi''(\lambda_*)t}}
                \int^\infty_{-\varepsilon} H(w)
			 \phi_{\frac{m}{t}}*
			\rho_{\frac{t-m}{t}}
			\left(\frac{w}{\sqrt{\Psi''(\lambda_*)t}}\right)\mathrm{d}w\\
			&\le
			\frac{
			(1+c_1\varepsilon)
				}{\sqrt{\Psi''(\lambda_*)t}}
           \int^\infty_{-\varepsilon} H(w)
\left(\sqrt{\frac{t-m}{t}}
			\rho
			\left(\frac{w}{ \sqrt{\Psi''(\lambda_*)t}}\right)+\sqrt{\frac{m}{t}}e^{-\frac{w^2}{2\Psi''(\lambda_*)t}}\right) \mathrm{d}w.
		\end{align}
		Combining this with
		\eqref{upper-bound-expectation-decomposition}, \eqref{upper-bound-A2}, \eqref{upper-bound-A1}
		 \eqref{upper-bound-integral-varphi},
and using the fact that $\rho(z)=0$ for $z\le 0$ and noticing  that $m=[\varepsilon t]$,
 we get that
there exist positive constants
$c_7$ (independent of $\varepsilon$) and  $t_1(\varepsilon)$, $c_8(\varepsilon)$ such that for $t>t_1(\varepsilon)$,
		\begin{align}\label{upper-bound-expectation}
			&\mathbf{E}_x^{\lambda_*}
			\left(h(\xi_t)1_{\{ \tau_{0}^->t\}}\right)
			\le \frac{2
			(1+c_7 \varepsilon)
			R^*(x)
			}{\sqrt{2\pi}\Psi''(\lambda_*) t}\int_{\mathbb{R}_+}H(w)
			\rho
			\left(\frac{w}{ \sqrt{\Psi''(\lambda_*) t}}\right) \mathrm{d}w\\
			&+c_7\sqrt{\varepsilon}
			\frac{2R^*(x)}{\Psi''(\lambda_*) t}
			\int_{-\varepsilon}^{\infty}\!
			H(w)\phi\left(\!\frac{w}{\sqrt{\Psi''(\lambda_*) t}}\!\right)
			\mathrm{d}w
			+c_8(\varepsilon)(1+x)
			\|H1_{[-\varepsilon, \infty)}\|_1
			\left(\frac{1}{t^{1+\varepsilon}}
			+\frac{1}{t^{1+\delta/2}}\right).
		\end{align}
		The proof is complete.
	\end{proof}
	
	\begin{lemma}\label{lemma-upper-bound-limsup}
		Assume that $\xi$
		is a L\'{e}vy process satisfying
		{\bf (H1)}, {\bf (H2)}, {\bf (H3)}
		and $\mathbf{E}_0\left(\xi_1\right)<0$.
		Then one can find a constant
		$C_7>0$ with the property that for any $\varepsilon \in(0,\varepsilon_0)$
		there exist positive constants $T_2(\varepsilon)$
		and $C_8(\varepsilon)$ such that for any $x>0$, $t>T_2(\varepsilon)$
		and any Borel functions
		$h,H:{\R}\to \R_+$
		satisfying $h \le_{\varepsilon}H$ and
		   $\int_{\mathbb{R}_+}H(z-\varepsilon)(1+z)\mathrm{d}z<\infty$,
		\begin{align}
			\mathbf{E}_x^{\lambda_*}\left(h(\xi_t)1_{\{ \tau_{0}^->t\}}\right)
			&\le \left(1+
			C_7t^{-1/2}+C_7\sqrt{\varepsilon}\right) \frac{2 R^*(x)
			}{\sqrt{2\pi\Psi''(\lambda_*)^3} t^{3/2}}\int_{\mathbb{R}_+}
			H(z-\varepsilon)
			\widehat{R}^*(z)
			\mathrm{d}z\\
			&\quad+\frac{
			C_8(\varepsilon)
			R^*(x)
			}{\sqrt{2\pi\Psi''(\lambda_*)^3} t^{3/2+\varepsilon}}\int_{\mathbb{R}_+}
			H(z-\varepsilon)
			(1+z)\mathrm{d}z\\
			&\quad+\frac{
			C_8(\varepsilon)
			(1+x)}{\sqrt{t}}\left(
			\frac{1}{t^{1+\varepsilon}}
			+\frac{1}{t^{1+\delta/2}}\right)\int_{\mathbb{R}_+}
			H(z-\varepsilon)
			(1+z)
			\mathrm{d}z.
		\end{align}
	\end{lemma}
	\begin{proof}
 		Fix $\varepsilon \in(0,\varepsilon_0)$ and let
	$h,H:\R \to \R_+$
satisfying $h \le_{\varepsilon}H$.
	For any $z\in \mathbb{R}$, we define
	\begin{align}\label{def-Hm}
		H_m(z)
		:=\mathbf{E}_0^{\lambda^*}\left(H(\xi_m+z)1_{\{\tau_{-z-\varepsilon}^->m\}}\right)
		=\mathbf{E}_z^{\lambda^*}\left(H(\xi_m)1_{\{\tau_{-\varepsilon}^->m\}}\right).
	\end{align}
		Fix $t\ge 2$
		and set $m=[t/2]$.
			For any $y>0$,
		we have
		\begin{align}\label{def-I-H}
			I_m(y)&:=
			\mathbf{E}_y^{\lambda_*}\left(h(\xi_m)1_{\{\tau_{0}^->m\}}\right)\\
			&\le
			\mathbf{E}_{y}^{\lambda_*}\left(H(\xi_m+v)
			1_{\{\tau_{-v-\varepsilon}^->m\}}
			\right)
			=H_m(y+v),\quad |v|\le \varepsilon.
		\end{align}
		Consequently, $I_m
		\le_{\varepsilon}H_m$.
		By the Markov property,
		\begin{align}
			&\mathbf{E}_x^{\lambda_*}\left(h(\xi_t)1_{\{ \tau_{0}^->t\}}\right)
			=\int_{\R_+}\mathbf{E}_y^{\lambda_*}\left(h(\xi_m) 1_{\{\tau_{0}^->m\}}\right)
			\mathbf{P}_x^{\lambda_*}\left(\xi_{t-m}\in \mathrm{d}y, \tau_0^->t-m\right)\\
			&=\int_{\R_+}I_m(y)
			\mathbf{P}_x^{\lambda_*}\left(\xi_{t-m}\in \mathrm{d}y, \tau_0^->t-m\right)
			=\mathbf{E}_x^{\lambda_*}\left(I_m(\xi_{t-m}),\tau_0^->t-m\right).
		\end{align}
		Now applying Lemma \ref{lemma-rough-upper-bound-limsup}
		with $h=I_m$, we get that for $t-m>T_1(\varepsilon)$,
\begin{equation}\label{decomposition-expression}
\mathbf{E}_x^{\lambda_*}\left(h(\xi_t)1_{\{ \tau_{0}^->t\}}\right)\le J_1+J_2+J_3,
\end{equation}
where
\begin{align*}
J_1&:=\frac{2
(1+C_5\varepsilon)
R^*(x)}{\sqrt{2\pi}\Psi''(\lambda_*)(t-m)}\int_{\mathbb{R}_+}H_m(w)
\rho
\left(\frac{w}{ \sqrt{\Psi''(\lambda_*)(t-m)}}\right) \mathrm{d}w,\\
J_2&:=C_5
	\sqrt{\varepsilon}
\frac{2R^*(x)}{\Psi''(\lambda_*)(t-m)}
\int_{-\varepsilon}^{\infty}
H_m(w)\phi\left(\frac{w}{\Psi''(\lambda_*)(t-m)}\right)
\mathrm{d}w,\\
J_3&:=C_6(\varepsilon)
\left(\frac{1}{(t-m)^{1+\varepsilon}}+\frac{1}{(t-m)^{1+\delta/2}}\right)
(1+x)
\|H_m1_{[-\varepsilon,\infty)}\|_1.
\end{align*}
We will deal with the upper bounds of $J_i$ separately.
We first deal with $J_1$.
Note that
		\begin{align}\label{equality-J1}
			J_1&=\frac{2
				(1+C_5\varepsilon)
				R^*(x)
				}{\sqrt{2\pi}\Psi''(\lambda_*) (t-m)}\int_{\mathbb{R}_+}
				\mathbf{E}_w^{\lambda_*}\left(H(\xi_m)
				1_{\{\tau_{-\varepsilon}^->m\}}
				\right)
				\rho
				\left(\frac{w}{ \sqrt{\Psi''(\lambda_*) (t-m)}}\right) \mathrm{d}w\\
				&=\frac{2
				(1+C_5\varepsilon)
				R^*(x)
				}{\sqrt{2\pi}\Psi''(\lambda_*) (t-m)}\!\int_{\mathbb{R}}\!
				\mathbf{E}_{w+\varepsilon}^{\lambda_*}\!\left(\!H(\xi_m-\varepsilon)
				1_{\{\tau_0^->m\}}
				\!\right)1_{\{w+\varepsilon\ge 0\}}
				\rho\!
				\left(\!\frac{w}{ \sqrt{\Psi''(\lambda_*) (t-m)}}\!\right) \!\mathrm{d}w\\
				&=\frac{2
				(1+C_5\varepsilon)
				R^*(x)
				}{\sqrt{2\pi}\Psi''(\lambda_*) (t-m)}\int_{\mathbb{R_+}}
				\mathbf{E}_{w}^{\lambda_*}\left(H(\xi_m-\varepsilon)
				 1_{\{\tau_0^->m\}}
				\right)
				\rho
				\left(\frac{w-\varepsilon}{ \sqrt{\Psi''(\lambda_*) (t-m)}}\right) \mathrm{d}w\\
				&=\frac{2
				(1+C_5\varepsilon)
				R^*(x)
				}{\sqrt{2\pi}\Psi''(\lambda_*) (t-m)}\int_{\mathbb{R_+}}H(w-\varepsilon)
				\mathbf{E}_{w}^{\lambda_*}\left(\rho
				\left(\frac{\widehat{\xi}_m-\varepsilon}{ \sqrt{\Psi''(\lambda_*) (t-m)}}\right)
				1_{\{\widehat{\tau}_0^->m\}}
				\right)
				 \mathrm{d}w,
		\end{align}
where in the last equality we used \eqref{duality-formula-1}.
		Using integration by parts,
				we get for any $z\in \mathbb{R}_+$,
		\begin{align}\label{equivalent-expression}
			&\mathbf{E}_{z}^{\lambda_*}\left(\rho
			\left(
				\frac{\widehat{\xi}_m-\varepsilon}
				{\sqrt{\Psi(\lambda_*)(t-m)}}\right)
			1_{\{\widehat{\tau}_0^->m\}}
			\right)\\
			=&\int_{\mathbb{R}_+}\rho'
			(u) \mathbf{P}_{z}^{\lambda_*}\left(
			 \frac{\widehat{\xi}_m-\varepsilon}
			 {\sqrt{\Psi(\lambda_*)(t-m)}}>u,
			\widehat{\tau}_0^->m
			\right)\mathrm{d}u\nonumber\\
			=&\int_{\mathbb{R}_+}\rho'
			(u) \mathbf{P}_{z}^{\lambda_*}\left(
				 \frac{\widehat{\xi}_m}
				{\sqrt{\Psi''(\lambda_*)m}}>
			\frac{u \sqrt{\Psi(\lambda_*)(t-m)}+\varepsilon}
			{ \sqrt{\Psi''(\lambda_*)m}},
			\widehat{\tau}_0^->m
			\right)\mathrm{d}u.
		\end{align}
		Set
		$$u_{m,\varepsilon}:=\frac{u \sqrt{\Psi(\lambda_*)(t-m)}+\varepsilon}{ \sqrt{\Psi''(\lambda_*)m}}.$$
		Applying Lemma \ref{lemma-effective-conditioned-integral-limit} to $\widehat{\xi}$,
		 we get that for $m>T_0(\varepsilon)$,
		\begin{align}
			\Big|
			\mathbf{P}_{z}^{\lambda_*}\left(\frac{\widehat{\xi}_m}{\sqrt{\Psi''(\lambda_*)m}}
			>  u_{m,\varepsilon},
			\widehat{\tau}_0^->m
			\right)
			-\frac{2\widehat{R}^*(z)}{\sqrt{2\pi m\Psi''(\lambda_*)}}
			\int_{u_{m,\varepsilon}}^{\infty}
			\rho(y)
			\mathrm{d}y\Big|
			\le \frac{
				C_4(\varepsilon)
				(1+z)}{m^{1/2+\varepsilon}}.
		\end{align}
		Substituting this into \eqref{equivalent-expression} and using the fact that
		$\int_{\mathbb{R}_+}\rho'(u) \mathrm{d}u\le c_1$ for some $c_1>0$,
		we get that
		\begin{align}\label{two-side-bound-expectation}
			&\Big|
			\mathbf{E}_{z}^{\lambda_*}\left(
			\rho
			\left(\frac{\widehat{\xi}_m-\varepsilon}
			{\sqrt{\Psi''(\lambda_*)(t-m)}}\right)1_{\{
			\widehat{\tau}_0^->m
			\}}\right)
			-\frac{2\widehat{R}^*(z)
			}{\sqrt{2\pi  m\Psi''(\lambda_*)}} \int_{\mathbb{R}_+}
			\rho'(u)
			e^{-\frac{u_{m,\varepsilon}^2}{2}}
				\mathrm{d}u\Big|\\
			&
				\le \frac{
					C_4(\varepsilon)
					(1+z)}{m^{1/2+\varepsilon}}\int_{\mathbb{R}_+}\rho'(u)
				\mathrm{d}u
			\le c_1\frac{C_4(\varepsilon)(1+z)}{m^{1/2+\varepsilon}}.
		\end{align}
		Using integration by parts again and the boundedness of $\rho'$, we get that there
		exists a positive constant
$c_2$ (independent of $\varepsilon$) such that
		\begin{align}\label{upper-bound-fisrt-part-expectation}
			&\int_{\mathbb{R}_+}
			\rho'(u)
			e^{-\frac{u_{m,\varepsilon}^2
				}{2}}
				\mathrm{d}u=
			 \sqrt{\frac{t-m}{m}}\int_{\mathbb{R}_+}
			 \rho(u)\rho
			(u_{m,\varepsilon})
				\mathrm{d}u\\
			&\le \sqrt{\frac{t-m}{m}}\int_{\mathbb{R}_+}
			\rho(u)\rho
			\left(\frac{u\sqrt{t-m}}{\sqrt{m}}\right)\mathrm{d}u+
			 \frac{c_2\varepsilon}{\sqrt{t}},
		\end{align}
where in the last inequality we used the mean value theorem.
		By a change of variables, we see that
		\begin{align}\label{upper-bound-second-part-expectation}
			&\sqrt{\frac{t-m}{m}}\int_{\mathbb{R}_+}
			\rho(u)\rho
			\left(\frac{u\sqrt{t-m}}{\sqrt{m}}\right)\mathrm{d}u
			=\frac{1}{\sqrt{m}}\int_{\mathbb{R}_+}
			\rho
			\left(\frac{y}{\sqrt{t-m}}\right)
			 \rho
			 \left(\frac{y}{\sqrt{m}}\right)\mathrm{d}y\\
			&=\frac{1}{\sqrt{m}}\int_{\mathbb{R}_+}\!\frac{y^2}{\sqrt{(t-m)m}}e^{-\frac{ty^2
				}{2m(t-m)}}\mathrm{d}y
			=\frac{\sqrt{m}(t-m)}{t^{3/2}}\int_{\mathbb{R}_+}\!y^2 e^{-\frac{y^2}{2}}\mathrm{d}y\\
			 &=\frac{\sqrt{2\pi m}(t-m)}{2t^{3/2}}.
		\end{align}
		Combining this with \eqref{two-side-bound-expectation},
		\eqref{upper-bound-fisrt-part-expectation} and
		\eqref{upper-bound-second-part-expectation},
		we get that there exist positive constants $c_4$ (independent of $\varepsilon$)
		and $c_5(\varepsilon)$ such that
		\begin{align}\label{upper-bound-J1a}
			J_1
			&\le
			\left(1+c_4 t^{-\frac{1}{2}}\right)\frac{2 \left(1+c_4\varepsilon\right)R^*(x)}{\sqrt{2\pi\Psi''(\lambda_*)^3} t^{3/2}}
			\int_{\mathbb{R}_+}
			H(z-\varepsilon)
			\widehat{R}^*(z)\mathrm{d}z\\
			&\quad+\frac{
			c_5(\varepsilon)
			R^*(x)}{\sqrt{2\pi\Psi''(\lambda_*)^3} t^{3/2+\varepsilon}}
			\int_{\mathbb{R}_+}
			H(z-\varepsilon)(1+z)
			\mathrm{d}z.
		\end{align}
		Next, we deal with $J_2$. Note that
		 \begin{align*}
			J_2
			&=\frac{2C_5\sqrt{\varepsilon} R^*(x)
			}{\Psi''(\lambda_*) (t-m)}
			\int_{-\varepsilon}^{\infty}
			\mathbf{E}_w^{\lambda_*}\left(H(\xi_m)
			1_{\{\tau_{-\varepsilon}^->m\}}
			\right)
			\phi\left(\frac{w}{\sqrt{\Psi''(\lambda_*) (t-m)}}\right)\mathrm{d}w\nonumber\\
			&=
			\frac{2C_5\sqrt{\varepsilon} R^*(x)
			}{\Psi''(\lambda_*) (t-m)}
			\int_{-\varepsilon}^{\infty}
			\mathbf{E}_{w+\varepsilon}^{\lambda_*}\left(H(\xi_m-\varepsilon)
			1_{\{\tau_{0}^->m\}}
			\right)
			\phi\left(\frac{w}{\sqrt{\Psi''(\lambda_*) (t-m)}}\right)\mathrm{d}w\nonumber\\
			&=
			\frac{2C_5\sqrt{\varepsilon} R^*(x)
			}{\Psi''(\lambda_*) (t-m)}
                \int_{\R_+}
			\mathbf{E}_{w}^{\lambda_*}\left(H(\xi_m-\varepsilon)
			1_{\{\tau_{0}^->m\}}
			\right)
			\phi\left(\frac{w-\varepsilon}{\sqrt{\Psi''(\lambda_*) (t-m)}}\right)\mathrm{d}w\nonumber\\
			& =\frac{2C_5\sqrt{\varepsilon} R^*(x)
			}{\Psi''(\lambda_*)(t-m)}\int_{\mathbb{R}_+}
			H(w-\varepsilon)
			\mathbf{E}_w^{\lambda_*}\left(\phi\left(
			\frac{\widehat{\xi}_m-\varepsilon}
			 {\sqrt{\Psi''(\lambda_*) (t-m)}}\right)
			1_{\{\widehat{\tau}_0^->m\}}
			\right)
			\mathrm{d}w,
		\end{align*}
where in the last equality we used \eqref{duality-formula-1}.
Now repeating the argument leading to \eqref{upper-bound-J1a},
we get that there exist positive constants $c_6$ (independent of $\varepsilon$) and $c_7(\varepsilon)$
such that
		\begin{align}\label{upper-bound-J2}
			J_2
			&\le\left(1+c_6\varepsilon t^{-\frac{1}{2}}\right)
			c_6\sqrt{\varepsilon}
			\frac{2R^*(x)
			}{\sqrt{2\pi \Psi''(\lambda_*)^3}  t^{3/2}}\int_{\mathbb{R}_+}
			H(z-\varepsilon)
			\widehat{R}^*(z)
			\mathrm{d}z\\
			&\quad+
			\frac{
			c_7(\varepsilon)
			R^*(x)
			}{\sqrt{2\pi\Psi''(\lambda_*)^3}t^{3/2+\varepsilon}}
			\int_{\mathbb{R}_+}
			H(z-\varepsilon)(1+z)
			\mathrm{d}z.
		\end{align}
		Finally, we deal with $J_3$.
		By the definition of $H_m$ and \eqref{upper-bound-tua}, we have
		\begin{align}\label{equivalent-Hm}
			&\|H_m1_{[-\varepsilon,\infty)}\|_1
			=\int_{\mathbb{R}}
			\mathbf{E}_y^{\lambda_*}\left(H(\xi_m)
			1_{\{\tau_{-\varepsilon}^->m\}}
			\right)
			1_{\{y\ge -\varepsilon\}}
			 \mathrm{d}y\\
			 &=\int_{\R}\mathbf{E}_{y+\varepsilon}^{\lambda_*}\left(H(\xi_m-\varepsilon)
			1_{\{\tau_0^->m\}}
			 \right)1_{\{y\ge -\varepsilon\}}
			  \mathrm{d}y\nonumber\\
& =\int_{\R_+}
\mathbf{E}_{y}^{\lambda_*}\left(H(\xi_m-\varepsilon)
			 1_{\{\tau_0^->m\}}
			 \right)
			  \mathrm{d}y\nonumber\\
			&=\int_{\R_+}H(z-\varepsilon)\mathbf{P}_{z}\left(
				\widehat{\tau}_0^->m
				\right)\mathrm{d}z
			\le c_8\int_{\mathbb{R}_+}H(z-\varepsilon)
			\frac{1+z}{\sqrt{m}}\mathrm{d}z,\nonumber
		\end{align}
		where in the last equality
         we used \eqref{duality-formula-1}
		 and $c_8$ is a positive constant independent of $\varepsilon$.
		Since $m=[t/2]$,
		there exists a positive constant
		$c_9(\varepsilon)$ such that
		\begin{align}\label{upper-bound-J3a}
			\qquad J_3
			\le \frac{
				c_9(\varepsilon)
				(1+x)}{\sqrt{t}}
			\left(\frac{1}{t^{1+\varepsilon}}
			+\frac{1}{t^{1+\delta/2}}\right)\int_{\mathbb{R}_+}(1+z)
			H(z-\varepsilon)
			\mathrm{d}z.
		\end{align}
		Combining \eqref{decomposition-expression}, \eqref{upper-bound-J1a}, \eqref{upper-bound-J2} and \eqref{upper-bound-J3a},
		we complete the proof.
	\end{proof}

	\begin{lemma}\label{lemma-rough-lower-bound-liminf}
		Assume that
		$\xi$ is a L\'{e}vy process satisfying
		 {\bf (H1)}, {\bf (H2)}, {\bf (H3)}
		and $\mathbf{E}_0(\xi_1)<0$.
Then one can find positive constants $C_9$  and $q$
with the property that for any $\varepsilon\in(0,\varepsilon_0)$ there exist
		positive constants
$T_3(\varepsilon)$ and $C_{10}(\varepsilon)$
		such that for any $x>0$, $t>T_3(\varepsilon)$
		and any integrable functions
		$h,H,g:\R\to\R_+$
		satisfying $g\le_{\varepsilon}h\le_{\varepsilon}H$,
		\begin{align*}
			\mathbf{E}_x^{\lambda_*}\left(h(\xi_t) 1_{\{ \tau_{0}^->t\}}\right)
			\ge &\frac{2R^*(x)
			}{\sqrt{2\pi}\Psi''(\lambda_*) t} \int_{\mathbb{R}_+}\left(
				g(w)1_{\{w\ge \varepsilon\}}
				-C_9\varepsilon h(w)\right)
			\rho\left(\frac{w}{ \sqrt{\Psi''(\lambda_*) t}} \right)
			\mathrm{d} w\\
			&-C_9
			 \varepsilon^{1/12}
			\frac{2R^*(x)
			}{\sqrt{2\pi}\Psi''(\lambda_*)t}
			\int_{-\varepsilon}^{\infty}
			H(u)  \phi\left(\frac{w}{ \sqrt{\Psi''(\lambda_*) t}}\right)\mathrm{d}w\\
			&-C_{10}(\varepsilon)(1+x)
			\|H1_{[-\varepsilon,\infty)}\|_1
			\left(\frac{1}{t^{1+\varepsilon}}
			+\frac{1}{t^{1+\delta/2}}
			+\frac{1}{t^{1+q}}
			\right),
		\end{align*}
		where $\Psi$ is the Laplace exponent of $\xi$.
	\end{lemma}
	\begin{proof}
		Fix $\varepsilon \in(0,\varepsilon_0)$ and let
		$h,H,g:\R\to\R_+$
		be integrable functions satisfying $g\le_{\varepsilon}h\le_{\varepsilon}H$.
		Then $g1_{[\varepsilon,\infty)}\le_{\varepsilon}h1_{[0,\infty)}\le_{\varepsilon}H1_{[-\varepsilon,\infty)}$.
		Fix  $t\ge 1$ and set $m=[\varepsilon t]$.
		By the Markov property,
		\begin{align*}
			\mathbf{E}_x^{\lambda_*}\left(h(\xi_t)1_{\{ \tau_{0}^->t\}}\right)
			&=\int_{\R_+}\mathbf{E}_y^{\lambda_*}\left(h(\xi_m)1_{\{ \tau_{0}^->m\}}\right)
			\mathbf{P}_x^{\lambda_*}\left(\xi_{t-m}\in \mathrm{d}y, \tau_0^->t-m\right)\\
			&=\int_{\R_+}\mathbf{E}_y^{\lambda_*}
			\left(h(\xi_m)1_{\{\xi_m\ge 0\}}\right)
			\mathbf{P}_x^{\lambda_*}\left(\xi_{t-m}\in \mathrm{d}y, \tau_0^->t-m\right)\\
			&\quad-\int_{\R_+}\mathbf{E}_y^{\lambda_*}
			\left(h(\xi_m)1_{\{\xi_m\ge 0\}}1_{\{\tau_{0}^-\le m\}}\right)
			\mathbf{P}_x^{\lambda_*}\left(\xi_{t-m}\in \mathrm{d}y, \tau_0^->t-m\right)\\
			&=:I_1(t)-I_2(t)
			=:I_1(t)-I_2^1(t)-I_2^2(t),
		\end{align*}
		where
		\begin{align*}
			I_1(t):=\int_{\R_+}\mathbf{E}_y^{\lambda_*}
			\left(h(\xi_m)1_{\{\xi_m\ge 0\}}\right)
			\mathbf{P}_x^{\lambda_*}\left(\xi_{t-m}\in \mathrm{d}y, \tau_0^->t-m\right),
		\end{align*}
		\begin{align*}
			I_2^1(t)
			:=\int_{0}^{\varepsilon^{1/6}\sqrt{[t]}}
			\mathbf{E}_y^{\lambda_*}\left(h(\xi_m) 1_{\{\xi_m\ge 0\}}1_{\{\tau_{0}^-\le m\}}\right)
			\mathbf{P}_x^{\lambda_*}\left(\xi_{t-m}\in \mathrm{d}y, \tau_0^->t-m\right),
		\end{align*}
		\begin{align*}
			I_2^2(t)
			:=\int_{\varepsilon^{1/6}\sqrt{[t]}}^{\infty}
				\mathbf{E}_y^{\lambda_*}\left(h(\xi_m) 1_{\{\xi_m\ge 0\}}1_{\{\tau_{0}^-\le m\}}\right)
			\mathbf{P}_x^{\lambda_*}\left(\xi_{t-m}\in \mathrm{d}y, \tau_0^->t-m\right).
		\end{align*}
The proof of the lemma is divided into the following three steps.

{\bf Step 1.}
In this step, we give a lower bound for $I_1(t)$.
		By \cite[Theorem 2.7]{GX-AIHP},
		there exist positive constants
		 $c_1$  (independent of $\varepsilon$)
		and $c_2(\varepsilon)$
		such that for any $m\ge 1$,
		\begin{align}
			\mathbf{E}_y^{\lambda_*}
			\left(h(\xi_m)1_{\{\xi_m\ge 0\}}\right)
			&\ge \frac{1}{ \sqrt{\Psi''(\lambda_*)m}}
			\int_{\mathbb{R}}
			\left(g(z)
			1_{\{z\ge \varepsilon\}}-c_1
			\varepsilon h(z)
			1_{\{z\ge 0\}}
			\right)\phi\left(\frac{z-y}{\sqrt{\Psi''(\lambda_*) m}}\right)\mathrm{d} z\\
			&\quad-\frac{c_2(\varepsilon)}{m^{(1+\delta)/2}}
			\|h1_{[0,\infty)}\|_1.
		\end{align}
		Note that $\|h1_{[0,\infty)}\|_1\le \|H1_{[0,\infty)}\|_1$.
		Following the analysis of $A_1(x)$ in Lemma \ref{lemma-rough-upper-bound-limsup}
		and
		using the lower bound in \eqref{two-side-bound-1},
		we see that
		there exist positive constants $t_1(\varepsilon)$ and $c_3(\varepsilon)$
		such that for $t>t_1(\varepsilon)$,
		\begin{align*}
			&\frac{1}{ \sqrt{\Psi''(\lambda_*)m}}\int_{\R_+}\left(\int_{\mathbb{R}}
			g(z)1_{\{z\ge \varepsilon\}}\phi\left(\frac{z-y}{\sqrt{\Psi''(\lambda_*) m}}\right)\mathrm{d} z\right)
			\mathbf{P}_x^{\lambda_*}\left(\xi_{t-m}\in \mathrm{d}y, \tau_0^->t-m\right)\\
			&\ge \frac{2
			R^*(x)}{\sqrt{2\pi}\Psi''(\lambda_*) t}
			\int_{\mathbb{R}_+}g(w)1_{\{w\ge \varepsilon\}}
			\rho\left(\frac{w}{ \sqrt{\Psi''(\lambda_*) t}} \right)
			\mathrm{d} w
			    -\frac{c_3(\varepsilon)
					(1+x)
					\|H1_{[0,\infty)}\|_1
					}{t^{1+\varepsilon}}.
		\end{align*}
		and
		using the upper bound in \eqref{two-side-bound-1}, we have
		\begin{align*}
			&\frac{c_1
				\varepsilon}{ \sqrt{\Psi''(\lambda_*)m}}\int_{\R_+}\left(\int_{\mathbb{R}}
			h(z)1_{\{z\ge 0\}}\phi\left(\frac{z-y}{\sqrt{\Psi''(\lambda_*) m}}\right)\mathrm{d} z\right)
			\mathbf{P}_x^{\lambda_*}\left(\xi_{t-m}\in \mathrm{d}y, \tau_0^->t-m\right)\\
			&\le \frac{2c_4
			\varepsilon R^*(x)
			}{\sqrt{2\pi}\Psi''(\lambda_*) t}\int_{\mathbb{R}_+}\!h(w)
			\rho\!
			\left(\!\frac{w}{ \sqrt{\Psi''(\lambda_*) t}}\!\right) \!\mathrm{d}w
			    +\!\frac{2c_4
				\sqrt{\varepsilon}R^*(x)}{\Psi''(\lambda_*) t}
			\int_{\R_+}
			h(w)\phi\!\left(\!\frac{w}{\sqrt{\Psi''(\lambda_*) t}}\!\right)\!
			\mathrm{d}w\\
			&\quad+ \frac{
				c_3(\varepsilon)
				(1+x)
				\|H1_{[0,\infty)}\|_1
				}{t^{1+\varepsilon}},
		\end{align*}
		where $c_4$ is a positive constant independent of $\varepsilon$.
		Thus
there exists a positive constant $c_5(\varepsilon)$ such that
for $t>t_1(\varepsilon)$,
		\begin{align}\label{lower-bound-I1}
			I_1(t)
			\ge& \frac{2
			R^*(x)}{\sqrt{2\pi}\Psi''(\lambda_*) t}
			\int_{\mathbb{R}_+}
			\left(g(w)1_{\{w\ge \varepsilon\}}
				-c_4\varepsilon h(w)\right)
			\rho\left(\frac{w}{ \sqrt{\Psi''(\lambda_*) t}} \right)
			\mathrm{d} w\\
			&-c_4\sqrt{\varepsilon}
			\frac{2R^*(x)}{\Psi''(\lambda_*) t}\!
			\int_{\R_+}\!
			h(w)\phi\!\left(\!\frac{w}{\sqrt{\Psi''(\lambda_*) t}}\!\right)\!
			\mathrm{d}w\\
			&-c_5(\varepsilon)
			(1+x)
			\|H1_{[0,\infty)}\|_1
			\left(\frac{1}{t^{1+\varepsilon}}
			+\frac{1}{t^{1+\delta/2}}\right).
		\end{align}

		{\bf Step 2.}
		Next, we give an upper bound for $I_2^1(t)$.
		Combining \eqref{local-limit-RW} and \eqref{upper-bound-tua}, we get that
		there exist positive constants $c_6$  (independent of $\varepsilon$)
		and $c_7(\varepsilon)$
		such that
		\begin{align*}
			I_2^1(t)
			&\le \int_{0}^{\varepsilon^{1/6}\sqrt{[t]}}
			\mathbf{E}_y^{\lambda_*}\left(
				h(\xi_m)1_{\{\xi_m\ge 0\}}
				\right)
			\mathbf{P}_x^{\lambda_*}\left(\xi_{t-m}\in \mathrm{d}y, \tau_0^->t-m\right)
			\\
			&\le \int_{0}^{\varepsilon^{1/6}\sqrt{[t]}}
					\!\frac{1+c_6\varepsilon
						}{ \sqrt{\Psi''(\lambda_*)m}}\!
			\int_{\mathbb{R}}\!H(z) 1_{\{z\ge -\varepsilon\}}
			\phi\!\left(\!\frac{z-y}{ \sqrt{\Psi''(\lambda_*) m}}\!\right)\mathrm{d}z
			 \mathbf{P}_x^{\lambda_*}\!\left(\xi_{t-m}\in \mathrm{d}y, \tau_0^->t-m\right)\\
			&\quad+\frac{c_7(\varepsilon)(1+x)
			\|H1_{[-\varepsilon,\infty)}\|_1
			}{m^{(1+\delta)/2}\sqrt{t-m}}.
		\end{align*}
		For any
		$u\in \mathbb{R}$,
		define
		\begin{align*}
			J(u):=\int_{0}^{\varepsilon^{1/6}\sqrt{[t]}}\left( \frac{1}{ \sqrt{\Psi''(\lambda_*) m}} \phi\left(\frac{u-y}{ \sqrt{\Psi''(\lambda_*) m}}\right)\right)
			\mathbf{P}_x^{\lambda_*}\left(\xi_{t-m}\in \mathrm{d}y, \tau_0^->t-m\right).
		\end{align*}
		Then by Fubini's theorem, we have
		\begin{align*}
			&\int_{0}^{\varepsilon^{1/6}\sqrt{[t]}}
			\frac{1}{ \sqrt{\Psi''(\lambda_*) m}}
			\int_{\mathbb{R}}H(z)1_{\{z\ge -\varepsilon\}}
			 \phi\left(\frac{z-y}{ \sqrt{\Psi''(\lambda_*) m}}\right)\mathrm{d}z
			\mathbf{P}_x^{\lambda_*}\left(\xi_{t-m}\in \mathrm{d}y, \tau_0^->t-m\right)\\
			&= \int_{\mathbb{R}}H(u)1_{\{u\ge -\varepsilon\}}
			J(u)\mathrm{d}u.
		\end{align*}
		For any $u\in\mathbb{R}_+$, define
		$$F_u(y):=\frac{1}{  \sqrt{\Psi''(\lambda_*) m}} \phi\left(\frac{u-y}{  \sqrt{\Psi''(\lambda_*) m}}\right).$$
		Using the definition of $J(u)$ and integration by parts, we get
		\begin{align*}
			J(u)
			&=\int_{0}^{\infty}F_u(y)\mathbf{P}_x^{\lambda_*}\left(\xi_{t-m}\in \mathrm{d}y, \xi_{t-m}\le \varepsilon^{1/6}\sqrt{[t]},\tau_0^->t-m\right)\\
			&
			\le
			\int_{0}^{\infty}F_u'(y)\mathbf{P}_x^{\lambda_*}\left(\xi_{t-m}>y, \xi_{t-m}\in(0,\varepsilon^{1/6}\sqrt{[t]}],\tau_0^->t-m\right)\mathrm{d}y\\
			&=\int_{0}^{\varepsilon^{1/6}\sqrt{[t]}}F_u'(y)\mathbf{P}_x^{\lambda_*}\left(
			\xi_{t-m}\in (y, \varepsilon^{1/6}\sqrt{[t]}],
			\tau_0^->t-m\right)\mathrm{d}y.
		\end{align*}
		Since $m=[\varepsilon t]$, using Lemma \ref{lemma-effective-conditioned-integral-limit},
		it holds that for $t-m>T_0(\varepsilon)$,
		\begin{align*}
			&\Big|
			\mathbf{P}_x^{\lambda_*}
			\left(
			\xi_{t-m}\in (y, \varepsilon^{1/6}\sqrt{[t]}],
			\tau_0^->t-m\right)\\
			&\quad-\frac{2R^*(x)
			}{\sqrt{2\pi  (t-m) \Psi''(\lambda_*)} }\int_{\frac{y}{ \sqrt{\Psi''(\lambda_*)(t-m)}}}^{\frac{\varepsilon^{1/6}\sqrt{[t]} }{ \sqrt{\Psi''(\lambda_*)(t-m)}}}
			\rho(z)\mathrm{d}z\Big|
			\le \frac{C_4(\varepsilon)
				(1+x)}{(t-m)^{1/2+\varepsilon}}.
		\end{align*}
		Now, using the fact
			\begin{align}
				F_u(\varepsilon^{1/6}\sqrt{[t]})-F_u(0)\le \frac{c_8(\varepsilon)}{\sqrt{t}},
			\end{align}
			for some $c_8(\varepsilon)>0$,
		we get that there exists a positive constant
		$c_9(\varepsilon)$ such that for $t-m>T_0(\varepsilon)$,
		\begin{align*}
			J(u)
			&\le \frac{
				c_9(\varepsilon)
				(1+x)}{(t-m)^{1/2+\varepsilon}}
			\frac{1}{\sqrt{t}}
			+ \frac{2R^*(x)}{\sqrt{2\pi (t-m)\Psi''(\lambda_*)}}\int_{0}^{\varepsilon^{1/6}\sqrt{[t]}}F_u'(y)\\
			&\quad \times \left(\mathcal{R}\left(\frac{\varepsilon^{1/6}\sqrt{[t]}}{ \sqrt{\Psi''(\lambda_*)(t-m)}}\right)-\mathcal{R}\left(\frac{y}{ \sqrt{\Psi''(\lambda_*)(t-m)}}\right)
			\right)\mathrm{d}y.
		\end{align*}
		Using integration by parts and the fact $F_u(0)\ge 0$ (see \cite[(3.33)]{GX-AIHP}),
		 we get that there exists a constant
		 $c_{10}$ (independent of $\varepsilon$) such that
		\begin{align*}
			&\int_{0}^{\varepsilon^{1/6}\sqrt{[t]}}F_u'(y) \left(
			\mathcal{R}
			\left(\frac{\varepsilon^{1/6}\sqrt{[t]}}{\sqrt{\Psi''(\lambda_*)(t-m)}}\right)-
			\mathcal{R}
			\left(\frac{y}{ \sqrt{\Psi''(\lambda_*)(t-m)}}\right)
			\right)
			\mathrm{d}y\\
			&\le
 \frac{
	c_{10}
	\varepsilon^{1/12}}
{\sqrt{t-m}}\phi\left(\frac{u}{ \sqrt{\Psi''(\lambda_*) t}}\right).
		\end{align*}
		It follows that
		  there exist positive constants $c_{11}$ (independent of $\varepsilon$),
		and $t_2(\varepsilon)$, $c_{12}(\varepsilon)$ such that for $t>t_2(\varepsilon)$,
		 \begin{align}\label{upper-bound-I_2^1}
			I_2^1(t)
		\le &c_{11}\varepsilon^{1/12} \frac{2R^*(x)
			}{\sqrt{2\pi\Psi''(\lambda_*)}t}
			\int_{-\varepsilon}^{\infty}
			H(u)  \phi\left(\frac{u}{ \sqrt{\Psi''(\lambda_*) t}}\right)\mathrm{d}u\\
			&+c_{12}(\varepsilon)(1+x)
			\|H1_{[-\varepsilon,\infty)}\|_1
			\left(\frac{1}{t^{1+\varepsilon}}+\frac{1}{t^{1+\delta/2}}\right).
		\end{align}

		{\bf Step 3.} Finally, we study the upper bound for $I_2^2(t)$.
		By the definition of $I_2^2(t)$, we have
		\begin{align}
			I_2^2(t)
			=\int_{\R}
			J_m(y)
			\mathbf{P}_x^{\lambda_*}\left(\xi_{t-m}\in \mathrm{d}y, \tau_0^->t-m\right),
		\end{align}
		where $J_m(y):=\mathbf{E}_y^{\lambda_*}\left(h(\xi_m) 1_{\{\xi_m\ge 0\}}1_{\{\tau_{0}^-\le m\}}\right)1_{\{y >\varepsilon^{1/6}\sqrt{[t]}\}}$.
		For any $z\in\mathbb{R}$, define
		\begin{align}
			M_m(z):=\mathbf{E}_z^{\lambda_*}\left(H(\xi_m) 1_{\{\xi_m\ge -\varepsilon\}}
			1_{\{\tau_{\varepsilon}^-\le m\}}
\right)1_{\{z+\varepsilon >\varepsilon^{1/6}\sqrt{[t]}\}}.
		\end{align}
		Consequently, $J_m\le_{\varepsilon}M_m$.
	         Applying Lemma \ref{lemma-rough-upper-bound-limsup} with $h$ and $H$ instead of $J_m$ and $M_m$,
				we get that for $t-m>T_1(\varepsilon)$,
		\begin{align}\label{upper-bound-I2}
			I_2^2(t)
			&\le \frac{2
			(1+C_5\varepsilon)
			R^*(x)
			}{\sqrt{2\pi} \Psi''(\lambda_*)(t-m)}\int_{\mathbb{R}_+}M_m(w)
			\rho
			\left(\frac{w}{\sqrt{\Psi''(\lambda_*) (t-m)}}\right) \mathrm{d}w\\
			&\quad+
			 \frac{2C_5
			 \sqrt{\varepsilon} R^*(x)
			}{\sqrt{2\pi}\Psi''(\lambda_*) (t-m)}
			\int_{-\varepsilon}^{\infty}
			M_m(w)
			e^{-\frac{w^2}{2\Psi''(\lambda_*)t}}
			\mathrm{d}w\\
			&\quad+C_6(\varepsilon)
			(1+x)
			\|M_m1_{[-\varepsilon,\infty)}\|_1
			\left(\frac{1}{(t-m)^{1+\varepsilon}}
			+\frac{1}{(t-m)^{1+\delta/2}}\right).
		\end{align}
       Now
we bound the three terms on the right-hand side of \eqref{upper-bound-I2} from above.
Using \eqref{dual-on-R},
		\begin{align}\label{upper-bound-Mm}
			\|M_m1_{[-\varepsilon,\infty)}\|_1
			&=\int_{\R} \mathbf{E}_z^{\lambda_*}\left(H(\xi_m) 1_{\{\xi_m\ge -\varepsilon\}}1_{\{\tau_{\varepsilon}^-\le m\}}
\right)1_{\{z+\varepsilon >\varepsilon^{1/6}\sqrt{[t]}\}}1_{\{z\ge -\varepsilon\}} \mathrm{d}z\\
&\le \int_{\R} \mathbf{E}_z^{\lambda_*}\left(H(\xi_m) 1_{\{\xi_m\ge -\varepsilon\}}
\right)1_{\{z+\varepsilon >\varepsilon^{1/6}\sqrt{[t]}\}}1_{\{z\ge -\varepsilon\}} \mathrm{d}z\nonumber\\
&= \int_{\R} H(z) 1_{\{z\ge -\varepsilon\}}\mathbf{P}_z^{\lambda_*}\left(\widehat{\xi}_m+\varepsilon >\varepsilon^{1/6}\sqrt{[t]}, \widehat{\xi}_m\ge -\varepsilon
\right) \mathrm{d}z\nonumber\\
			&\le
			\int_{-\varepsilon}^{\infty} H(z) \mathrm{d}z
			\le \|H1_{[-\varepsilon,\infty)}\|_1.\nonumber
		\end{align}
		Moreover, using the definition of $M_m$ and
 \eqref{duality-formula-2}, we get that
		\begin{align*}
			&\int_{\mathbb{R}_+}M_m(w)
			\rho
			\left(\frac{w}{\sqrt{\Psi''(\lambda_*)(t-m)}}\right) \mathrm{d}w\\
			&= \int_{\mathbb{R}_+}
			\mathbf{E}_w^{\lambda_*}\left(H(\xi_m) 1_{\{\xi_m\ge -\varepsilon\}}1_{\{\tau_{\varepsilon}^-\le m\}}
\right)1_{\{w+\varepsilon>\varepsilon^{1/6}\sqrt{[t]}\}}
			\rho \left(\frac{w}{ \sqrt{\Psi''(\lambda_*) (t-m)}}\right) \mathrm{d}w\\
			&=\int_{\mathbb{R}}
			\mathbf{E}_{w+\varepsilon}^{\lambda_*}\left(H(\xi_m) 1_{\{\xi_m\ge -\varepsilon\}}1_{\{\tau_{\varepsilon}^-\le m\}}
\right)1_{\{w+2\varepsilon>\varepsilon^{1/6}\sqrt{[t]}\}}
			\rho \left(\frac{w+\varepsilon}{ \sqrt{\Psi''(\lambda_*) (t-m)}}\right) \mathrm{d}w
			\\
			&=
			\int_{\mathbb{R}}
			\mathbf{E}_{w}^{\lambda_*}\left(H(\xi_m+\varepsilon) 1_{\{\xi_m+\varepsilon\ge -\varepsilon\}}1_{\{\tau_{0}^-\le m\}}\right)
			1_{\{w+2\varepsilon>\varepsilon^{1/6}\sqrt{[t]}\}}
			\rho \left(\frac{w+\varepsilon}{ \sqrt{\Psi''(\lambda_*) (t-m)}}\right) \mathrm{d}w
			\\
			 &=
			\int_{\mathbb{R}}
			H(w+\varepsilon)1_{\{w\ge -2\varepsilon\}}
			\mathbf{E}_w^{\lambda_*}
			\left(\rho \left(
				\frac{\widehat{\xi}_m+\varepsilon}
				{ \sqrt{\Psi''(\lambda_*) (t-m)}}\right)1_{\{
					\widehat{\xi}_m+2\varepsilon>\varepsilon^{1/6}\sqrt{[t]},
					\widehat{\tau}_0^-\le m\}}
			\right)\mathrm{d}w\\
			&=:J_1(t)+J_2(t),
		\end{align*}
		where
		\begin{align*}
			J_1(t):=\int_{-2\varepsilon}^{\varepsilon^{1/4}\sqrt{[t]}}
			H(w+\varepsilon)
			\mathbf{E}_w^{\lambda_*}
			\left(\rho \left(
				\frac{\widehat{\xi}_m+\varepsilon}
				{ \sqrt{\Psi''(\lambda_*) (t-m)}}\right)1_{\{
					\widehat{\xi}_m+2\varepsilon>\varepsilon^{1/6}\sqrt{[t]},
					\widehat{\tau}_0^-\le m\}}
			\right)\mathrm{d}w,
		\end{align*}
		\begin{align*}
			J_2(t):=\int_{\varepsilon^{1/4}\sqrt{[t]}}^{\infty}
			H(w+\varepsilon)
			\mathbf{E}_w^{\lambda_*}
			\left(\rho \left(
				\frac{\widehat{\xi}_m+\varepsilon}
				{ \sqrt{\Psi''(\lambda_*) (t-m)}}\right)1_{\{
					\widehat{\xi}_m+2\varepsilon>\varepsilon^{1/6}\sqrt{[t]},
					\widehat{\tau}_0^-\le m\}}
			\right)\mathrm{d}w.
		\end{align*}
		Next, we consider the upper bounds of $J_1(t)$ and $J_2(t)$ separately.
	We claim that there exist positive
	 constants $c_{13}$ and $q$ (both independent of $\varepsilon$),
	and $c_{14}(\varepsilon)$ such that
		\begin{align}\label{upper-bound-J1(t)}
			J_1(t)\le
			c_{13}
			 \varepsilon^{1/6} \int_{-2\varepsilon}^{\varepsilon^{1/4}\sqrt{[t]}}
			H(w+\varepsilon)
			\mathrm{d}w,
		\end{align}
and
\begin{align}\label{upper-bound-J_2(t)}
			J_2(t)
			\le c_{13}
			\varepsilon^{1/12}
			\int_{\varepsilon^{1/4}\sqrt{[t]}}^{\infty}
			H(w+\varepsilon)\phi\left(\frac{w}{\sqrt{\Psi''(\lambda_*) t}}\right)\mathrm{d}w
			+\frac{c_{14}(\varepsilon)
				}{t^q}\|H1_{[-\varepsilon,\infty)}\|_1.
		\end{align}
Using \eqref{upper-bound-J1(t)},  \eqref{upper-bound-J_2(t)} and the fact that $\phi$ is bounded,
		we immediately get
		there exists a positive constant
		$c_{15}$  (independent of $\varepsilon$)
		such that
		\begin{align}\label{upper-bound-M-rho}
			&\int_{\mathbb{R}_+}M_m(w)
			\rho
			\left(\frac{w}{\sqrt{\Psi''(\lambda_*)(t-m)}}\right) \mathrm{d}w\\
			\le &c_{15}
			\varepsilon^{1/12}
		\int_{-2\varepsilon}^{\infty}
			H(w+\varepsilon)\phi\left(\frac{w}{ \sqrt{\Psi''(\lambda_*) t}}\right)\mathrm{d}w+\frac{
				c_{14}(\varepsilon)
				}{t^q}\|H1_{[-\varepsilon,\infty)}\|_1.
		\end{align}
Similarly,
there exist constants
 $c_{16}$ (independent of $\varepsilon$) and $c_{17}(\varepsilon)$
such that
		\begin{align}
			\int_{-\varepsilon}^{\infty}
			M_m(w)e^{-\frac{w^2}{2\Psi''(\lambda_*)t}}
			\mathrm{d}w
			\le c_{16}
			\varepsilon^{1/12}
			\int_{-2\varepsilon}^{\infty}H(w+\varepsilon)
			 \phi\left(\frac{w}{\sqrt{\Psi''(\lambda_*)t}}\right)
			\mathrm{d}w
			+\frac{c_{17}(\varepsilon)}
			{t^{q}}
			\|H 1_{[-\varepsilon,\infty)}\|_1.
		\end{align}
		Combining the last two displays with \eqref{upper-bound-I2}
		and \eqref{upper-bound-Mm},
		we get that there exist positive constants
		 $c_{18}$ (independent of $\varepsilon$), $t_3(\varepsilon)$ and $c_{19}(\varepsilon)$
		such that for $t>t_3(\varepsilon)$,
		\begin{align}\label{upper-bd-I_2^2}
			I_2^2(t)
			\le&
			\frac{c_{18}
			R^*(x)}{\sqrt{2\pi}\Psi''(\lambda_*) t}
				\varepsilon^{1/12}\!\!
			\int_{-2\varepsilon}^{\infty}H(w+\varepsilon)
			\phi\Big(\frac{w}{\sqrt{\Psi''(\lambda_*)t}}\Big)
			\mathrm{d}w
			\\&
			+ c_{19}(\varepsilon)
			(1+x)\|H 1_{[-\varepsilon,\infty)}\|_1
			\left(\frac{1}{t^{1+\varepsilon}}
			+\frac{1}{t^{1+\delta/2}}
			+\frac{1}{t^{1+q}}
		\right).\nonumber
		\end{align}

		Combining \eqref{upper-bound-I_2^1} and \eqref{upper-bd-I_2^2},
		and using the fact that there exists $c_{20}>0$ such that $R^*(x)\le c_{20}(1+x)$,
		 we get that
		there exist positive constants
		$c_{21}$ (independent of $\varepsilon$),
		$c_{22}(\varepsilon)$
		and $t_4(\varepsilon)$,
		such that for $t>t_4(\varepsilon)$,
\begin{align}
			I_2(t)
			&\le
			c_{20}
			 \varepsilon^{1/12} \frac{2R^*(x)
			}{\sqrt{2\pi}\Psi''(\lambda_*)t}
			\int_{-\varepsilon}^{\infty}
			H(u)  \phi\left(\frac{u}{ \sqrt{\Psi''(\lambda_*) t}}\right)\mathrm{d}u\\
			&\quad+
			c_{22}(\varepsilon)
			(1+x)\|H 1_{[-\varepsilon,\infty)}\|_1
			\left(\frac{1}{t^{1+\varepsilon}}
			+\frac{1}{t^{1+\delta/2}}+\frac{1}{t^{1+q}}\right).
		\end{align}
		Combining this with \eqref{lower-bound-I1} gives the desired result.

Now we prove the claims  \eqref{upper-bound-J1(t)} and \eqref{upper-bound-J_2(t)}.
		Using the boundedness of $\rho$ and the fact that
there exists a positive constant $t_5$ independent of $\varepsilon$ such that
		\begin{align}\label{e:elemineq}
			\varepsilon^{1/6}\sqrt{[t]}-\varepsilon^{1/4}\sqrt{[t]}-2\varepsilon>\frac{1}{2}\varepsilon^{1/6}\sqrt{[t]},
			\quad
			t>t_5.
		\end{align}
		Therefore,  we get that
	there exists a positive constant
	$c_{23}$ such that
	for $t>t_5$,
		\begin{align}\label{upper-bound-J1-1}
			J_1(t)
			&\le  c_{23}
			\int_{-2\varepsilon}^{\varepsilon^{1/4}\sqrt{[t]}}
			H(w+\varepsilon)
			\mathbf{P}_w^{\lambda_*}
			\left(\widehat{\xi}_m+2\varepsilon>\varepsilon^{1/6}\sqrt{[t]}
			\right)\mathrm{d}w\\
			&
			\le  c_{23}
			\mathbf{P}_0^{\lambda_*}
			\left(\widehat{\xi}_m>\frac{1}{2}\varepsilon^{1/6}\sqrt{[t]}\right)
			\int_{-2\varepsilon}^{\varepsilon^{1/4}\sqrt{[t]}}
			H(w+\varepsilon)
			\mathrm{d}w.
		\end{align}
		Using \cite[Lemma 3.4]{GX-AIHP}
		with  $u=v=\frac{1}{2}\varepsilon^{1/6}\sqrt{[t]}$,
		since $m=[\varepsilon t]$, we get that
		there exist  positive constants
		$c_{24}$ and $c_{25}$ both independent of $\varepsilon $
		such that
		\begin{align}\label{upper-bound-hat{xi}}
			\mathbf{P}_0^{\lambda_*}
			\left(\widehat{\xi}_m>\frac{1}{2}\varepsilon^{1/6}\sqrt{[t]}\right)
			&\le 2 \exp\Big\{\left(1+\frac{4m}{\varepsilon^{1/3}[t]}\right)\Big\}
			+m\mathbf{P}_0^{\lambda_*}\left(|\widehat{\xi}_1|>\frac{1}{2}\varepsilon^{1/6}\sqrt{[t]}\right)\nonumber\\
			&\le  c_{24}
			\varepsilon^{2/3}+[\varepsilon t] \frac{4 \mathbf{E}_0^{\lambda_*}(\widehat{\xi}_1^2)}{\varepsilon^{1/3}[t]}
			\le  c_{25}
			\varepsilon^{1/6},
		\end{align}
		where in the second inequality we used Chebyshev's inequality and {\bf(H1)}.
		Combining this with \eqref{upper-bound-J1-1}, we complete the proof of \eqref{upper-bound-J1(t)}.
		
Next we prove \eqref{upper-bound-J_2(t)}.
		Using \eqref{e:elemineq} and H\"older's inequality, we get that for all $t>t_5$ and $w>0$, we have
		\begin{align}\label{upper-bound-joint-proba}
			&\mathbf{P}_w^{\lambda_*}
			\left(\widehat{\xi}_m+2\varepsilon>\varepsilon^{1/6}\sqrt{[t]},
			\widehat{\tau}_0^-\le m
			\right)\\
			&\le \mathbf{P}_w^{\lambda_*}
			\left(\max_{s\in[0,m]}|\widehat{\xi}_s|>\frac{1}{2}\varepsilon^{1/6}\sqrt{[t]}
			\right)^{1/2}
			\mathbf{P}_w^{\lambda_*}
			\left(\widehat{\tau}_{0}^-\le m\right)^{1/2}\nonumber\\
			&=\mathbf{P}_w^{\lambda_*}
			\left(\max_{s\in[0,m]}|\widehat{\xi}_s|>\frac{1}{2}\varepsilon^{1/6}\sqrt{[t]}
			\right)^{1/2}
			\mathbf{P}_0^{\lambda_*}
			\left(\widehat{\tau}_{-w}^-\le m\right)^{1/2}.\nonumber
		\end{align}
		By Lemma \ref{lemma-relation-with-BM}, there exists a Brownian motion
		$W$ with diffusion coefficient $\Psi''(\lambda_*)$,
		starting from the origin, such that for any $t\ge 1$ and $x>0$,
		\begin{align}
			\mathbf{P}_0^{\lambda_*}(\widehat{A}_t)
			\le
			\frac{C_3(2\varepsilon)}
			{t^{(\frac{1}{2}-2\varepsilon)(\delta+2)-1}},
		\end{align}
		where $\widehat{A}_t$ is defined by
		\begin{align}
			\widehat{A}_t:=\Big\{\sup_{s\in[0,1]}|\widehat{\xi}_{ts}-\widehat{W}_{ts}|>t^{\frac{1}{2}-2\varepsilon}\Big\}.
		\end{align}
		Therefore,
		 there exists positive constant $q$ (independent of $\varepsilon$) and $c_{26}(\varepsilon)$
		such that
		\begin{align}\label{proba-A_m}
			\mathbf{P}_0^{\lambda_*}
			\left(\widehat{\tau}_{-w}^-\le m,\widehat{A}_m\right)
			\le
			\frac{C_3(2\varepsilon)
				}{m^{(\frac{1}{2}-2\varepsilon)(\delta+2)-1}}
			\le \frac{c_{26}(\varepsilon)}{t^{2q}}.
		\end{align}
		Moreover, for $w>\varepsilon^{1/4}\sqrt{[t]}$, we have
		there exists a positive constant
		 $c_{27}$
		such that
		\begin{align}\label{proba-A_m^C}
			&\mathbf{P}_0^{\lambda_*}
			\left(\widehat{\tau}_{-w}^-\le m,A_m^c\right)
			=\mathbf{P}_0^{\lambda_*}
			\left(
				\inf_{s\in[0,m]}\widehat{\xi}_s<-w,\widehat{A}_m^c
				\right)\\
			&\le \mathbf{P}_0^{\lambda_*}\left(
				\inf_{s\in[0,m]}\widehat{W}_s
				<m^{\frac{1}{2}-2\varepsilon}-w\right)
			=\frac{2}{\sqrt{2\pi\Psi''(\lambda_*) m}}\int_{w-m^{\frac{1}{2}-2\varepsilon}}^{\infty}e^{-\frac{s^2}{2\Psi''(\lambda_*) m}}\mathrm{d}s\nonumber\\
			&\le  c_{27}
			\int_{\frac{w}{2 \sqrt{\Psi''(\lambda_*) m}}}^{\infty}e^{-\frac{s^2}{2}}\mathrm{d}s
			 \le  \frac{2 c_{27}
			\sqrt{\Psi''(\lambda_*) m}}{w} e^{-\frac{w^2}{8\Psi''(\lambda_*) m}},\nonumber
		\end{align}
		where in the last inequality we used the fact that $\int_{a}^{\infty}e^{-\frac{s^2}{2}}\mathrm{d}s\le \frac{1}{a}e^{-\frac{a^2}{2}}$ for any $a>0$. 		
		Combining \eqref{proba-A_m^C} and \eqref{proba-A_m},
		for $w>\varepsilon^{1/4}\sqrt{[t]}$,
		since $\frac{\sqrt{m}}{w}\le 1$,
		it holds that
		\begin{align}\label{upper-bound-tau<m}
			\mathbf{P}_0^{\lambda_*}
			\left(\widehat{\tau}_{-w}^-\le m\right)^{1/2}
			\le c_{28}
			\phi\left(\frac{w}{ \sqrt{\Psi''(\lambda_*) t}}\right)+\frac{
				c_{29}(\varepsilon)
				}{t^q},
		\end{align}
		for some positive constants
		$c_{28}$  and $c_{29}(\varepsilon)$.
		Similarly, we can get that
		\begin{align}\label{upper-bound-max-xi}
			\mathbf{P}_w^{\lambda_*}
			\left(\max_{s\in[0,m]}|\widehat{\xi}_s|>\frac{1}{2}\varepsilon^{1/6}\sqrt{[t]}
			\right) \le
			c_{30}\varepsilon^{1/6}+\frac{c_{31}(\varepsilon)}{t^{2q}},
		\end{align}
			for some positive constants $c_{30}$ and $c_{31}(\varepsilon)$.
		Combining this with \eqref{upper-bound-joint-proba}, we get
		there exist positive constants
		$c_{32}$  and $c_{33}(\varepsilon)$ such that
		\begin{align}
			\mathbf{P}_w^{\lambda_*}
			\left(\widehat{\xi}_m+2\varepsilon>\varepsilon^{1/6}\sqrt{[t]},
			\widehat{\tau}_0^-\le m
			\right)
			\le
			c_{32}
			\varepsilon^{1/12}\phi\left(\frac{w}{ \sqrt{\Psi''(\lambda_*) t}}\right)+\frac{
			c_{33}(\varepsilon)
				}{t^q},\quad w>\varepsilon^{1/4}\sqrt{[t]}.
		\end{align}
		This completes the proof of \eqref{upper-bound-J_2(t)}.
	\end{proof}
	
	\begin{lemma}\label{lemma-lower-bound-liminf1}
		Assume that $\xi$ is a  L\'{e}vy process satisfying {\bf(H1)}, {\bf(H2)} {\bf(H3)} and $\mathbf{E}_0[\xi_1]<0$.
Then one can find  positive constants  $C_{11}$  and $q$
		 with the property that for any $\varepsilon\in(0,\varepsilon_0)$
		there exist positive constants $T_4(\varepsilon)$ and $C_{12}(\varepsilon)$
		 such that for any $x>0$,
		$t>T_4(\varepsilon)$ and any  Borel functions
		$h,H,g:\R \to\R_+$
		satisfying $g\le_{\varepsilon}h\le_{\varepsilon}H$
		and $\int_{\R_+}H(z-\varepsilon)(1+z)\mathrm{d}z<\infty$,
		\begin{align*}
			\mathbf{E}_x^{\lambda^*}
			\left(h(\xi_t) 1_{\{ \tau_{0}^->t\}}\right)
			&\ge
			 \left(1-C_{11}t^{-1/2}-C_{12}(\varepsilon)t^{-\varepsilon}\right)
			\frac{2R^*(x)}{\sqrt{2\pi \Psi''(\lambda_*)^3} t^{3/2}}\int_{\mathbb{R}_+}
			g(z+\varepsilon)
			\widehat{R}^*(z)\mathrm{d}z\\
			&\quad
			-C_{12}(\varepsilon)
			   \left(1+C_{11}\varepsilon t^{-1/2}+t^{-\varepsilon}\right)
			\frac{2R^*(x)}{\sqrt{2\pi \Psi''(\lambda_*)^3} t^{3/2}}\int_{\mathbb{R}_+}
			H(z-\varepsilon)
			\widehat{R}^*(z)\mathrm{d}z\\
			&\quad
			-\frac{C_{12}(\varepsilon)(1+x)}{\sqrt{t}}
			\left(\frac{1}{t^{1+\varepsilon}}
			+\frac{1}{t^{1+\delta/2}}+\frac{1}{t^{1+q}}
			\right)\int_{\R_+}
			H(z-\varepsilon)(1+z)
			\mathrm{d}z.
		\end{align*}
	\end{lemma}
	\begin{proof}
		Recall that the functions $H_m$ and $I_m$ are defined in \eqref{def-Hm} and \eqref{def-I-H}.
		Fix $\varepsilon\in(0,\varepsilon_0)$ and let
		$h,H,g:\R\to\R_+$
		be Borel functions
		satisfying
		$g\le_{\varepsilon}h\le_{\varepsilon}H$
            and $\int_{\R_+}H(z-\varepsilon)(1+z)\mathrm{d}z<\infty$.
		For any $y\in\mathbb{R}$, define
		\begin{align}
			N_m(y):=\mathbf{E}_y^{\lambda_*}\left(g(\xi_m) 1_{\{\xi_m \ge \varepsilon\}}1_{\{\tau_{\varepsilon}^-> m\}}\right).
		\end{align}
		Then for any $y>0$ and $|v|\le \varepsilon$,
		\begin{align}
			N_m(y)
			\le \mathbf{E}_y^{\lambda_*}\left(h(\xi_m+v) 1_{\{\xi_m \ge \varepsilon\}}1_{\{\tau_{\varepsilon}^-> m\}}\right)
			\le \mathbf{E}_{y}^{\lambda_*}\left(h(\xi_m+v)1_{\{\tau_{-v}^-> m\}}\right)
			= I_m(y+v).
		\end{align}
		Therefore,
		$N_m\le_{\varepsilon}I_m \le_{\varepsilon}H_m$.
		Applying Lemma \ref{lemma-rough-lower-bound-liminf}
		with $h=I_m$, we get that
for $t-m>T_3(\varepsilon)$,
		\begin{align*}
			&\mathbf{E}_x^{\lambda_*}\left(h(\xi_t) 1_{\{ \tau_{0}^->t\}}\right)
			=\mathbf{E}_x^{\lambda_*}\left(I_m(\xi_{t-m}),\tau_0^->t-m\right)\\
			&\ge \frac{2R^*(x)}{\sqrt{2\pi}\Psi''(\lambda_*)(t-m)}
			 \int_{\mathbb{R}_+}N_m(z)
			 1_{\{z\ge \varepsilon\}}
			 \rho \left(\frac{z}{ \sqrt{\Psi''(\lambda_*)(t-m)}} \right)\mathrm{d} z\\
			 &-\frac{2C_9
			 \varepsilon R^*(x)}{\sqrt{2\pi}\Psi''(\lambda_*)(t-m)}
			 \int_{\mathbb{R}_+}I_m(z)\rho \left(\frac{z}{ \sqrt{\Psi''(\lambda_*)(t-m)}} \right)\mathrm{d} z\\
			&-\frac{
			C_9
			\varepsilon^{1/12} R^*(x)
			}{\sqrt{2\pi}\Psi''(\lambda_*)(t-m)}
			\int_{-\varepsilon}^{\infty}
			H_m(z)\phi\left(\frac{z}{ \sqrt{\Psi''(\lambda_*)(t-m)}}\right)\mathrm{d}z\\
			&-C_{10}(\varepsilon)
			(1+x)
			\|H_m1_{[-\varepsilon,\infty)}\|_1
			\left( \frac{1}{(t-m)^{1+\delta/2}}+\frac{1}{(t-m)^{1+\varepsilon}}
			+\frac{1}{(t-m)^{1+q}}
			\right)=:\sum_{i=1}^{4}K_i,
		\end{align*}
			where $q$ is the constant in Lemma \ref{lemma-rough-lower-bound-liminf}.
	 By \eqref{duality-formula-1}, we have
		\begin{align*}
			K_1
			=&\frac{2R^*(x)
			}{\sqrt{2\pi}\Psi''(\lambda_*)(t-m)}\! \int_{\mathbb{R}_+}\!
			\mathbf{E}_z^{\lambda_*}\!
			\left(g(\xi_m)1_{\{\xi_m\ge \varepsilon\}} 1_{\{\tau_{\varepsilon}^-> m\}}\right)1_{\{z\ge \varepsilon\}}
			\rho\!\left(\!\frac{z}{ \sqrt{\Psi''(\lambda_*)(t-m)}}\! \right)\!\mathrm{d} z\\
			=&\frac{2R^*(x)
			}{\sqrt{2\pi}\Psi''(\lambda_*)(t-m)} \int_{\mathbb{R}_+}
			g(z+\varepsilon) \mathbf{E}_z^{\lambda_*}\left(\rho\left(\frac{\widehat{\xi}_m+\varepsilon}{ \sqrt{\Psi''(\lambda_*)(t-m)}} \right)1_{\{\widehat{\tau}_0^-> m\}}\right)
			\mathrm{d} z.
		\end{align*}
		Repeating the argument leading to \eqref{upper-bound-J1a}, we get
		that there exist positive constants $c_1$ (independent of $\varepsilon$) and $c_2(\varepsilon)$
		such that
		\begin{align}\label{lower-bound-K1}
			K_1
			&\ge \left(1-c_1t^{-1/2}\right)
			\frac{2R^*(x)
			}{\sqrt{2\pi \Psi''(\lambda_*)^3} t^{3/2}}\int_{\mathbb{R}_+}
			g(z+\varepsilon)\widehat{R}^*(z)
			\mathrm{d}z\\
			&\quad-\frac{2c_2(\varepsilon)
			R^*(x)
			}{\sqrt{2\pi \Psi''(\lambda_*)^3} t^{3/2+\varepsilon}}
			\int_{\mathbb{R}_+}
			g(z+\varepsilon)(1+z)
			\mathrm{d}z.
		\end{align}
		Using an argument similar to that leading to
		\eqref{upper-bound-J1a},
		we get
		that there exist positive constants $c_3$ independent of $\varepsilon$ and $c_4(\varepsilon)$ such that
		\begin{align}\label{lower-bound-K2}
		    K_2
			&\ge
			-c_3\varepsilon(1+c_3\varepsilon t^{-1/2})
			\frac{2R^*(x)
			}{\sqrt{2\pi \Psi''(\lambda_*)^3}t^{3/2}} \int_{\mathbb{R}_+}
			H(z-\varepsilon)
			\widehat{R}^*(z)
			\mathrm{d} z\\
			&\quad
			-c_4(\varepsilon)
			\frac{2R^*(x)
			}{\sqrt{2\pi \Psi''(\lambda_*)^3}t^{3/2+\varepsilon}} \int_{\mathbb{R}_+}
			H(z-\varepsilon)
			(1+z)\mathrm{d} z,
		\end{align}
and
		\begin{align}\label{lower-bound-K3}
            K_3
			&\ge
			-
			c_3
			\varepsilon^{1/12}(1+c_3\varepsilon t^{-1/2})
			 \frac{2R^*(x)
			}{\sqrt{2\pi \Psi''(\lambda_*)^3}t^{3/2}} \int_{\mathbb{R}_+}
			H(z-\varepsilon)
			\widehat{R}^*(z)
			\mathrm{d} z\\
			&\quad
			-c_4(\varepsilon)
			\frac{2R^*(x)
			}{\sqrt{2\pi \Psi''(\lambda_*)^3}t^{3/2+\varepsilon}} \int_{\mathbb{R}_+}
			H(z-\varepsilon)
			(1+z)\mathrm{d} z.
		\end{align}
		Moreover, by \eqref{equivalent-Hm}, we get that
		\begin{align}\label{lower-bound-K4}
			K_4
			&\ge -\frac{
				c_5(\varepsilon)
				(1+x)}{\sqrt{t}}
			\left(\frac{1}{t^{1+\varepsilon}}
			+\frac{1}{t^{1+\delta/2}}
			+\frac{1}{t^{1+q}}
			\right)\int_{\R_+}(1+z)
			H(z-\varepsilon)
			\mathrm{d}z,
		\end{align}
		for some positive constant $c_5(\varepsilon)$.
Combining \eqref{lower-bound-K1}, \eqref{lower-bound-K2}, \eqref{lower-bound-K3} and \eqref{lower-bound-K4},
		we get the desired result.
	\end{proof}

	{\bf Proof of Theorem \ref{thm-tau-rho<0}:}
	Since $h:\mathbb{R} \to \mathbb{R}_+$
	is a Borel function and
	$z\mapsto h(z)(1+|z|)$ is directly Riemann integrable,
	by \cite[Lemma 2.3]{GX-AIHP}, there exists $a \in(0,1)$ such that
	$\int_{\mathbb{R}}\bar{h}_{a,\varepsilon}(1+|z|)\mathrm{d}z<\infty$,
	for any $\varepsilon \in(0,a)$, where $\bar{h}_{a,\varepsilon}$ is defined in \eqref{def-varepsilon-domain}.
	Applying Lemma \ref{lemma-upper-bound-limsup} to $h$, we have
for $t>T_2(\varepsilon)$,
	\begin{align*}
		t^{3/2}\mathbf{E}_x^{\lambda_*}\left(h(\xi_t)1_{\{ \tau_{0}^->t\}}\right)
		&\le
	\left(1+C_7t^{-1/2}+C_7\sqrt{\varepsilon}\right)
		\frac{2 R^*(x)}{\sqrt{2\pi\Psi''(\lambda_*)^3}}\int_{\mathbb{R}_+}
		\bar{h}_{a_m,\varepsilon}(z-\varepsilon)
		\widehat{R}^*(z)
		\mathrm{d}z\\
		&\quad+\frac{
			2C_7R^*(x)
		}{\sqrt{2\pi\Psi''(\lambda_*)^3} t^{\varepsilon}}\int_{\mathbb{R}_+}
		\bar{h}_{a_m,\varepsilon}(z-\varepsilon)
		(1+z)\mathrm{d}z\\
		&\quad+
		C_8(\varepsilon)
		(1+x)\left(
		\frac{1}{t^{\varepsilon}}
		+\frac{1}{t^{\delta/2}}\right)\int_{\mathbb{R}_+}
		\bar{h}_{a_m,\varepsilon}(z-\varepsilon)
		(1+z)\mathrm{d}z,
	\end{align*}
	where $a_m=2^{-m}a$, $m\ge 0$.
	On the other hand, by Lemma \ref{lemma-lower-bound-liminf1},
	we have
	for $t>T_4(\varepsilon)$,
	\begin{align*}
		t^{3/2}\mathbf{E}_x^{\lambda_*}\left(h(\xi_t) 1_{\{ \tau_{0}^->t\}}\right)
		&\ge
		\left(1-C_{11}t^{-1/2}-C_{12}(\varepsilon)t^{-\varepsilon}\right)
		\frac{2R^*(x)
		}{\sqrt{2\pi \Psi''(\lambda_*)^3} }\int_{\mathbb{R}_+}
		\underline{h}_{a_m,\varepsilon}(z+\varepsilon)
		\widehat{R}^*(z)
		\mathrm{d}z\\
		&\quad
		-C_{11}\varepsilon\left(1+C_{11}\varepsilon t^{-1/2}+t^{-\varepsilon}\right)
		\frac{2R^*(x)
		}{\sqrt{2\pi \Psi''(\lambda_*)^3} }\int_{\mathbb{R}_+}
		\bar{h}_{a_m,\varepsilon}(z-\varepsilon)
		\widehat{R}^*(z)
		\mathrm{d}z\\
		&\quad
		-C_{12}(\varepsilon)
		(1+x)\left(\frac{1}{t^{\varepsilon}}
		+\frac{1}{t^{\delta/2}}
		+\frac{1}{t^{q}}
		\right)\int_{\R_+}(1+z)
		\bar{h}_{a_m,\varepsilon}(z-\varepsilon)
		\mathrm{d}z.
	\end{align*}
	Since $h$ is not almost
	everywhere $0$ on $(0,\infty)$,
	we have
	$$\int_{\R_+}h(z)\widehat{R}^*(z)\mathrm{d}z\ge \widehat{R}^*(0)\int_{\R_+}h(z)\mathrm{d}z>0.$$
	Thus,
	\begin{align}\label{upper-bound-loacl-limit}
		\limsup_{t\to\infty}
		\frac{t^{3/2}\mathbf{E}_x^{\lambda_*}\left(h(\xi_t)1_{\{ \tau_{0}^->t\}}\right)}{\frac{
			2 R^*(x)
			}{\sqrt{2\pi\Psi''(\lambda_*)^3}}\int_{\mathbb{R}_+}h(z)
		\widehat{R}^*(z)\mathrm{d}z}
		&\le
		\left(1+C_7\sqrt{\varepsilon}\right)
		\limsup_{t\to\infty}I(\varepsilon,m),
	\end{align}
and
	\begin{align}
		\limsup_{t\to\infty}
		\frac{t^{3/2}\mathbf{E}_x^{\lambda_*}\left(h(\xi_t)1_{\{ \tau_{0}^->t\}}\right)}{\frac{
			2 R^*(x)
			}{\sqrt{2\pi\Psi''(\lambda_*)^3}}\int_{\mathbb{R}_+}h(z)
		\widehat{R}^*(z)\mathrm{d}z}
		&\ge
		\limsup_{t\to\infty}
		\left(J(\varepsilon,m)
		-C_{11}\varepsilon
		I(\varepsilon,m)\right),
	\end{align}
	where
	\begin{align}
		I(\varepsilon,m):=\frac{\int_{\mathbb{R}_+}
		\bar{h}_{a_m,\varepsilon}(z-\varepsilon)
			\widehat{R}^*(z)\mathrm{d}z}{\int_{\mathbb{R}_+}h(z)
			\widehat{R}^*(z)\mathrm{d}z}, \quad
		J(\varepsilon,m):=\frac{\int_{\mathbb{R}_+}
		\underline{h}_{a_m,\varepsilon}(z+\varepsilon)
			(1+z)\mathrm{d}z}{\int_{\mathbb{R}_+}h(z)
			\widehat{R}^*(z)\mathrm{d}z}.
	\end{align}
	Repeating the argument
	for  $I(y, \varepsilon, m)$ on \cite[pp. 40--41]{GX-AIHP}, we get
	\begin{align}
		\lim_{\varepsilon\to \infty}\limsup_{t\to\infty}I(\varepsilon,m)
		=1.
	\end{align}
	 This combined with \eqref{upper-bound-loacl-limit} yields that
	\begin{align}
		\limsup_{t\to\infty}
		t^{3/2}\mathbf{E}_x^{\lambda_*}\left(h(\xi_t)1_{\{ \tau_{0}^->t\}}\right)
		\le \frac{2 R^*(x)}{\sqrt{2\pi\Psi''(\lambda_*)^3}}\int_{\mathbb{R}_+}h(z)
		\widehat{R}^*(z)\mathrm{d}z.
	\end{align}
	The lower bound can be obtained in a similar way and this gives the desired result.
	\qed

	\section{Proof of Theorem \ref{thm-survival-probability} and Theorem \ref{thm-tail probability-Mt}}\label{section-proof-thm1-thm2}
	
	In this section, we give the proofs of  Theorems \ref{thm-survival-probability} and  \ref{thm-tail probability-Mt}. For any $x,t>0$ and $y\ge 0$, define
	\begin{align}\label{def-u}
		u(x,t):=\P_x(\zeta>t),
	\end{align}
	and
	\begin{align}\label{def-Q}
		Q_y(x,t):=\P_x(M_t>y).
	\end{align}
	It is easy to see that
	\begin{align}
		Q_0(x,t):=\P_x(M_t>0)
		=\P_x(\zeta>t)=u(x,t).
	\end{align}

Let $B_b^+(\mathbb{R}_+)$ be the space of non-negative bounded Borel functions on $\mathbb{R}_+$.
The following result is \cite[Lemma 2.1]{Hou24} which is true for any  branching killed L\'{e}vy process.

	\begin{lemma}\label{lemma-Integral-equation}
		For any $h\in B_b^+(\mathbb{R}_+)$, the function
		\begin{align}
			u_h(x,t):=\E_x\left(e^{-\int_{\mathbb{R}_+}h(y)
			Z_t^0(\mathrm{d}y)
			}\right),\quad t>0,~x>0,
		\end{align}
		solves the equation
		\begin{align}
			u_h(x,t)=\mathbf{E}_x\left(e^{-h(\xi_{t\land \tau_0^-})}\right)+\beta \mathbf{E}_x\left( \int_{0}^{t}\left(\sum_{k=0}^{\infty}p_ku_h(\xi_{s\land \tau_0^-},t-s)^k-u_h(\xi_{s\land \tau_0^-},t-s)\right)\mathrm{d}s\right).
		\end{align}
		Consequently, $v_h(x,t)=1-u_h(x,t)$ satisfies
		\begin{align}
			v_h(x,t)=\mathbf{E}_x\left(1-e^{-h(\xi_{t\land \tau_0^-})}\right)-\mathbf{E}_x\left(\int_{0}^{t}\Phi(v_h(\xi_{s\land \tau_0^-},t-s))\mathrm{d}s\right).
		\end{align}
	\end{lemma}

	The next result is also valid for any branching killed L\'{e}vy process.
	
	\begin{lemma}\label{lemma-expression-Q}
		For any $x,t>0$ and $y\ge 0$,
		it holds that
		\begin{align}\label{expectation-expression-Q}
			Q_y(x,t)
			=e^{-\alpha t}\mathbf{E}_x\left(1_{\{\tau_0^->t,\xi_t>y\}}e^{-\int_{0}^{t} \varphi(Q_y(\xi_s,t-s))\mathrm{d}s}\right).
		\end{align}
	\end{lemma}
	\begin{proof}
		For any $x,t>0$ and $y\ge 0$, by the dominated convergence we have
		\begin{align}
			1-Q_y(x,t)
			&=\P_x(M_t\le y)
	        =\P_x(Z^0_t(y,\infty)=0)
			=\lim_{\theta \to \infty}
        \E_x\left(e^{-\theta Z^0_t(y,\infty)}\right)\\
			&=\lim_{\theta \to \infty} \E_x\left(e^{-\int_{\mathbb{R}_+}\theta 1_{(y,\infty)}(z)
    Z^0_t(\mathrm{d}z)}\right).
		\end{align}
		Now applying  Lemma \ref{lemma-Integral-equation}
		with $h(z)=1_{(y,\infty)}(z)$,
		we get
		\begin{align}\label{expression-Q}
			Q_y(x,t)
			&=\lim_{\theta \to \infty}\mathbf{E}_x\left(1-e^{-\theta 1_{(y,\infty)}(\xi_{t\land \tau_0^-})}\right)-\mathbf{E}_x\left(\int_{0}^{t}\Phi(Q_y(\xi_{s\land \tau_0^-},t-s))\mathrm{d}s\right)\\
			&=\mathbf{P}_x\left(\xi_{t\land \tau_0^-}>y\right)-\mathbf{E}_x\left(\int_{0}^{t}\Phi(Q_y(\xi_{s\land \tau_0^-},t-s))\mathrm{d}s\right).
		\end{align}
Thus $Q_y(x,t)$ is a bounded solution of the following equation
\begin{align}\label{int-Q}
			u(x,t)=\mathbf{P}_x\left(\xi_{t\land \tau_0^-}>y\right)-\mathbf{E}_x\left(\int_{0}^{t}\Phi(u(\xi_{s\land \tau_0^-},t-s))\mathrm{d}s\right).
		\end{align}
It follows from \cite[(4.8), p.102]{Li} that there is a unique positive locally bounded solution to  \eqref{int-Q}.
Thus we only need to prove that the right side of \eqref{expectation-expression-Q} is also a solution \eqref{int-Q}.
For $s\in [0,t]$, define
		\begin{align}
			A_{s,t}=-\int_{s}^{t}\frac{\Phi(Q_y(\xi_{r\land \tau_0^{-}},t-r))}{Q_y(\xi_{r\land \tau_0^{-}},t-r)}\mathrm{d}r.
		\end{align}
Note that $\frac{\Phi(u)}{u}=\varphi(u)+\alpha$ for $u\in(0,1]$. The right side of \eqref{expectation-expression-Q} can be written as $\mathbf{E}_x\left(e^{A_{0,t}}1_{\{\tau_0^->t,\xi_t>y\}}\right)$.
		It is elementary to check that
		\begin{align}
			e^{A_{0,t}}=1-\int_{0}^{t}e^{A_{s,t}}\frac{\Phi(Q_y(\xi_{s\land \tau_0^{-}},t-s))}{Q_y(\xi_{s\land \tau_0^{-}},t-s)}\mathrm{d}s.
		\end{align}
		Hence we have
		\begin{align}\label{equation-A}
			\mathbf{E}_x\left(e^{A_{0,t}}1_{\{\tau_0^->t,\xi_t>y\}}\right)
			=&\mathbf{P}_x\left(\xi_{t\land \tau_0^-}>y\right)\\
			-&\mathbf{E}_x\left(1_{\{\tau_0^->t,\xi_t>y\}}\int_{0}^{t}e^{A_{s,t}}\frac{\Phi(Q_y(\xi_s,t-s))}{Q_y(\xi_s,t-s)}\mathrm{d}s\right).
		\end{align}
		Now using the Markov property and the fact that
		\begin{align}
			A_{s,t}=\int_{0}^{t-s}\frac{\Phi(Q_y(\xi_{(r+s)\land \tau_0^{-}},t-r-s))}{Q_y(\xi_{(r+s)\land \tau_0^{-}},t-r-s)}\mathrm{d}r,
		\end{align}
we see that  \eqref{equation-A}
implies that $\mathbf{E}_x\left(e^{A_{0,t}}1_{\{\tau_0^->t,\xi_t>y\}}\right)$
solves \eqref{int-Q}.
Thus, we have
		\begin{align}
			Q_y(x,t)
&=\mathbf{E}_x\left(e^{A_{0,t}}1_{\{\tau_0^->t,\xi_t>y\}}\right)\\
			&=e^{-\alpha t}\mathbf{E}_x\left(1_{\{\tau_0^->t,\xi_t>y\}}e^{-\int_{0}^{t}\varphi(Q_y(\xi_r,t-s))\mathrm{d}s}\right).
		\end{align}
		This gives the desired result.
	\end{proof}

	The next lemma will be used
	to prove the lower bounds in Theorems \ref{thm-survival-probability} and  \ref{thm-tail probability-Mt}.

	\begin{lemma}\label{lemma-lower-bound-liminf}
		Assume that \eqref{LLogL-moment-condition} holds and
		$\xi$ is a L\'{e}vy process satisfying {\bf (H1)}.
		Let $x>0$.
		\begin{enumerate}
			\item If $\mathbf{E}_0\left(\xi_1\right)=0$, then for any $y\ge 0$, we have
			\begin{align}
				\liminf_{t\to\infty}\sqrt{t}e^{\alpha t}Q_{\sqrt{t}y}(x,t)\ge
				2C_{sub}R(x)\phi_{\sigma^2}(y).
			\end{align}
			\item If $\mathbf{E}_0\left(\xi_1\right)>0$, we have
			\begin{align}
				\liminf_{t\to\infty}e^{\alpha t}u(x,t)
				\ge q_x C_{sub}.
			\end{align}
			Moreover, for any $y\in\mathbb{R}$, we have
			\begin{align}
				\liminf_{t\to\infty}e^{\alpha t}Q_{\sqrt{t}y+\mathbf{E}_0\left(\xi_1\right) t}(x,t)
				\ge
				q_x C_{sub}
				\int_{\frac{y}{\sigma}}^{\infty}\phi(z)\mathrm{d}z.
			\end{align}
			
			\item
			 			If $\mathbf{E}_0\left(\xi_1\right)<0$ and {\bf (H2)} and {\bf (H3)} hold,
			then for any $y\ge 0$, we have
			\begin{align}
				\liminf_{t\to\infty}t^{3/2}e^{\left(\alpha-\Psi(\lambda_*)\right)t} Q_y(x,t)
				\ge \frac{2C_{sub}R^*(x)
				e^{\lambda_*x}}{\sqrt{2\pi \Psi''(\lambda_*)^3}}\int_{y}^{\infty}e^{-\lambda_* z}
				\widehat{R}^*(z) \mathrm{d}z.
			\end{align}
		\end{enumerate}
	\end{lemma}
	\begin{proof}
		For any $y\ge 0$, by the definition of $Q$, we have
		\begin{align}
			Q_y(x,t)
			\le \P_x(\zeta>t)
			\le \P_x(\widetilde{\zeta}>t)=g(t).
		\end{align}
		It follows from Lemma \ref{lemma-expression-Q} that
		\begin{align}\label{lower-bound-Q}
			&Q_{y}(x,t)
			=e^{-\alpha t}\mathbf{E}_x\left(1_{\{\tau_0^->t,\xi_t> y\}}e^{-\int_{0}^{t} \varphi(Q_{y}(\xi_s,t-s))\mathrm{d}s}\right)\\
			&\ge e^{-\alpha t}e^{-\int_{0}^{t} \varphi(g(t-s))\mathrm{d}s}\mathbf{P}_x\left(\tau_0^->t,\xi_t> y\right)
			\ge C_{sub}e^{-\alpha t}\mathbf{P}_x\left(\tau_0^->t,\xi_t> y\right),
		\end{align}
where the last inequality follows from \eqref{Constant-C-sub}.
Applying Lemma \ref{lemma-tau0>t-rho=0}, we immediately get the assertion of (1).
Using the fact that $u(x, t)=Q_0(x, t)$, $\Psi'(0+)=\mathbb{E}_0(\xi_1)$ and applying \eqref{lim-tua_0>t} and \eqref{lower-bound-Q}, we get
\begin{align}
	\liminf_{t\to\infty}e^{\alpha t} u(x,t) \ge \liminf_{t\to\infty}C_{sub} \mathbf{P}_x\left(\tau_0^->t\right)
	=q_xC_{sub}.
\end{align}
This gives
the first result of (2). Applying Lemma \ref{lemma-tau0>t-rho>0} with $y$  replaced by $\sqrt{t}y+\mathbf{E}_0\left(\xi_1\right)t$,
we get the second result of (2).
Applying Theorem \ref{lemma-tau0>t-rho<0} with $f(x)=1_{(y,\infty)}(x)$, we get the assertion of (3).
	\end{proof}

	In the following
	three lemmas, we prove the upper bounds in Theorems \ref{thm-survival-probability} and \ref{thm-tail probability-Mt}.

	\begin{lemma}\label{lemma-upper-bound-limsup-rho=0}
		Assume that \eqref{LLogL-moment-condition} holds and
		$\xi$ is a L\'{e}vy process satisfying {\bf (H1)}.
		If $\mathbf{E}_0\left(\xi_1\right)=0$, then for any $x>0$ and $y\ge 0$, we
		have
		\begin{align}
			\limsup_{t\to\infty}\sqrt{t}e^{\alpha t}Q_{\sqrt{t}y}(x,t)\le
			2C_{sub}R(x)\phi_{\sigma^2}(y).
		\end{align}
	\end{lemma}
	\begin{proof}
		Recall that $\varphi$ and $Q_y(\cdot, t)$ are increasing functions.
		Fix an $N>0$. For $t>N$, by Lemma \ref{lemma-expression-Q}, we have
		\begin{align}
			Q_{\sqrt{t}y}(x,t)
			&\le e^{-\alpha t}\mathbf{E}_x\left(1_{\{\tau_0^->t,\xi_t>\sqrt{t}y\}}e^{-\int_{t-N}^{t} \varphi(Q_{\sqrt{t}y}(\xi_s,t-s))\mathrm{d}s}\right)\\
			&\le e^{-\alpha t}\mathbf{E}_x\left(1_{\{\tau_0^->t,\xi_t>\sqrt{t}y\}}e^{-\int_{0}^{N} \varphi(Q_{\sqrt{t}y}(\inf_{r\in[t-N,t]}\xi_r,s))\mathrm{d}s}\right).
		\end{align}
		Take a $\gamma \in(0,\frac{1}{2})$ and define
		\begin{align}
			J_1(t):=\mathbf{E}_x\left(1_{\{\tau_0^->t,\xi_t>\sqrt{t}y,\inf_{r\in[t-N,t]}\xi_r\ge \sqrt{t}y+t^{\gamma}\}}e^{-\int_{0}^{N} \varphi(Q_{\sqrt{t}y}(\inf_{r\in[t-N,t]}\xi_r,s))\mathrm{d}s}\right),
		\end{align}
		\begin{align}
			J_2(t):=\mathbf{E}_x\left(1_{\{\tau_0^->t,\xi_t>\sqrt{t}y,\inf_{r\in[t-N,t]}\xi_r< \sqrt{t}y+t^{\gamma}\}}e^{-\int_{0}^{N} \varphi(Q_{\sqrt{t}y}(\inf_{r\in[t-N,t]}\xi_r,s))\mathrm{d}s}\right).
		\end{align}
		Then $Q_{\sqrt{t}y}(x,t) \le e^{-\alpha t}(J_1(t)+J_2(t))$.
		Since $Q_{\sqrt{t}y}(x,t)$ is increasing in $x$, we have
		\begin{align}\label{upper-bound-J1}
			J_1(t)
			\le e^{-\int_{0}^{N} \varphi(Q_{\sqrt{t}y}(\sqrt{t}y+t^{\gamma},s))\mathrm{d}s}\mathbf{P}_x\left(\tau_0^->t,\xi_t>\sqrt{t}y\right).
		\end{align}
		By \eqref{lower-bound-Q} and  \eqref{eq-eg(t)},
		we have
		\begin{align}
			Q_{\sqrt{t}y}(x,t)
			\ge g(t)\mathbf{P}_x\left(\tau_0^->t,\xi_t>\sqrt{t}y\right).
		\end{align}
		Thus,
		\begin{align}
			e^{-\int_{0}^{N} \varphi(Q_{\sqrt{t}y}(\sqrt{t}y+t^{\gamma},s))\mathrm{d}s}
			\le \exp\Big\{-\int_{0}^{N}\varphi\left(g(s)\mathbf{P}_{\sqrt{t}y+t^{\gamma}}\left(\tau_0^->s,\xi_s>\sqrt{t}y\right)\right)
			\mathrm{d} s
			\Big\}.
		\end{align}
		Plugging this into \eqref{upper-bound-J1} and applying the dominated convergence theorem, we get
		\begin{align}\label{limsup-upper-bound-J1}
			\limsup_{N\to\infty}\limsup_{t\to\infty} \frac{J_1(t)}{\mathbf{P}_x\left(\tau_0^->t,\xi_t>\sqrt{t}y\right)}
			\le \limsup_{N\to\infty}e^{-\int_{0}^{N} \varphi(g(s))\mathrm{d}s}
			=C_{sub}.
		\end{align}
		Therefore, by Lemma \ref{lemma-tau0>t-rho=0}, we have
		\begin{align}\label{upper-bound-limsup-J1}
			\limsup_{N\to\infty}\limsup_{t\to\infty}\sqrt{t}J_1(t)
			\le
			2C_{sub}R(x)\phi_{\sigma^2}(y).
		\end{align}
Now we show that
\begin{equation}\label{toshow-J2}
\lim_{t\to\infty}\sqrt{t}J_2(t)=0.
\end{equation}
		For any $\epsilon>0$
		and $t>N$,
		it holds that
		\begin{align*}
			J_2(t)&\le \mathbf{P}_x\left(\tau_0^->t,\xi_t>\sqrt{t}y,\inf_{r\in[t-N,t]}\xi_r< \sqrt{t}y+t^{\gamma}\right)\\
			&\le \mathbf{P}_x\left(\tau_0^->t,\sqrt{t}y<\xi_t\le \sqrt{t}(y+\epsilon)\right)\\
			&\quad+\mathbf{P}_x\left(\tau_0^->t,\xi_t>\sqrt{t}(y+\epsilon),\inf_{r\in[t-N,t]}\xi_r< \sqrt{t}y+t^{\gamma}\right).
		\end{align*}
		By Lemma \ref{lemma-tau0>t-rho=0}, we have
		\begin{align}
			\lim_{t\to\infty}\sqrt{t}\mathbf{P}_x\left(\tau_0^->t,\sqrt{t}y<\xi_t\le \sqrt{t}(y+\epsilon)\right)
			=\frac{2R(x)}{\sqrt{2\pi \sigma^2}}
			\int_{\frac{y}{\sigma}}^{\frac{y+\epsilon}{\sigma}}
			\rho(z)
			\mathrm{d} z\xrightarrow{\epsilon\to 0}0.
		\end{align}
		For any $t>0$ and $\kappa  \in(0,\frac{\delta}{2(2+\delta)})$, define
		\begin{align}\label{def-event-A}
			A_t:=\left\{\sup_{s\in[0,1]}|\xi_{ts}-\xi_0-W_{ts}|>t^{\frac{1}{2}-\kappa}\right\},
		\end{align}
		where $W$ is the Brownian  motion
		in Lemma \ref{lemma-relation-with-BM}.
		Then by the Markov property of $\xi$, for $k<t-N$,
		\begin{align*}
			&\mathbf{P}_x\left(\tau_0^->t,\xi_t>\sqrt{t}(y+\epsilon),\inf_{r\in[t-N,t]}\xi_r< \sqrt{t}y+t^{\gamma}\right)\\
			=&\mathbf{E}_x\left(1_{\{\tau_0^->k\}}\mathbf{P}_{\xi_k}\left(\tau_0^->t-k,\xi_{t-k}>\sqrt{t}(y+\epsilon),\inf_{r\in[t-k-N,t-k]}\xi_r< \sqrt{t}y+t^{\gamma}\right)\right)\\
			\le & H_1(t)+H_2(t),
		\end{align*}
		where
		\begin{align}
			H_1(t):=\mathbf{E}_x\left(1_{\{\tau_0^->k\}}\mathbf{P}_{\xi_k}\left(\tau_0^->t-k,\xi_{t-k}>\sqrt{t}(y+\epsilon),A_{t-k}\right)\right),
		\end{align}
		\begin{align}
			H_2(t):=\mathbf{E}_x\left(\mathbf{P}_{\xi_k}\left(\xi_{t-k}>\sqrt{t}(y+\epsilon),\inf_{r\in[t-k-N,t-k]}\xi_r< \sqrt{t}y+t^{\gamma}, A_{t-k}^c\right)\right).
		\end{align}
To prove \eqref{toshow-J2}, we only need to prove
\begin{equation}\label{H1-H2-limit0}
\limsup_{t\to\infty}\sqrt{t}H_1(t)=0,\quad\mbox{ and }\quad \limsup_{t\to\infty}\sqrt{t}H_2(t)=0.
\end{equation}
		Using \eqref{upper-bound-tua} and Lemma \ref{lemma-relation-with-BM},
		we get that for any $k<t$,
		\begin{align}
			H_1(t)
			&\le \frac{C C_3(\kappa)}
			{(t-k)^{(\frac{1}{2}-\kappa)(\delta+2)
			-1}
			}\frac{1+x}{\sqrt{k}},
		\end{align}
		where $C>0$ is a constant.
		Taking $k=\frac{t}{2}$,  we get that
		\begin{align}\label{limit-H1}
			\limsup_{t\to\infty}\sqrt{t}H_1(t)
			= 0.
		\end{align}
		For any $z>0$, we have
		\begin{align*}
			&\mathbf{P}_z\left(\xi_{t-k}>\sqrt{t}(y+\epsilon),\inf_{r\in[t-k-N,t-k]}\xi_r< \sqrt{t}y+t^{\gamma},A_{t-k}^c\right)\\
			\le & \mathbf{Q}_z\left(W_{t-k}>\sqrt{t}(y+\epsilon)-t^{\frac{1}{2}-\kappa},\inf_{r\in[t-k-N,t-k]}W_r < \sqrt{t}y+t^{\gamma}+t^{\frac{1}{2}-\kappa}\right),
		\end{align*}
		where $(W_t,\mathbf{Q}_z)$ is a mean 0 Brownian motion with diffusion coefficient  $\sigma^2$, starting from $z$.
		Therefore, for any $z>0$,
		using the reflection principle for Brownian motion, we get
		\begin{align}
			&\lim_{t\to\infty}\sqrt{t}\mathbf{P}_z\left(\xi_{t-k}>\sqrt{t}(y+\epsilon),\inf_{r\in[t-k-N,t-k]}\xi_r< \sqrt{t}y+t^{\gamma},A_{t-k}^c\right)\\
			\le & \lim_{t\to\infty}\sqrt{t}\mathbf{Q}_0\left(\inf_{r\in[0,N]}W_r<-\epsilon\sqrt{t}+t^{\gamma}+2t^{\frac{1}{2}-\kappa}\right)\\
			=&\lim_{t\to\infty}\sqrt{t}\mathbf{Q}_0\left(\max_{r\in[0,N]}W_r>\epsilon \sqrt{t}-t^{\gamma}-2t^{\frac{1}{2}-\kappa}\right)
			=0,
		\end{align}
		which implies that
		\begin{align}\label{limit-H2}
			\lim_{t\to\infty}\sqrt{t} H_2(t)=0.
		\end{align}
Then \eqref{H1-H2-limit0} follows from \eqref{limit-H1} and \eqref{limit-H2}, and we complete the proof.
	\end{proof}
	
	\begin{lemma}\label{lemma-upper-bound-limsup-rho>0}
		Assume that \eqref{LLogL-moment-condition} holds and that
		$\xi$ is a L\'{e}vy process satisfying {\bf (H1)}.
		If $\mathbf{E}_0\left(\xi_1\right)>0$, then for any $x>0$, we have
		\begin{align}
			\limsup_{t\to\infty}e^{\alpha t}u(x,t)
			\le q_x C_{sub}.
		\end{align}
		Moreover, for any $y\in\mathbb{R}$, we have
		\begin{align}
			\limsup_{t\to\infty}e^{\alpha t}Q_{\sqrt{t}y+\mathbf{E}_0\left(\xi_1\right)t}(x,t)
		    \le q_x C_{sub}
			\int_{\frac{y}{\sigma}}^{\infty}\phi(z) \mathrm{d}z.
		\end{align}
	\end{lemma}
	\begin{proof}
		Take $\gamma \in(0,\frac{1}{2})$ and fix an $N>0$.
		For $t>N$,
		using Lemma \ref{lemma-expression-Q},
		we have
		\begin{align}
			Q_{\sqrt{t}y+\mathbf{E}_0\left(\xi_1\right) t}(x,t)
			&\le e^{-\alpha t}\mathbf{E}_x\left(1_{\{\tau_0^->t,\xi_t>\sqrt{t}y+\mathbf{E}_0\left(\xi_1\right) t\}}e^{-\int_{t-N}^{t} \varphi(Q_{\sqrt{t}y+\mathbf{E}_0\left(\xi_1\right) t}(\xi_s,t-s))\mathrm{d}s}\right)\\
			&\le e^{-\alpha t}\mathbf{E}_x\left(1_{\{\tau_0^->t,\xi_t>\sqrt{t}y+\mathbf{E}_0\left(\xi_1\right) t\}}e^{-\int_{0}^{N} \varphi(Q_{\sqrt{t}y+\mathbf{E}_0\left(\xi_1\right) t}(\inf_{r\in[t-N,t]}\xi_r,s))\mathrm{d}s}\right)\\
			&=:e^{-\alpha t}
					(K_1(t)+K_2(t)),
		\end{align}
		where
		\begin{align*}
					K_1(t):=
			\mathbf{E}_x\left(1_{\{\tau_0^->t,\xi_t>\sqrt{t}y+\mathbf{E}_0\left(\xi_1\right) t,\inf_{r\in[t-N,t]}\xi_r\ge \sqrt{t}y+\mathbf{E}_0\left(\xi_1\right) t+t^{\gamma}\}}e^{-\int_{0}^{N} \varphi(Q_{\sqrt{t}y+\mathbf{E}_0\left(\xi_1\right) t}(\inf_{r\in[t-N,t]}\xi_r,s))\mathrm{d}s}\right),
		\end{align*}
		\begin{align*}
				K_2(t):=
			\mathbf{E}_x\left(1_{\{\tau_0^->t,\xi_t>\sqrt{t}y+\mathbf{E}_0\left(\xi_1\right) t,\inf_{r\in[t-N,t]}\xi_r< \sqrt{t}y+\mathbf{E}_0\left(\xi_1\right) t+t^{\gamma}\}}e^{-\int_{0}^{N} \varphi(Q_{\sqrt{t}y+\mathbf{E}_0\left(\xi_1\right) t}(\inf_{r\in[t-N,t]}\xi_r,s))\mathrm{d}s}\right).
		\end{align*}
				Repeating the argument leading to \eqref{limsup-upper-bound-J1}, we obtain that
		\begin{align}\label{upper-bound-limsup-J1-rho>0}
			\limsup_{N\to\infty}\limsup_{t\to\infty} \frac{
							K_1(t)
				}{\mathbf{P}_x\left(\tau_0^->t,\xi_t>\sqrt{t}y+\mathbf{E}_0\left(\xi_1\right) t\right)}
			\le C_{sub}.
		\end{align}
		Therefore, by Lemma \ref{lemma-tau0>t-rho>0}, we have
		\begin{align}
			\limsup_{N\to\infty}\limsup_{t\to\infty}
						K_1(t)
			\le q_x C_{sub} \int_{\frac{y}{\sigma}}^{\infty}\phi(z) \mathrm{d}z.
		\end{align}
		Next, we show that
			$\lim_{t\to\infty}K_2(t)=0$.
		For $\epsilon>0$, it holds that
		\begin{align*}
    		K_2(t)
		   \le &
\mathbf{P}_x\left(\tau_0^->t,\xi_t>\sqrt{t}y+\mathbf{E}_0\left(\xi_1\right) t,\inf_{r\in[t-N,t]}\xi_r< \sqrt{t}y+\mathbf{E}_0\left(\xi_1\right) t+t^{\gamma}\right)\\
\le&
\mathbf{P}_x\left(\tau_0^->t,\sqrt{t}y+\mathbf{E}_0\left(\xi_1\right)t<\xi_t\le \sqrt{t}(y+\epsilon)+\mathbf{E}_0\left(\xi_1\right) t\right)\\
			&+\mathbf{P}_x\left(\xi_t>\sqrt{t}(y+\epsilon)+\mathbf{E}_0\left(\xi_1\right) t,\inf_{r\in[t-N,t]}\xi_r< \sqrt{t}y+\mathbf{E}_0\left(\xi_1\right) t+t^{\gamma}\right).
		\end{align*}
		By Lemma \ref{lemma-tau0>t-rho>0}, we have
		\begin{align}\label{first-part-J2}
			\lim_{t\to\infty}\mathbf{P}_x\left(\tau_0^->t,\sqrt{t}y+\mathbf{E}_0\left(\xi_1\right) t<\xi_t\le \sqrt{t}(y+\epsilon)+\mathbf{E}_0\left(\xi_1\right)t\right)
			\xrightarrow{\epsilon\to 0} 0.
		\end{align}
		For $\mathbf{E}_0(\xi_1)>0$, since $\left((\xi_t-\mathbf{E}_0(\xi_1)t)_{t\ge 0}, (\mathbf{P}_x)_{x\in \mathbb{R}}\right)$ is a L\'{e}vy process satisfying $\mathbf{E}_0(\xi_1-\mathbf{E}_0(\xi_1))=0$, it follows from Lemma \ref{lemma-relation-with-BM} that there exists a Brownian motion $W$ with diffusion coefficient $\sigma^2=\mathbf{E}_0(\xi_1^2)$ starting from the origin such that for all $t\ge 1$,
		\begin{align}\label{upper-bound-Dt}
			\mathbf{P}_x
			\left(\sup_{s\in[0,1]}|(\xi_{ts}-\mathbf{E}_0(\xi_1)ts)-x-W_{ts}|>t^{\frac{1}{2}-\kappa}\right)
			\le \frac{C_3(\kappa)
			}{t^{(\frac{1}{2}-\kappa)(\delta+2)-1}},
		\end{align}
		where $\kappa$ and $C_3(\kappa)$ are defined in Lemma \ref{lemma-relation-with-BM}.
		Let
		\begin{align}
			D_t:=\left\{\sup_{s\in[0,1]}|(\xi_{ts}-\mathbf{E}_0(\xi_1)ts)-x-W_{ts}|>t^{\frac{1}{2}-\kappa}\right\},
		\end{align}
		then by \eqref{upper-bound-Dt}, we have
		$\lim_{t\to\infty}\mathbf{P}_x\left(D_t\right)=0$.
		Moreover, we have
		\begin{align}
			&\mathbf{P}_x\left(\xi_t>\sqrt{t}(y+\epsilon)+\mathbf{E}_0\left(\xi_1\right) t,\inf_{r\in[t-N,t]}\xi_r< \sqrt{t}y+\mathbf{E}_0\left(\xi_1\right) t+t^{\gamma}\right)\\
			&\le \mathbf{P}_x\left(D_t\right)
			+\mathbf{P}_x\left(\xi_t>\sqrt{t}(y+\epsilon)+\mathbf{E}_0\left(\xi_1\right) t,\inf_{r\in[t-N,t]}\xi_r< \sqrt{t}y+\mathbf{E}_0\left(\xi_1\right) t+t^{\gamma}, D_t^c
			\right).
		\end{align}
		Furthermore, using the reflection principle for Brownian motion,
		we get that, as $t\to\infty$,
		\begin{align*}
			&\mathbf{P}_x\left(\xi_t>\sqrt{t}(y+\epsilon)+\mathbf{E}_0\left(\xi_1\right) t,\inf_{r\in[t-N,t]}\xi_r< \sqrt{t}y+\mathbf{E}_0\left(\xi_1\right) t+t^{\gamma},
			D_t^c
			\right)\\
			\le &\mathbf{Q}_x \left(W_t >\sqrt{t}(y+\epsilon)
			-t^{\frac{1}{2}-\kappa}, \inf_{r\in[t-N,t]}W_r<\sqrt{t}y
			+t^{\gamma}+t^{\frac{1}{2}-\kappa}\right)\\
			\le & \mathbf{Q}_0 \left( \inf_{r\in[0,N]}W_r<-\varepsilon \sqrt{t} +t^{\gamma}+2 t^{\frac{1}{2}-\kappa}\right)
			=\mathbf{Q}_0 \left( \max_{r\in[0,N]}W_r>\varepsilon \sqrt{t} -t^{\gamma}-2 t^{\frac{1}{2}-\kappa}\right)
			\to 0,
		\end{align*}
		where $(W_t,\mathbf{Q}_x)$ is a mean 0 Brownian motion with diffusion coefficient  $\sigma^2$, starting from $x$.
		This combined with \eqref{upper-bound-limsup-J1-rho>0}
		and \eqref{first-part-J2} gives the desired result.
	\end{proof}

	\begin{lemma}\label{lemma-upper-bound-limsup-rho<0}
		Fix an $N>0$. Assume that \eqref{LLogL-moment-condition} holds and
		$\xi$ is a L\'{e}vy process satisfying
		{\bf (H1)}, {\bf (H2)} and {\bf (H3)}.
		If $\mathbf{E}_0\left(\xi_1\right)<0$, then we have
		\begin{align}
			&\lim_{t\to\infty}
			t^{3/2}e^{-\Psi(\lambda_*) t}
			\mathbf{E}_x\left(1_{\{\tau_0^->t,\xi_t>y\}}e^{-\int_{t-N}^{t} \varphi(Q_y(\xi_s,t-s))\mathrm{d}s}\right)\\
			=&e^{(\alpha-\Psi(\lambda_*))N}\frac{2R^*(x)
			e^{\lambda_*x}}{\sqrt{2\pi \Psi''(\lambda_*)^3}}\int_{\R_+}\P_z(M_N>y)e^{-\lambda_* z}
			\widehat{R}^*(z)
			\mathrm{d}z.
		\end{align}
	\end{lemma}
	\begin{proof}
		By the Markov property,
		\begin{align}
			&\mathbf{E}_x\left(1_{\{\tau_0^->t,\xi_t>y\}}e^{-\int_{t-N}^{t} \varphi(Q_y(\xi_s,t-s))\mathrm{d}s}\right)\\
			&=\mathbf{E}_x\left(1_{\{\tau_0^->t-N\}}\mathbf{E}_{\xi_{t-N}}\left(1_{\{\tau_0^->N,\xi_N>y\}}e^{-\int_{0}^{N} \varphi(Q_y(\xi_s,N-s))\mathrm{d}s}\right)\right)\\
			&=:\mathbf{E}_x\left(1_{\{\tau_0^->t-N\}}f_N^y(\xi_{t-N})\right),
		\end{align}
		where for any $z\ge 0$, $f_N^y$ is defined by
		\begin{align}\label{def-f_N^y}
			f_N^y(z):=\mathbf{E}_z\left(1_{\{\tau_0^->N,\xi_N>y\}}e^{-\int_{0}^{N} \varphi(Q_y(\xi_s,N-s))\mathrm{d}s}\right).
		\end{align}
By Lemma \ref{lemma-expression-Q},
		\begin{align}\label{expression-f_N^z}
			f_N^y(z)=e^{\alpha N}Q_y(z,N),
		\end{align}
which implies that $f_N^y(z)$ is bounded and increasing with respect to $z$.
Then $f_N^y$  is a.e. continuous.
By \cite[Corollary 3.2]{Hin}, $f_N^y(z)e^{-\lambda_* z}(1+|z|)$ is directly Riemann integrable with respect to $z$.
Applying Theorem \ref{lemma-tau0>t-rho<0} with $f$ replaced by $f_N^y$, we get
		\begin{align}
			&\lim_{t\to\infty}t^{3/2}e^{-\Psi(\lambda_*)t}\mathbf{E}_x\left(1_{\{\tau_0^->t,\xi_t>y\}}e^{-\int_{t-N}^{t} \varphi(Q_y(\xi_s,t-s))\mathrm{d}s}\right)\\
			=&e^{-\Psi(\lambda_*)N}\frac{2R^*(x)
			e^{\lambda_*x}}{\sqrt{2\pi\Psi''(\lambda_*)^3}}\int_{\mathbb{R}_+}f_N^y(z)e^{-\lambda_* z}
			\widehat{R}^*(z) \mathrm{d}z,
		\end{align}
		which gives the desired result together with \eqref{expression-f_N^z}.
	\end{proof}

	{\bf Proofs of Theorems \ref{thm-survival-probability} and \ref{thm-tail probability-Mt}:}
	Combining Lemmas \ref{lemma-lower-bound-liminf}, \ref{lemma-upper-bound-limsup-rho=0} and \ref{lemma-upper-bound-limsup-rho>0}, we get
	parts (1) and (2) of both Theorem \ref{thm-survival-probability} and Theorem \ref{thm-tail probability-Mt} immediately.
	Next, we prove part (3) of both theorems. For $\mathbf{E}_0[\xi_1]<0$, fix $N>0$ and $y\ge 0$.
	By Lemma \ref{lemma-expression-Q}, we have for $t\ge N$,
	\begin{align}\label{upper-bound-Q_y}
		Q_y(x,t)\le e^{-\alpha t} \mathbf{E}_x\left(1_{\{\tau_0^->t,\xi_t>y\}}e^{-\int_{t-N}^{t} \varphi(Q_y(\xi_s,t-s))\mathrm{d}s}\right).
	\end{align}
	Combining this with Lemma \ref{lemma-upper-bound-limsup-rho<0}, we get that
	\begin{align}\label{limsup-upper-bd-Q}
		&\limsup_{t\to\infty}t^{3/2}e^{(\alpha-\Psi(\lambda_*))t}Q_y(x,t)\\
		\le &\frac{2R^*(x)
		e^{\lambda_*x}}{\sqrt{2\pi \Psi''(\lambda_*)^3}}\liminf_{N\to\infty}e^{(\alpha-\Psi(\lambda_*))N}\int_{\R_+}\P_z(M_N>y)e^{-\lambda_* z}
		\widehat{R}^*(z)
		\mathrm{d}z.
	\end{align}
	Moreover, using the fact that
	$Q_y(x,t)\le g(t)=\mathbb{P}_0(\widetilde{\zeta}>t)$
	and Lemma \ref{lemma-expression-Q}, we get
	\begin{align}
		Q_y(x,t)\ge e^{-\alpha t} \mathbf{E}_x\left(1_{\{\tau_0^->t,\xi_t>y\}}e^{-\int_{t-N}^{t} \varphi(Q_y(\xi_s,t-s))\mathrm{d}s}\right)e^{-\int_{0}^{t-N} \varphi(g(t-s))\mathrm{d}s}.
	\end{align}
	Using Lemma \ref{lemma-upper-bound-limsup-rho<0} again, we have
	\begin{align}\label{liminf-lower-bd-Q}
		&\liminf_{t\to\infty}t^{3/2}e^{(\alpha-\Psi(\lambda_*))t}Q_y(x,t)\\
		\ge &\frac{2R^*(x)
		e^{\lambda_*x}}{\sqrt{2\pi \Psi''(\lambda_*)^3}}\limsup_{N\to\infty}e^{(\alpha-\Psi(\lambda_*))N}\int_{\R_+}\P_z(M_N>y)e^{-\lambda_* z}
		\widehat{R}^*(z)
		\mathrm{d}z.
	\end{align}
	Combining \eqref{limsup-upper-bd-Q} and \eqref{liminf-lower-bd-Q}, we obtain that
	\begin{align}
		&\lim_{t\to\infty}t^{3/2}e^{(\alpha-\Psi(\lambda_*))t}Q_y(x,t)\\
		=&\frac{2R^*(x)
		e^{\lambda_*x}}{\sqrt{2\pi \Psi''(\lambda_*)^3}}\lim_{N\to\infty}e^{(\alpha-\Psi(\lambda_*))N}\int_{\R_+}\P_z(M_N>y)e^{-\lambda_* z}
		\widehat{R}^*(z)
		\mathrm{d}z:=\frac{2R^*(x)
		e^{\lambda_*x}}{\sqrt{2\pi \Psi''(\lambda_*)^3}}
		C_y,
	\end{align}
	where $C_y:=\lim_{N\to\infty}e^{(\alpha-\Psi(\lambda_*))N}\int_{\R_+}\P_z(M_N>y)e^{-\lambda_* z}
	\widehat{R}^*(z)\mathrm{d}z$. Next, we show that $C_y \in (0,\infty)$. First, applying Lemma \ref{lemma-lower-bound-liminf} (3), we get
	\begin{align}
		C_y\ge C_{sub} \int_{y}^{\infty} e^{-\lambda_* z}\widehat{R}^*(z) \mathrm{d}z>0.
	\end{align}
	Using \eqref{upper-bound-Q_y} and taking $f(x)=1_{(y,\infty)}(x)$ in Theorem \ref{lemma-tau0>t-rho<0}, we get
	\begin{align}
		\limsup_{t\to\infty}t^{3/2}e^{(\alpha-\Psi(\lambda_*))t}Q_y(x,t)
		&\le \lim_{t\to\infty}t^{3/2}e^{-\Psi(\lambda_*)t} \mathbf{P}_x\left(\tau_0^->t,\xi_t>y\right)\\
		&=\frac{2R^*(x)
		e^{\lambda_*x}}{\sqrt{2\pi\Psi''(\lambda_*)^3}}\int_{y}^{\infty}e^{-\lambda_* z}
		\widehat{R}^*(z) \mathrm{d}z.
	\end{align}
	Therefore, $C_y\le \int_{y}^{\infty}e^{-\lambda_* z}
		\widehat{R}^*(z) \mathrm{d}z<\infty$.
	This completes the proof.
	\qed
	
	{\bf Proof of Corollary \ref{thm-yaglom-limit-theorem}:}
	We only prove (3).
	Combining Theorem \ref{thm-survival-probability} and \ref{thm-tail probability-Mt}, for any $0<a<b$, we get that
	\begin{align}
		&\lim_{t\to\infty}\P_x\left(M_t\in (a,b]|\zeta>t\right)
		=\lim_{t\to\infty}\frac{Q_a(x,t)-Q_b(x,t)}{u(x,t)}\\
		=&\frac{\lim_{N\to\infty}\int_{0}^{\infty} \mathbb{P}_z(M_N \in (a,b])e^{-\lambda_* z}\widehat{R}^*(z)\mathrm{d}z}{\lim_{N\to\infty}\int_{0}^{\infty}\mathbb{P}_z(M_N\in (0,\infty) )e^{-\lambda_* z}\widehat{R}^*(z)\mathrm{d}z}.
	\end{align}
	Therefore, there exists a random variable
$(X,\mathbb{P})$
 such that $\mathbb{P}_x(M_t\in \cdot|\zeta>t)$ vaguely converge to
 $\mathbb{P}(X\in \cdot)$.
	Moreover, by \eqref{lower-bound-Q}, we have
	\begin{align*}
		\P_x\left(M_t>y|\zeta>t\right)
		=\frac{Q_y(x,t)}{u(x,t)}
		\le \frac{e^{-\alpha t}\mathbf{P}_x\left(\tau_0^->t,\xi_t> y\right)}{C_{sub}e^{-\alpha t}\mathbf{P}_x\left(\tau_0^->t\right)}
		=\frac{1}{C_{sub}}\mathbf{P}_x\left(\xi_t> y|\tau_0^->t\right).
	\end{align*}
	Thus by
	Theorem \ref{lemma-tau0>t-rho<0},
	the tightness of $M_t$
	under $\mathbb{P}_x(\cdot|\zeta>t)$
	follows from the tightness of $\xi_t$ under $\mathbf{P}_x(\cdot|\tau_0^->t)$.
	This gives the desired result.
	\qed

	\section{Proof of Theorem \ref{thm-tail-M}}\label{section-proof-thm-tail-M}
	Recall that $\alpha=\beta(1-m)$.
	For any $0<x<y$, define
	\begin{align}
		v(x,y):=\P_x(M>y).
	\end{align}
	The following result is valid for any
branching killed L\'{e}vy processes.
	
	\begin{lemma}\label{lemma-expression-v}
		For any $0<x<y$, it holds that
		\begin{align}
			v(x,y)=\mathbf{E}_x\left(1_{\{\tau_y^{+}<\tau_0^{-}\}} e^{-\alpha \tau_y^{+}-\int_{0}^{\tau_y^{+}}\varphi(v(\xi_s,y))\mathrm{d}s
			}\right),
		\end{align}
		where $\varphi$ is defined
by \eqref{def_Phi}.
Consequently, for $0<x<z<y$, by the strong Markov property, we have
		\begin{align}
			v(x,y)=\mathbf{E}_x\left(1_{\{\tau_z^{+}<\tau_0^{-}\}}
			v(\xi_{\tau_z^{+}},y)
			e^{-\alpha \tau_z^{+}-\int_{0}^{\tau_z^{+}}\varphi(v(\xi_s,y))\mathrm{d}s
			}\right).
		\end{align}
	\end{lemma}
	\begin{proof}
		For $0< x<y$, comparing the first branching time with $\tau_y^+$, we have
		\begin{align*}
			v(x,y)=&\int_{0}^{\infty} \beta e^{-\beta s} \mathbf{P}_x(\tau_y^+< \tau_0^-,\tau_y^+\le s)\mathrm{d}s\\
			&+\int_{0}^{\infty}\beta e^{-\beta s}\mathbf{E}_x\left(\left(1-\sum_{k=0}^{\infty} p_k\left(1-v(\xi_s,y)\right)^k \right)  1_{\{\tau_y^+\land \tau_0^->s\}}\right)
			\mathrm{d}s\\
			=&\mathbf{E}_x	\left(e^{-\beta \tau_y^+}1_{\{\tau_y^+<\tau_0^-\}}\right)
			+\int_{0}^{\infty}\beta e^{-\beta s}\mathbf{E}_x\left( \left(1-\sum_{k=0}^{\infty}p_k \left(1-v(\xi_s,y)\right)^k\right)1_{\{\tau_y^+\land \tau_0^->s\}}\right)\mathrm{d}s.
		\end{align*}
		By \cite[Lemma 4.1]{Dynkin01}, the above equation is equivalent to
		\begin{align*}
			&v(x,y)+\beta \int_{0}^{\infty}\mathbf{E}_x\left(v(\xi_s,y)1_{\{\tau_y^+\land \tau_0^->s\}}\right)\mathrm{d}s\\	&=\mathbf{P}_x\left(\tau_y^+<\tau_0^-\right)
			+\beta\int_{0}^{\infty}\mathbf{E}_x\left( \left(1-\sum_{k=0}^{\infty}p_k\left(1-v(\xi_s,y)\right)^k\right)1_{\{\tau_y^+\land \tau_0^->s\}}\right)\mathrm{d}s,
		\end{align*}
		which is also equivalent to
		\begin{align}
			v(x,y)=\mathbf{P}_x\left(\tau_y^+<\tau_0^-\right)
			- \mathbf{E}_x\left( \int_{0}^{\tau_y^+\land\tau_0^-} \Phi(v(\xi_s,y))\mathrm{d}s\right),
		\end{align}
		where $\Phi$ is defined in \eqref{def_Phi}.
		 By repeating the argument leading to \eqref{expectation-expression-Q}, we get the desired result.
	\end{proof}
	
	In the remainder of this section, we always assume that $((\xi_t)_{t\ge 0}, ( \mathbf{P}_x)_{x\in \R})$ is a spectrally negative L\'{e}vy process.
	\begin{lemma}\label{lemma-upper-bound-v}
	Assume that $\xi$
	is a spectrally negative L\'{e}vy process.
		For any $0<x<y$, we have
		\begin{align}\label{upper-bound-v}
			v(x,y)\le \frac{e^{x \psi(\alpha)}W_{\psi(\alpha)}^{(0)}(x)}{e^{y \psi(\alpha)}W_{\psi(\alpha)}^{(0)}(y)}
			\le e^{(x-y)\psi(\alpha)}.
		\end{align}
	\end{lemma}
	\begin{proof}
		Since the function $\varphi$ is non-negative, combining Lemma \ref{lemma-expression-v} and Theorem \ref{thm-exit-problems} (2), we get
		\begin{align}
			v(x,y)
			\le \mathbf{E}_x\left(e^{-\alpha \tau_y^{+}
			}1_{\{\tau_y^{+}<\tau_0^{-}\}} \right)
			=\frac{W^{(\alpha)}(x)}{W^{(\alpha)}(y)},\quad x< y.
		\end{align}
		This combined with Lemma \ref{lemma-scale-property} yields that
		\begin{align}
			v(x,y)\le \frac{e^{x \psi(\alpha)}W_{\psi(\alpha)}^{(0)}(x)}{e^{y \psi(\alpha)}W_{\psi(\alpha)}^{(0)}(y)}
			\le e^{(x-y)\psi(\alpha)}.
		\end{align}
		This gives the desired result.
	\end{proof}
	
	{\bf Proof of Theorem \ref{thm-tail-M}:}
	By Lemmas \ref{lemma-expression-v} and  \ref{lemma-change-measure},
	we have
	\begin{align*}
		v(x,y)
		&=\mathbf{E}_x\left(1_{\{\tau_y^{+}<\tau_0^{-}\}} e^{-\alpha \tau_y^{+}-\int_{0}^{\tau_y^{+}}\varphi(v(\xi_s,y))\mathrm{d}s}\right)\\
		&=e^{(x-y)\psi(\alpha)}\mathbf{E}_x^{\psi(\alpha)}\left(1_{\{\tau_y^{+}<\tau_0^{-}\}} e^{-\int_{0}^{\tau_y^{+}}\varphi(v(\xi_s,y))\mathrm{d}s}\right).
	\end{align*}
	Fix a $\gamma\in (0,1)$, by the Markov property of $(\xi_t, \mathbf{P}_x^{\psi(\alpha)})$, we have
	\begin{align}\label{decomposition-v}
		v(x,y)
		&=e^{(x-y)\psi(\alpha)}\mathbf{E}_x^{\psi(\alpha)}\left(1_{\{\tau_{y-y^{\gamma}}^{+}<\tau_0^{-}\}} e^{-\int_{0}^{\tau_{y-y^{\gamma}}^{+}}\varphi(v(\xi_s,y))\mathrm{d}s}\right)\\
		&\qquad\times \mathbf{E}_{y-y^{\gamma}}^{\psi(\alpha)}\left(1_{\{\tau_y^{+}<\tau_0^{-}\}} e^{-\int_{0}^{\tau_y^{+}}\varphi(v(\xi_s,y))\mathrm{d}s}\right)
		=:e^{(x-y)\psi(\alpha)}A_1(x,y)A_2(y),
	\end{align}
	where
	\begin{align*}
		A_1(x,y):=\mathbf{E}_x^{\psi(\alpha)}\left(1_{\{\tau_{y-y^{\gamma}}^{+}<\tau_0^{-}\}} e^{-\int_{0}^{\tau_{y-y^{\gamma}}^{+}}\varphi(v(\xi_s,y))\mathrm{d}s}\right),
	\end{align*}
	\begin{align*}
		A_2(y):=\mathbf{E}_{y-y^{\gamma}}^{\psi(\alpha)}\left(1_{\{\tau_y^{+}<\tau_0^{-}\}} e^{-\int_{0}^{\tau_y^{+}}\varphi(v(\xi_s,y))\mathrm{d}s}\right).
	\end{align*}
	We first  consider the asymptotic behavior of $A_1(x,y)$ as $y\to\infty$.
	We claim that
	\begin{align}\label{asymptotic-A1}
		\lim_{y\to \infty}\left(\mathbf{P}_x^{\psi(\alpha)}
		(\tau_{y-y^{\gamma}}^{+}<\tau_0^{-})
		-A_1(x,y) \right)=0.
	\end{align}
	Indeed, using the inequality $1-e^{-|x|}\le |x|$, we get
	\begin{align}\label{two-side-bound}
		0&\le \mathbf{P}_x^{\psi(\alpha)}
		(\tau_{y-y^{\gamma}}^{+}<\tau_0^{-})
		-A_1(x,y)\\
		&=\mathbf{E}_x^{\psi(\alpha)}\left(1_{\{\tau_{y-y^{\gamma}}^{+}<\tau_0^{-}\}} \left(1-e^{-\int_{0}^{\tau_{y-y^{\gamma}}^{+}}\varphi(v(\xi_s,y))\mathrm{d}s}\right)\right)\nonumber\\
		&\le \mathbf{E}_x^{\psi(\alpha)}\left(1_{\{\tau_{y-y^{\gamma}}^{+}<\tau_0^{-}\}}\int_{0}^{\tau_{y-y^{\gamma}}^{+}}\varphi(v(\xi_s,y))\mathrm{d}s\right)
		\le \mathbf{E}_x^{\psi(\alpha)}\left(\int_{0}^{\tau_{y-y^{\gamma}}^{+}}\varphi(v(\xi_s,y))\mathrm{d}s\right).\nonumber
	\end{align}
	Set $y_*(x):=\inf\{t\ge y-y^{\gamma},t-x\in \mathbb{N}\}$. By \eqref{upper-bound-v}, we have
	\begin{align*}
		&\mathbf{E}_x^{\psi(\alpha)}\left(\int_{0}^{\tau_{y-y^{\gamma}}^{+}}\varphi(v(\xi_s,y))\mathrm{d}s\right)
		\le \mathbf{E}_x^{\psi(\alpha)}\left(\int_{0}^{\tau_{y_*(x)}^+}\varphi\left(e^{(\xi_s-y)\psi(\alpha)}\right)\mathrm{d}s\right)\\
		& = \sum_{k=0}^{y_*(x)-x-1} \mathbf{E}_x^{\psi(\alpha)} \left(\int_{\tau_{x+k}^+}^{\tau_{x+k+1}^+} \varphi\left(e^{(\xi_s-y)\psi(\alpha)}\right)\mathrm{d}s\right)\\
		&\leq  \sum_{k=0}^{y_*(x)-x-1} \mathbf{E}_x^{\psi(\alpha)} \left(\tau_{x+k+1}^+- \tau_{x+k}^+ \right)\varphi\left(e^{\psi(\alpha)(x+k+1-y)}\right)\\
		&= \mathbf{E}_0^{\psi(\alpha)} \left(\tau_{1}^+\right) \sum_{k=1}^{y_*(x)} \varphi\left(
		e^{-\psi(\alpha)(y-x-1-y_*(x)+k)}\right).
	\end{align*}
	By the definition of $y_*(x)$,
	we have that for $y$ large enough,
	\[
	y-x-1-y_*(x) \geq y-x-1-(y-y^\gamma+1)= y^\gamma-x-2.
	\]
	Therefore, when $y$ is sufficient large so that $y^\gamma -x-2 \geq  y^{\gamma/2}$,
	by \eqref{property-varphi}, we have
	\begin{align}\label{Upper-Difference}
		& \mathbf{E}_x^{\psi(\alpha)}\left(\int_{0}^{\tau_{y-y^{\gamma}}^{+}}\varphi(v(\xi_s,y))\mathrm{d}s\right)\\
		& \leq \mathbf{E}_0^{\psi(\alpha)} \left(\tau_{1}\right) \sum_{k=1}^{\infty } \varphi\left(	e^{-\psi(\alpha)(y^{\gamma/2}+k)}\right)  \leq	\mathbf{E}_0^{\psi(\alpha)} \left(\tau_{1}\right)
		\int_0^\infty \varphi\left(	e^{-\psi(\alpha)(y^{\gamma/2}+z)}\right) \mathrm{d}z\\
		& =	\mathbf{E}_0^{\psi(\alpha)} \left(\tau_{1}\right)\int_{y^{\gamma/2}}^\infty \varphi\left( e^{-\psi(\alpha)z}\right) \mathrm{d}z		 \stackrel{y\to\infty}{\longrightarrow} 0.
	\end{align}
	This combined with \eqref{two-side-bound} yields \eqref{asymptotic-A1}. Using Lemma \ref{lemma-change-measure} and Theorem \ref{thm-exit-problems}(2), we get
	\begin{align}
		&\lim_{y\to\infty}A_1(x,y)
		=\lim_{y\to\infty} \mathbf{P}_x^{\psi(\alpha)}
		(\tau_{y-y^{\gamma}}^{+}<\tau_0^{-})
		\\
		&=\lim_{y\to\infty}
		e^{(y-y^{\gamma}-x)\psi(\alpha)} \mathbf{E}_x\left(e^{-\alpha \tau_{y-y^{\gamma}}^+}1_{\{\tau_{y-y^{\gamma}}^+<\tau_0^-\}}\right)
		=\lim_{y\to\infty}e^{(y-y^{\gamma}-x)\psi(\alpha)} \frac{W^{(\alpha)}(x)}{W^{(\alpha)}(y-y^{\gamma})}.
	\end{align}
Using \eqref{limit-behvior-W}, we get that
	\begin{align}
		W^{(\alpha)}(y-y^{\gamma})
		\sim \frac{e^{\psi(\alpha)(y-y^{\gamma})}}{\Psi'(\psi(\alpha))},
		\quad \text{as}~ y\to\infty.
	\end{align}
	Therefore,
\begin{align}\label{asymptotic-behavior-A1}
		&\lim_{y\to\infty}A_1(x,y)=e^{-\psi(\alpha)x}\Psi'(\psi(\alpha))W^{(\alpha)}(x).
\end{align}
	Next, we consider the asymptotic behavior of $A_2(y)$ as $y\to\infty$. Recall that
	\begin{align}
		A_2(y)=&\mathbf{E}_{y-y^{\gamma}}^{\psi(\alpha)}\left(1_{\{\tau_y^{+}<\tau_0^{-}\}} e^{-\int_{0}^{\tau_y^{+}}\varphi(v(\xi_s,y))\mathrm{d}s}\right)\\
		=&\mathbf{E}_{y-y^{\gamma}}^{\psi(\alpha)}\left( e^{-\int_{0}^{\tau_y^{+}}\varphi(v(\xi_s,y))\mathrm{d}s}\right)
		-\mathbf{E}_{y-y^{\gamma}}^{\psi(\alpha)}\left(1_{\{\tau_y^{+}\ge\tau_0^{-}\}} e^{-\int_{0}^{\tau_y^{+}}\varphi(v(\xi_s,y))\mathrm{d}s}\right).
	\end{align}
	We claim that
	\begin{align}\label{claim1}
		\lim_{y\to\infty}\mathbf{E}_{y-y^{\gamma}}^{\psi(\alpha)}\left( e^{-\int_{0}^{\tau_y^{+}}\varphi(v(\xi_s,y))\mathrm{d}s}\right)=C_*(\alpha)
		\in (0,1],
	\end{align}
	and
	\begin{align}\label{claim2}
		\lim_{y\to\infty}\mathbf{E}_{y-y^{\gamma}}^{\psi(\alpha)}\left(1_{\{\tau_y^{+}\ge\tau_0^{-}\}} e^{-\int_{0}^{\tau_y^{+}}\varphi(v(\xi_s,y))\mathrm{d}s}\right)=0.
	\end{align}
Then we get
\begin{align}\label{expression-A2}
		\lim_{y\to\infty}A_2(y)=C_*(\alpha).
\end{align}
Combining \eqref{decomposition-v}, \eqref{asymptotic-behavior-A1} and \eqref{expression-A2} gives that
\begin{align}
		\lim_{y\to\infty}
		e^{y\psi(\alpha)}
		v(x,y)=C_*(\alpha)\Psi'(\psi(\alpha)) W^{(\alpha)}(x),
	\end{align}
which gives the desired result.
Now we are left to prove \eqref{claim1} and \eqref{claim2}.
By Lemma \ref{lemma-change-measure} and Theorem \ref{thm-exit-problems}, we have
	\begin{align}
		&\mathbf{E}_{y-y^{\gamma}}^{\psi(\alpha)}\left(1_{\{\tau_y^{+}\ge\tau_0^{-}\}} e^{-\int_{0}^{\tau_y^{+}}\varphi(v(\xi_s,y))\mathrm{d}s}\right)
		\le \mathbf{P}_{y-y^{\gamma}}^{\psi(\alpha)}\left(\tau_y^{+}\ge\tau_0^{-}\right)
		=1-\mathbf{P}_{y-y^{\gamma}}^{\psi(\alpha)}\left(\tau_y^{+}<\tau_0^{-}\right)\\
		&=1-e^{y^{\gamma}\psi(\alpha)} \mathbf{E}_{y-y^{\gamma}}\left(e^{-\alpha \tau_y^+}1_{\{\tau_y^+<\tau_0^-\}}\right)
		=1-e^{y^{\gamma}\psi(\alpha)} \frac{W^{(\alpha)}(y-y^{\gamma})}{W^{(\alpha)}(y)}
	\end{align}
	which tends to $0$ as $y \to\infty$ by
	\eqref{limit-behvior-W}. Thus  \eqref{claim2} is valid.
	To prove \eqref{claim1}, for any $y>0$, define
	\begin{align}
		G(y):=\mathbf{E}_{y-y^{\gamma}}^{\psi(\alpha)}\left( e^{-\int_{0}^{\tau_y^{+}}\varphi(v(\xi_s,y))\mathrm{d}s}\right).
	\end{align}
	For any $z>y$, by the
	translation invariance
	and the strong Markov property of
	$\xi$, we have
	\begin{align}
		G(z)
		&=\mathbf{E}_{z-z^{\gamma}}^{\psi(\alpha)}\left( e^{-\int_{0}^{\tau_z^{+}}\varphi(v(\xi_s,z))\mathrm{d}s}\right)
		=\mathbf{E}_{0}^{\psi(\alpha)}\left( e^{-\int_{0}^{\tau_{z^{\gamma}}^{+}}\varphi(v(\xi_s+z-z^{\gamma},z))\mathrm{d}s}\right)\\
		&=\mathbf{E}_{0}^{\psi(\alpha)}\left( e^{-\int_{0}^{\tau_{z^{\gamma}-y^{\gamma}}^{+}}\varphi(v(\xi_s+z-z^{\gamma},z))\mathrm{d}s}\right)
		\mathbf{E}_{z^{\gamma}-y^{\gamma}}^{\psi(\alpha)}\left( e^{-\int_{0}^{\tau_{z^{\gamma}}^{+}}\varphi(v(\xi_s+z-z^{\gamma},z))\mathrm{d}s}\right),
	\end{align}
	where the first term of the above display is dominated by 1 from above and the second term is equal to $\mathbf{E}_{0}^{\psi(\alpha)}\left( e^{-\int_{0}^{\tau_{y^{\gamma}}^{+}}\varphi(v(\xi_s+z-y^{\gamma},z))\mathrm{d}s}\right)$.
	It follows that
	\begin{align}\label{upper-bound-G}
		G(z)\le \mathbf{E}_{0}^{\psi(\alpha)}\left( e^{-\int_{0}^{\tau_{y^{\gamma}}^{+}}\varphi(v(\xi_s+z-y+y-y^{\gamma},z-y+y))\mathrm{d}s}\right).
	\end{align}
	Note that for any $w>0$, it holds that
	\begin{align}
		&v(x+w,y+w)=\P_{x+w}\left(\exists ~t>0, u\in N_t~s.t.~\min_{s\le t}X_u(s)>0, X_u(t)>y+w\right)\\
		&\ge \P_{x+w}\left(\exists~t>0, u\in N_t~s.t.~\min_{s\le t}X_u(s)>w, X_u(t)>y+w\right)=v(x,y).
	\end{align}
	This combined with \eqref{upper-bound-G} gives that for $z>y$,
	\begin{align}
		G(z)\le \mathbf{E}_{0}^{\psi(\alpha)}\left( e^{-\int_{0}^{\tau_{y^{\gamma}}^{+}}\varphi(v(\xi_s+y-y^{\gamma},y))\mathrm{d}s}\right)
		=G(y).
	\end{align}
	Thus, the limit $C_*(\alpha):=\lim_{y\to \infty}G(y)$ exists.
	It is obvious that $C_*(\alpha)\le 1$.
	Next, we only need to show $C_*(\alpha)>0$.
	We assume without loss of generality that $y$ is an integer. By the strong Markov property and Jensen's inequality,
	\begin{align*}
		G(y)
		&=\frac{\mathbf{E}_{0}^{\psi(\alpha)}\left( e^{-\int_{0}^{\tau_y^{+}}\varphi(v(\xi_s,y))\mathrm{d}s}\right)}{\mathbf{E}_{0}^{\psi(\alpha)}\left( e^{-\int_{0}^{\tau_{y-y^{\gamma}}^{+}}\varphi(v(\xi_s,y))\mathrm{d}s}\right)}
		\ge \mathbf{E}_{0}^{\psi(\alpha)}\left( e^{-\int_{0}^{\tau_y^{+}}\varphi(v(\xi_s,y))\mathrm{d}s}\right)\\
		&\ge \exp\left\{ -\sum_{n=1}^{y}\mathbf{E}_{0}^{\psi(\alpha)}\left(\int_{\tau_{n-1}^+}^{\tau_n^+}\varphi(v(\xi_s,y))\mathrm{d}s\right)\right\}.
	\end{align*}
	By  \eqref{upper-bound-v}, we get
	\begin{align}
		\int_{\tau_{n-1}^+}^{\tau_n^+}\varphi(v(\xi_s,y))\mathrm{d}s
		\le (\tau_n^+-\tau_{n-1}^+)\varphi(v(n,y))
		\le (\tau_n^+-\tau_{n-1}^+)\varphi\left(e^{(n-y)\psi(\alpha)}\right).
	\end{align}
	Note that under $\mathbf{P}_{0}^{\psi(\alpha)}$, $\{\tau_n^+-\tau_{n-1}^+\}$ are i.i.d. random variables with finite first moment.
	Therefore,
	\begin{align*}
		G(y)
		&\ge \exp\left\{ -\sum_{n=1}^{y}\varphi\left(e^{(n-y)\psi(\alpha)}\right)\mathbf{E}_{0}^{\psi(\alpha)}\left((\tau_n^+-\tau_{n-1}^+)\right)\right\}\\
		&=\exp\left\{ -\mathbf{E}_{0}^{\psi(\alpha)}(\tau_1^+)\sum_{n=0}^{y-1}\varphi\left(e^{-n\psi(\alpha)}\right)\right\}
		\ge \exp\left\{ -\mathbf{E}_{0}^{\psi(\alpha)}(\tau_1^+)\sum_{n=0}^{\infty}\varphi\left(e^{-n\psi(\alpha)}\right)\right\},
	\end{align*}
	which implies that
	\begin{align}
		C_*(\alpha)
		\ge \exp\left\{ -\mathbf{E}_{0}^{\psi(\alpha)}(\tau_1^+)\sum_{n=0}^{\infty}\varphi\left(e^{-n\psi(\alpha)}\right)\right\}.
	\end{align}
	According to \eqref{property-varphi}, we have
	\begin{align}
		\sum_{n=0}^{\infty}\varphi\left(e^{-n\psi(\alpha)}\right)
		\le \varphi(1)+\int_{0}^{\infty}\varphi\left(e^{-z\psi(\alpha)}\right)\mathrm{d} z<\infty,
	\end{align}
	which implies that $C_*(\alpha)>0$. This gives the desired result.
	\qed

	\vspace{.1in}
	\textbf{Acknowledgment}: We thank the referees for very helpful comments and suggestions.

\end{document}